\documentclass[11pt,fleqn]{article}
\usepackage{geometry}                
\usepackage{graphicx}
\usepackage{amsthm,amsmath}
\usepackage{amssymb,dsfont,mathrsfs}

\numberwithin{equation}{section}
\newtheorem{theorem}{Theorem}[section]
\newtheorem{lemma}{Lemma}[section]
\newtheorem{proposition}{Proposition}[section]
\newtheorem{corollary}{Corollary}[section]
\newtheorem{remark}{Remark}[section]
\newtheorem{assumption}{Assumption}[section]

\def\ba{\boldsymbol{a}}

\def\bc{\boldsymbol{c}}

\def\bg{\boldsymbol{g}}

\def\bi{\boldsymbol{i}}

\def\bu{\boldsymbol{u}}
\def\bv{\boldsymbol{v}}

\def\bx{\boldsymbol{x}}

\def\bJ{\boldsymbol{J}}

\def\bX{\boldsymbol{X}}
\def\bY{\boldsymbol{Y}}

\def\bvarphi{\boldsymbol{\varphi}}

\def\bnu{\boldsymbol{\nu}}

\def\bzero{\mathbf{0}}

\def\calD{\mathcal{D}}
\def\calE{\mathcal{E}}
\def\calG{\mathcal{G}}

\def\scrI{\mathscr{I}}

\def\spr{\mbox{\rm spr}}
\def\cp{\mbox{\rm cp}}
\def\diag{\mbox{\rm diag}}
\def\adj{\mbox{\rm adj}}

\def\Ker{\mbox{\rm Ker}}

\def\vec{\mbox{\rm vec}}
\title{Exact asymptotics of the stationary tail probabilities in an arbitrary direction in a two-dimensional discrete-time QBD process}
\author{Toshihisa Ozawa  \\ 
Faculty of Business Administration, Komazawa University \\
1-23-1 Komazawa, Setagaya-ku, Tokyo 154-8525, Japan \\
E-mail: toshi@komazawa-u.ac.jp
}
\date{}

\begin{document}

\maketitle

\begin{abstract}
We deal with a discrete-time two-dimensional quasi-birth-and-death process (2d-QBD process for short) on $\mathbb{Z}_+^2\times S_0$, where $S_0$ is a finite set, and give a complete expression for the asymptotic decay function of the stationary tail probabilities in an arbitrary direction. The 2d-QBD process is a kind of random walk in the quarter plane \textit{with a background process}.
In our previous paper (\textit{Queueing Systems, vol.\,102, pp.\,227--267, 2022}), we have obtained the asymptotic decay rate of the stationary tail probabilities in an arbitrary direction and clarified that if the asymptotic decay rate $\xi_{\bc}$, where $\bc$ is a direction vector in $\mathbb{N}^2$, is less than a certain value $\theta_{\bc}^{max}$, the sequence of the stationary tail probabilities in the direction $\bc$ geometrically decays without power terms, asymptotically. 
In this paper, we give the function according to which the sequence asymptotically decays, including the case where $\xi_{\bc}=\theta_{\bc}^{max}$. When  $\xi_{\bc}=\theta_{\bc}^{max}$, the function is given by an exponential function with power term $k^{-\frac{1}{2}}$ except for two boundary cases, where it is given by just an exponential function without power terms.  This result coincides with the existing result for a random walk in the quarter plane \textit{without background processes}, obtained by Malyshev (\textit{Siberian Math.\,J., vol.\,12, ,pp.\,109--118, 1973}).

\smallskip
\textit{Note}. 
In this version of the paper, the complete expression of the asymptotic decay function in the direction $\bc$ is given. In the previous version, it contained an unknown parameter $l$. We have clarified that $l=1$. 

\smallskip
\textit{Keywards}: quasi-birth-and-death process, Markov modulated reflecting random walk, Markov additive process, asymptotic decay rate, asymptotic decay function, stationary distribution, matrix analytic method

\end{abstract}

%
%
\section{Introduction} \label{sec:intro}

We deal with a discrete-time two-dimensional quasi-birth-and-death process (2d-QBD process for short) $\{\bY_n\}=\{(\boldsymbol{X}_n,J_n)\}$ on $\mathbb{Z}_+^2\times S_0$, where $S_0$ is a finite set. This model is a Markov modulated reflecting random walk (MMRRW for short) whose transitions are skip free, and the MMRRW is a kind of reflecting random walk (RRW for short) \textit{with a background process}, where the transition probabilities of the RRW vary depending on the state of the background process. 
One-dimensional QBD processes have been introduced by Macel Neuts and studied in the literature as one of the essential stochastic models in the queueing theory (see, for example, \cite{Bini05,Latouche99,Neuts94,Neuts89}). The 2d-QBD process is a two-dimensional version of one-dimensional QBD process, and it enable us to analyze, for example, two-node queueing networks and two-node polling models. 

Assume the 2d-QBD process $\{\bY_n\}$ is positive recurrent and denote by $\bnu=(\bnu_{\bx}, \bx\in\mathbb{Z}_+^2)$ the stationary distribution, where $\bnu_{\bx}=(\nu_{(\bx,j)}, j\in S_0)$ and $\nu_{(\bx,j)}$ is the stationary probability that the process is in the state $(\bx,j)$. Our interest is asymptotics of the stationary distribution $\bnu$, especially, tail asymptotics in an arbitrary direction. Let an integer vector $\bc=(c_1,c_2)$ be nonzero and nonnegative. Two typical objects of our study are the asymptotic decay rate $\xi_{\bc}$ and asymptotic decay function $h_{\bc}(k)$ defined as, for $j\in S_0$,  
\begin{align*}
&\xi_{\bc} = - \lim_{k\to\infty} \frac{1}{k} \log \nu_{(k \bc,j)}, \\
&\lim_{k\to\infty} \frac{\nu_{(k\bc,j)}}{h_{\bc}(k)} = g_{\bc,j},
\end{align*}
where $g_{\bc,j}$ is a positive constant. Under a certain condition, the asymptotic decay rate of the probability sequence $\{\nu_{(k\bc+\bx,j)}: k\ge 0\}$ does not depend on $\bx$ and $j$ if it exists, see Proposition 2.3 of Ozawa \cite{Ozawa22}. 
In the case where $\bc=(1,0)$ or $\bc=(0,1)$, the asymptotic decay rate $\xi_{\bc}$ has been obtained in Ozawa \cite{Ozawa13}, see Corollary 4.3 therein, and the asymptotic decay function $h_{\bc}(k)$ in Ozawa and Kobayashi \cite{Ozawa18}, see Theorem 2.1 therein. The results in the case where $\bc=(c,0)$ or $\bc=(0,c)$ for $c\ge 2$ can automatically be obtained from those in \cite{Ozawa13,Ozawa18}. 
In the case where $\bc=(c_1,c_2)\ge(1,1)$, the asymptotic decay rate $\xi_{\bc}$ has been obtained in \cite{Ozawa22}, see Theorem 3.2 therein. We state that result of \cite{Ozawa22} in Theorem \ref{th:decay_rate} of this paper. In \cite{Ozawa22}, it has also been clarified that the asymptotic decay function is given by $h_{\bc}(k) = e^{-\xi_{\bc} k}$ if $\xi_{\bc}$ is less than a certain value $\theta_{\bc}^{max}$.
For other related works on asymptotics of the stationary distributions in 2d-RRWs with and without background processes, see Section 1 of \cite{Ozawa22} and references therein. 

In this paper, we give a complete expression for the asymptotic decay function $h_{\bc}(k)$ when $\bc=(c_1,c_2)\ge(1,1)$, including the case where $\xi_{\bc}=\theta_{\bc}^{max}$. 
To this end, we clarify the analytic properties of the vector generating function of the stationary probabilities along the direction $\bc$, given by $\bvarphi^{\bc}(z)=\sum_{k=0}^\infty z^k \bnu_{k \bc}$. The point $z=e^{\xi_{\bc}}$ is a singular point of the vector function $\bvarphi^{\bc}(z)$, and if $\xi_{\bc}=\theta_{\bc}^{max}$, the point $z=e^{\theta_{\bc}^{max}}$ is a branch point of $\bvarphi^{\bc}(z)$ with order one. 
From this result, we obtain the expression of $h_{\bc}(k)$ when $\xi_{\bc}=\theta_{\bc}^{max}$, which is given by an exponential function with power term $k^{-\frac{1}{2}}$, i.e., $h_{\bc}(k) = k^{-\frac{1}{2}} e^{-\xi_{\bc} k}$, except for two boundary cases, see Proposition \ref{pr:power_term} of Section \ref{sec:mainresults}. In the boundary cases, it is given as $h_{\bc}(k) = e^{-\xi_{\bc} k}$. We state the whole expression of the asymptotic decay function in Theorem \ref{th:main_theorem} of Section \ref{sec:mainresults}.
This result coincides with that for a 2d-RRW \textit{without background processes}, obtained by Malyshev \cite{Malyshev73}. 
The asymptotic decay function of the probability sequence $\{\nu_{(k\bc+\bx,j)}: k\ge 0\}$ does not depend on $\bx$ and $j$, see Corollary \ref{co:decay_function_homogeneous} of Section \ref{sec:mainresults}. 
We also generalize a part of the existing our results. One crucial point in analyzing the asymptotic decay function is how to analytically extend the G-matrix function appeared in the vector generating function of the stationary probabilities. The G-matrix function is a solution to a matrix quadratic equation each entry of whose coefficient matrices is a Laurent polynomial.
In \cite{Ozawa18}, analytic extension of the G-matrix function has been done under the assumption that all the eigenvalues of the G-matrix function are distinct, see Assumption 4.1 and Lemma 4.5 of \cite{Ozawa18}. This assumption is not easy to verify in general. We, therefore, remove the assumption and give a general formula of the Jordan decomposition of the G-matrix function, see Section \ref{sec:G_analyticextension}. The G-matrix function can analytically be extended through the Jordan decomposition.

%
The rest of the paper is organized as follows.
In Section \ref{sec:mainresults}, we describe the 2d-QBD process in detail and state assumptions and main results. 
In Section \ref{sec:preliminary}, an analytic extension of the G-matrix function is given in a general setting. The definition of G-matrix \textit{in the reverse direction} and its property are also given in the same section. 
The proof of the main results is given in Section \ref{sec:gf_analytic_properties}, where we demonstrate that the vector function $\bvarphi^{\bc}(z)$ is element-wise analytic in the open disk with radius $e^{\xi_{\bc}}+\varepsilon$ for some $\varepsilon>0$, except for the point $z=e^{\xi_{\bc}}$, and clarify its singularity at the point $z=e^{\xi_{\bc}}$. The asymptotic decay function is obtained from those results.
The paper concludes with a remark in Section \ref{sec:conclusion}.

%
%
\section{Model description and main results} \label{sec:mainresults}

%
\subsection{Model description}

We consider the same model as that described in \cite{Ozawa22} and use the same notation. 

Denote by $\scrI_2$ the set of all the subsets of $\{1,2\}$, i.e., $\scrI_2=\{\emptyset,\{1\},\{2\},\{1,2\}\}$, and we use it as an index set. Divide $\mathbb{Z}_+^2$ into $2^2=4$ exclusive subsets defined as 
\[
\mathbb{B}^\alpha=\{\bx=(x_1,x_2)\in\mathbb{Z}_+^2: \mbox{$x_i>0$ for $i\in\alpha$, $x_i=0$ for $i\in\{1,2\}\setminus \alpha$} \},\ \alpha\in\scrI_2.  
\]
%
%
Let $\{\bY_n\}=\{(\bX_n,J_n)\}$ be a 2d-QBD process on $\mathbb{Z}_+^2\times S_0$, where $S_0=\{1,2,...,s_0\}$. 
Let $P$ be the transition probability matrix of $\{\bY_n\}$ and represent it in block form as $P=\left( P_{\bx,\bx'}; \bx,\bx'\in\mathbb{Z}_+^2 \right)$, where $P_{\bx,\bx'}=(p_{(\bx,j),(\bx',j')}; j,j'\in S_0)$ and $p_{(\bx,j),(\bx',j')}=\mathbb{P}(\bY_1=(\bx',j')\,|\,\bY_0=(\bx,j))$. For $\alpha\in\scrI_2$ and $i_1,i_2\in\{-1,0,1\}$, let $A^\alpha_{i_1,i_2}$ be a one-step transition probability block from a state in $\mathbb{B}^\alpha$, 
where we assume the blocks corresponding to impossible transitions are zero (see Fig.\ \ref{fig:fig11}). %
%
\begin{figure}[t]
\begin{center}
\includegraphics[width=55mm,trim=0 0 0 0]{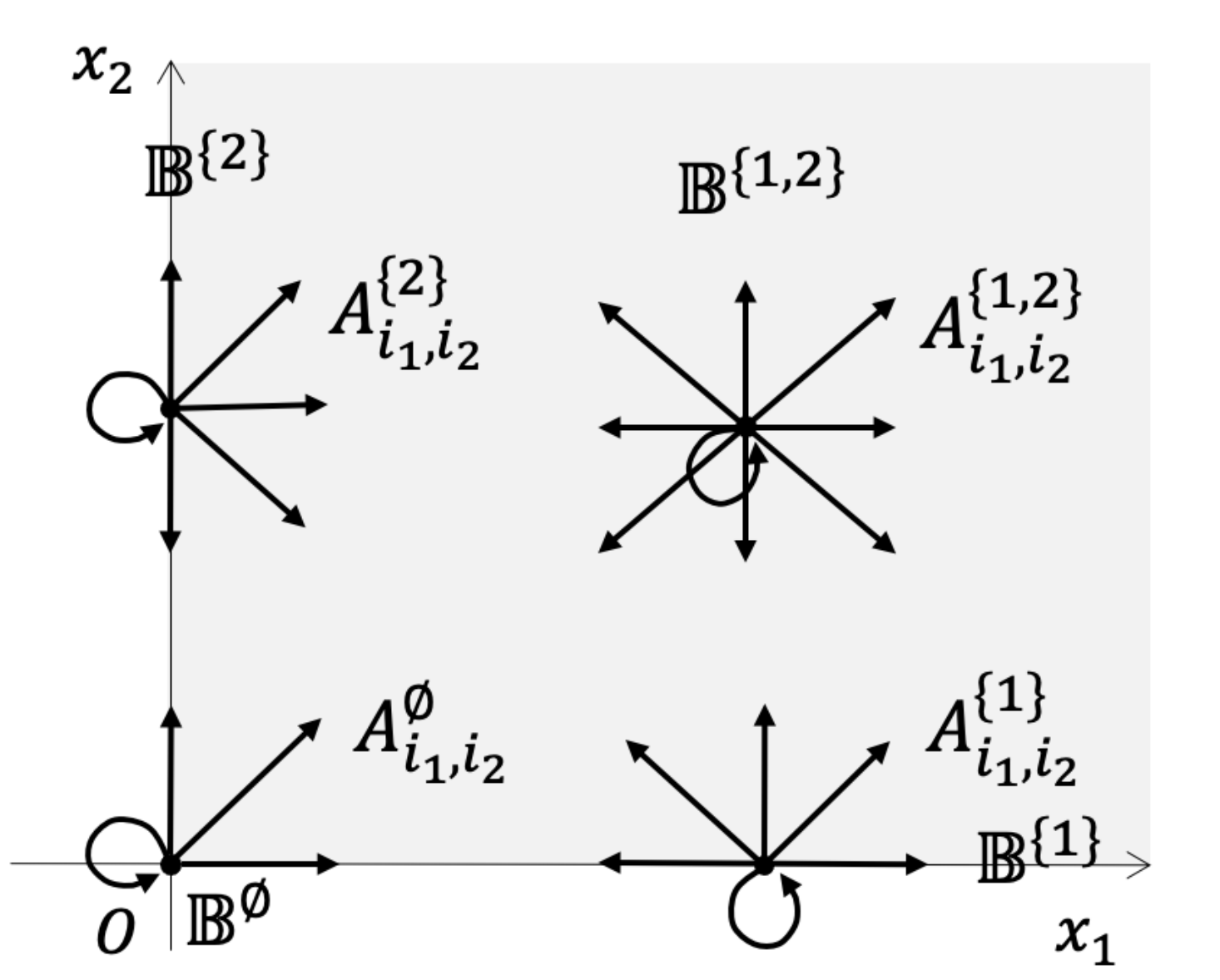} 
\caption{Transition probability blocks}
\label{fig:fig11}
\end{center}
\end{figure}
%
Since the level process is skip free, for every $\bx,\bx'\in\mathbb{Z}_+^2$, $P_{\bx,\bx'}$ is given by 
\begin{equation}
P_{\bx,\bx'} 
= \left\{ \begin{array}{ll} 
A^\alpha_{\bx'-\bx}, & \mbox{if $\bx\in\mathbb{B}^\alpha$ for some $\alpha\in\scrI_2$ and $\bx'-\bx\in\{-1,0,1\}^2$}, \cr
O, & \mbox{otherwise}.
\end{array} \right.
\end{equation}
We assume the following condition throughout the paper. 
\begin{assumption} \label{as:QBD_irreducible}
The 2d-QBD process $\{\bY_n\}$ is irreducible and aperiodic. 
\end{assumption}

%
Next, we define several Markov chains derived from the 2d-QBD process.
For a nonempty set $\alpha\in\scrI_2$, let $\{\bY^\alpha_n\}=\{(\bX^\alpha_n,J^\alpha_n)\}$ be a process derived from the 2d-QBD process $\{\bY_n\}$ by removing the boundaries that are orthogonal to the $x_i$-axis for each $i\in\alpha$. 
The process $\{\bY^{\{1\}}_n\}$ is a Markov chain on $\mathbb{Z}\times\mathbb{Z}_+\times S_0$ whose transition probability matrix $P^{\{1\}}=(P^{\{1\}}_{\bx,\bx'};\bx,\bx'\in\mathbb{Z}\times\mathbb{Z}_+)$ is given as
\begin{equation}
P^{\{1\}}_{\bx,\bx'} 
= \left\{ \begin{array}{ll} 
A^{\{1\}}_{\bx'-\bx}, & \mbox{if $\bx\in\mathbb{Z}\times\{0\}$ and $\bx'-\bx\in\{-1,0,1\}\times\{0,1\}$}, \cr
A^{\{1,2\}}_{\bx'-\bx}, & \mbox{if $\bx\in\mathbb{Z}\times\mathbb{N}$ and $\bx'-\bx\in\{-1,0,1\}^2$}, \cr
O, & \mbox{otherwise},
\end{array} \right.
\end{equation}
where $\mathbb{N}$ is the set of all positive integers.
The process $\{\bY^{\{2\}}_n\}$ on $\mathbb{Z}_+\times\mathbb{Z}\times S_0$ and its transition probability matrix $P^{\{2\}}=(P^{\{2\}}_{\bx,\bx'};\bx,\bx'\in\mathbb{Z}_+\times\mathbb{Z})$ are analogously defined. The process $\{\bY^{\{1,2\}}_n\}$ is a Markov chain on $\mathbb{Z}^2\times S_0$, whose transition probability matrix $P^{\{1,2\}}=(P^{\{1,2\}}_{\bx,\bx'};\bx,\bx'\in\mathbb{Z}^2)$ is given as
\begin{equation}
P^{\{1,2\}}_{\bx,\bx'} 
= \left\{ \begin{array}{ll} 
A^{\{1,2\}}_{\bx'-\bx}, & \mbox{if $\bx'-\bx\in\{-1,0,1\}^2$}, \cr
O, & \mbox{otherwise}. 
\end{array} \right.
\end{equation}
Regarding $X^{\{1\}}_{1,n}$ as the additive part, we see that the process $\{\bY^{\{1\}}_n\}=\{(X^{\{1\}}_{1,n},(X^{\{1\}}_{2,n},J^{\{1\}}_n))\}$ is a Markov additive process (MA-process for short) with the background state $(X^{\{1\}}_{2,n},J^{\{1\}}_n)$ (with respect to MA-processes, see, for example, Ney and Nummelin \cite{Ney87}). The process $\{\bY^{\{2\}}_n\}=\{(X^{\{2\}}_{2,n},(X^{\{2\}}_{1,n},J^{\{2\}}_n))\}$ is also an MA-process, where $X^{\{2\}}_{2,n}$ is the additive part and $(X^{\{2\}}_{1,n},J^{\{2\}}_n)$ the background state, and $\{\bY^{\{1,2\}}_n\}=\{(X^{\{1,2\}}_{1,n},X^{\{1,2\}}_{2,n}),J^{\{1,2\}}_n)\}$ an MA-process, where $(X^{\{1,2\}}_{1,n},X^{\{1,2\}}_{2,n})$ the additive part and $J^{\{1,2\}}_n$ the background state. We call them the induced MA-processes derived from the original 2d-QBD process. 
%
Let $\{ \bar A^{\{1\}}_i: i\in\{-1,0,1\} \}$ be the Markov additive kernel (MA-kernel for short) of the induced MA-process $\{\bY^{\{1\}}_n\}$, which is the set of transition probability blocks and defined as, for $i\in\{-1,0,1\}$,  
\begin{align*}
& \bar A^{\{1\}}_i = \left( \bar A^{\{1\}}_{i,(x_2,x_2')}; x_2,x_2'\in\mathbb{Z}_+ \right),\\
&\bar A^{\{1\}}_{i,(x_2,x_2')} = \left\{ \begin{array}{ll}
 A^{\{1\}}_{i,x_2'-x_2}, & \mbox{if $x_2=0$ and $x_2'-x_2\in\{0,1\}$}, \cr
 A^{\{1,2\}}_{i,x_2'-x_2}, & \mbox{if $x_2\ge 1$ and $x_2'-x_2\in\{-1,0,1\}$}, \cr
 O, & \mbox{otherwise}. 
 \end{array} \right.
\end{align*}
Let $\{ \bar A^{\{2\}}_i: i\in\{-1,0,1\} \}$ be the MA-kernel of $\{\bY^{\{2\}}_n\}$, defined in the same way. With respect to $\{\bY^{\{1,2\}}_n\}$, the MA-kernel is given by $\{ A^{\{1,2\}}_{i_1,i_2}: i_1,i_2\in\{-1,0,1\} \}$. 
We assume the following condition throughout the paper. 
\begin{assumption} \label{as:MAprocess_irreducible}
The induced MA-processes $\{\bY^{\{1\}}_n\}$, $\{\bY^{\{2\}}_n\}$ and $\{\bY^{\{1,2\}}_n\}$ are irreducible and aperiodic. 
\end{assumption}

%
%
According to \cite{Ozawa22}, we assume several other technical conditions for the induced MA-process $\{\bY_n^{\{1,2\}}\}$, concerning irreducibility and aperiodicity on subspaces. 
Let $\{\bY^+_{n}\}=\{(\bX^+_n,J^+_n)\}$ be a lossy Markov chain derived from the induced MA-process $\{\bY^{\{1,2\}}_n\}$ by restricting the state space of the additive part to $\mathbb{N}^2$. The process $\{\bY^+_n\}$ is a Markov chain on the state space $\mathbb{N}^2\times S_0$ whose transition probability matrix $P^+$ is given as $P^+=(P^{\{1,2\}}_{\bx,\bx'};\bx,\bx'\in\mathbb{N}^2)$, where  $P^+$ is strictly substochastic. 
We assume the following condition throughout the paper.
\begin{assumption} \label{as:Y12_onZpZp_irreducible}
$\{\bY^+_n\}$ is irreducible and aperiodic. 
\end{assumption}
%
%
For $k\in\mathbb{Z}$, let $\mathbb{Z}_{\le k}$ and $\mathbb{Z}_{\ge k}$ be the set of integers less than or equal to $k$ and that of integers greater than or equal to $k$, respectively. We assume the following condition throughout the paper. For what this assumption implies, see Remark 3.1 of \cite{Ozawa22}. 
\begin{assumption} \label{as:Y12_onZpmZmp_irreducible}
\begin{itemize}
\item[(i)] The lossy Markov chain derived from the induced MA-process $\{\bY^{\{1,2\}}_n\}$ by restricting the state space to $\mathbb{Z}_{\le 0}\times \mathbb{Z}_{\ge 0}\times S_0$ is irreducible and aperiodic. 
\item[(ii)] The lossy Markov chain derived from $\{\bY^{\{1,2\}}_n\}$ by restricting the state space to $\mathbb{Z}_{\ge 0}\times \mathbb{Z}_{\le 0}\times S_0$ is irreducible and aperiodic. 
\end{itemize}
\end{assumption}

%
The stability condition of the 2d-QBD process has already been obtained in \cite{Ozawa19}. 
Let $a^{\{1\}}$, $a^{\{2\}}$ and $\ba^{\{1,2\}}=(a^{\{1,2\}}_1,a^{\{1,2\}}_2)$ be the mean drifts of the additive part in the induced MA-processes $\{\bY^{\{1\}}_n\}$, $\{\bY^{\{2\}}_n\}$ and $\{\bY^{\{1,2\}}_n\}$, respectively.
By Corollary 3.1 of \cite{Ozawa19}, the stability condition of the 2d-QBD process $\{\bY_n\}$ is given as follows:
\begin{lemma} \label{le:stability_cond}
\begin{itemize}
\item[(i)] In the case where $a^{\{1,2\}}_1<0$ and $a^{\{1,2\}}_2<0$, the 2d-QBD process $\{\bY_n\}$ is positive recurrent if $a^{\{1\}}<0$ and $a^{\{2\}}<0$, and it is transient if either $a^{\{1\}}>0$ or $a^{\{2\}}>0$. 
\item[(ii)] In the case where $a^{\{1,2\}}_1\ge 0$ and $a^{\{1,2\}}_2<0$, $\{\bY_n\}$ is positive recurrent if $a^{\{1\}}<0$, and it is transient if $a^{\{1\}}>0$. 
\item[(iii)] In the case where $a^{\{1,2\}}_1<0$ and $a^{\{1,2\}}_2\ge 0$, $\{\bY_n\}$ is positive recurrent if $a^{\{2\}}<0$, and it is transient if $a^{\{2\}}>0$. 
\item[(iv)] If one of $a^{\{1,2\}}_1$ and $a^{\{1,2\}}_2$ is positive and the other is non-negative, then $\{\bY_n\}$ is transient.
\end{itemize}
\end{lemma}

For the explicit expression of the mean drifts, see Section 3.1 of \cite{Ozawa19} and its related parts. 
We assume the following condition throughout the paper.
\begin{assumption} \label{as:2dQBD_stable}
The condition in Lemma \ref{le:stability_cond} that ensures the 2d-QBD process $\{\bY_n\}$ is positive recurrent holds.
\end{assumption}

Denote by $\bnu$ the stationary distribution of $\{\bY_n\}$, where $\bnu=(\bnu_{\bx},\bx\in\mathbb{Z}_+^2)$, $\bnu_{\bx}=(\nu_{(\bx,j)}, j\in S_0)$ and $\nu_{(\bx,j)}$ is the stationary probability that the 2d-QBD process is in the state $(\bx,j)$.

%
\subsection{Main results}
%
Let $\bar A^{\{1\}}_*(z)$ and  $\bar A^{\{2\}}_*(z)$ be the matrix generating functions of the MA-kernels of $\{\bY^{\{1\}}_n\}$ and $\{\bY^{\{2\}}_n\}$, respectively, defined as
\[
\bar A^{\{1\}}_*(z) = \sum_{i\in\{-1,0,1\}} z^i \bar A^{\{1\}}_i,\quad 
\bar A^{\{2\}}_*(z) = \sum_{i\in\{-1,0,1\}} z^i \bar A^{\{2\}}_i. 
\]
The matrix generating function of the MA-kernel of $\{\bY^{\{1,2\}}_n\}$ is given by $A^{\{1,2\}}_{*,*}(z_1,z_2)$, defined as
\[
A^{\{1,2\}}_{*,*}(z_1,z_2) = \sum_{i_1,i_2\in\{-1,0,1\}} z_1^{i_1} z_2^{i_2} A^{\{1,2\}}_{i_1,i_2}. 
\]
%
Let $\Gamma^{\{1\}}$, $\Gamma^{\{2\}}$ and $\Gamma^{\{1,2\}}$ be domains in which  the convergence parameters of $\bar A^{\{1\}}_*(e^{\theta_1})$, $\bar A^{\{2\}}_*(e^{\theta_2})$ and $A^{\{1,2\}}_{*,*}(e^{\theta_1},e^{\theta_2})$ are greater than $1$, respectively, i.e., 
\begin{align*}
&\Gamma^{\{1\}} = \{(\theta_1,\theta_2)\in\mathbb{R}^2: \cp(\bar A^{\{1\}}_*(e^{\theta_1}))>1 \},\quad 
\Gamma^{\{2\}} = \{(\theta_1,\theta_2)\in\mathbb{R}^2: \cp(\bar A^{\{2\}}_*(e^{\theta_2}))>1 \}, \\
&\Gamma^{\{1,2\}} = \{(\theta_1,\theta_2)\in\mathbb{R}^2: \cp(A^{\{1,2\}}_{*,*}(e^{\theta_1},e^{\theta_2}))>1 \}, 
\end{align*}
where, for a nonnegative square matrix $A$ with a finite or countable dimension, $\cp(A)$ denote the convergence parameter of $A$, i.e., $\cp(A) = \sup\{r\in\mathbb{R}_+: \sum_{n=0}^\infty r^n A^n<\infty,\ \mbox{entry-wise} \}$. 
We have $\cp(A^{\{1,2\}}_{*,*}(e^{\theta_1},e^{\theta_2}))=\spr(A^{\{1,2\}}_{*,*}(e^{\theta_1},e^{\theta_2}))^{-1}$, where for a square complex matrix $A$, $\spr(A)$ is the spectral radius of $A$. 
By Lemma A.1 of  Ozawa \cite{Ozawa21}, $\cp(\bar A^{\{1\}}_*(e^\theta))^{-1}$ and $\cp(\bar A^{\{2\}}_*(e^\theta))^{-1}$ are log-convex in $\theta$, and the closures of $\Gamma^{\{1\}}$ and $\Gamma^{\{2\}}$ are convex sets; $\spr(A^{\{1,2\}}_*(e^{\theta_1},e^{\theta_2}))$ is also log-convex in $(\theta_1,\theta_2)$, and the closure of $\Gamma^{\{1,2\}}$ is a convex set. Furthermore, by Proposition B.1 of Ozawa \cite{Ozawa21}, $\Gamma^{\{1,2\}}$ is bounded under Assumption \ref{as:MAprocess_irreducible}. 
We depict an example of the domains $\Gamma^{\{1,2\}}$, $\Gamma^{\{1\}}$ and $\Gamma^{\{2\}}$ in Fig.\ \ref{fig:fig12}.
%
\begin{figure}[t]
\begin{center}
\includegraphics[width=130mm,trim=0 0 0 0]{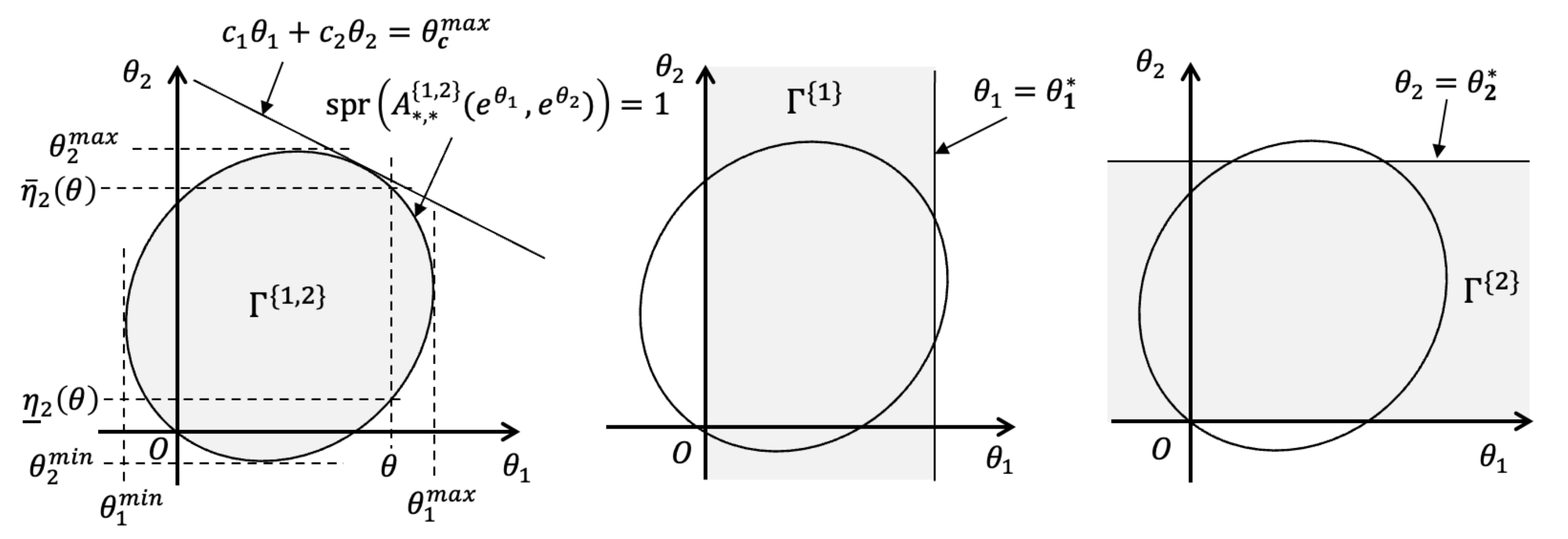} 
\caption{Domains $\Gamma^{\{1,2\}}$, $\Gamma^{\{1\}}$ and $\Gamma^{\{2\}}$}
\label{fig:fig12}
\end{center}
\end{figure}

%
We define several extreme values and functions with respect to the domains. For $i\in\{1,2\}$, define $\theta_i^{min}$ and $\theta_i^{max}$ as 
\[
\theta_i^{min} = \inf\{ \theta_i\in\mathbb{R}: (\theta_1,\theta_2)\in\Gamma^{\{1,2\}} \},\quad 
\theta_i^{max} = \sup\{ \theta_i\in\mathbb{R}: (\theta_1,\theta_2)\in\Gamma^{\{1,2\}} \}, 
\]
and for a direction vector $\bc=(c_1,c_2)\in\mathbb{N}^2$, $\theta_{\bc}^{max}$ as 
\[
\theta_{\bc}^{max} = 
\sup\{c_1\theta_1+c_2\theta_2: (\theta_1,\theta_2)\in\Gamma^{\{1,2\}} \}. 
\]
By Lemma 2.3 of \cite{Ozawa18}, under Assumption \ref{as:2dQBD_stable}, the domain $\Gamma^{\{1,2\}}$ includes positive points, i.e., $\{(\theta_1,\theta_2)\in\Gamma^{\{1,2\}}: \theta_1>0,\,\theta_2>0 \}\ne\emptyset$, and this implies $\theta_{\bc}^{max}>0$ for every direction vector $\bc\in\mathbb{N}^2$. 
For $\theta_1\in[\theta_1^{min},\theta_1^{max}]$, there exist two real solutions to equation $\spr(A^{\{1,2\}}_{*,*}(e^{\theta_1},e^{\theta_2}))=1$, counting multiplicity. Denote them by $\theta_2=\underline{\eta}_2(\theta_1)$ and $\theta_2=\bar{\eta}_2(\theta_1)$, respectively,  where $\underline{\eta}_2(\theta_1)\le\bar{\eta}_2(\theta_1)$. For $\theta_2\in[\theta_2^{min},\theta_2^{max}]$, also denote by $\theta_1=\underline{\eta}_1(\theta_2)$ and $\theta_1=\bar{\eta}_1(\theta_2)$ the two real solutions to the equation $\spr(A^{\{1,2\}}_{*,*}(e^{\theta_1},e^{\theta_2}))=1$, where $\underline{\eta}_1(\theta_2)\le\bar{\eta}_1(\theta_2)$. 
For $i\in\{1,2\}$, define $\theta_i^*$ as
\[
\theta_i^* = \sup\{\theta_i\in\mathbb{R}: (\theta_1,\theta_2)\in\Gamma^{\{i\}} \}. 
\]
For another characterization of $\theta_i^*$, see Proposition 3.7 of Ozawa \cite{Ozawa13}, where $\theta_i^*$ is denoted by $z_0$. 
By Lemma 2.5  of \cite{Ozawa18} and its related parts, under Assumption \ref{as:2dQBD_stable}, we have $\theta_i^*>0$ for $i\in\{1,2\}$. 

%
In terms of these points and functions, we geometrically classify the model into four types according to Section 4.1 of \cite{Ozawa22}. 
Define two points $\mathrm{Q}_1$ and $\mathrm{Q}_2$ as $\mathrm{Q}_1=(\theta_1^*,\bar{\eta}_2(\theta_1^*))$ and $\mathrm{Q}_2=(\bar{\eta}_1(\theta_2^*),\theta_2^*)$, respectively. Using these points, we define the following classification (see Fig.\ \ref{fig:classification}). 
\begin{itemize}
\item[] Type 1: $\theta_1^* \ge \bar{\eta}_1(\theta_2^*)$ and $\bar{\eta}_2(\theta_1^*) \le \theta_2^*$, \\
Type 2: $\theta_1^* < \bar{\eta}_1(\theta_2^*)$ and $\bar{\eta}_2(\theta_1^*) > \theta_2^*$, \\
Type 3: $\theta_1^* \ge \bar{\eta}_1(\theta_2^*)$ and $\bar{\eta}_2(\theta_1^*) > \theta_2^*$, \\
Type 4: $\theta_1^* < \bar{\eta}_1(\theta_2^*)$ and $\bar{\eta}_2(\theta_1^*) \le \theta_2^*$.
\end{itemize}
%
\begin{figure}[tb]
\begin{center}
\includegraphics[width=160mm,trim=0 0 0 0]{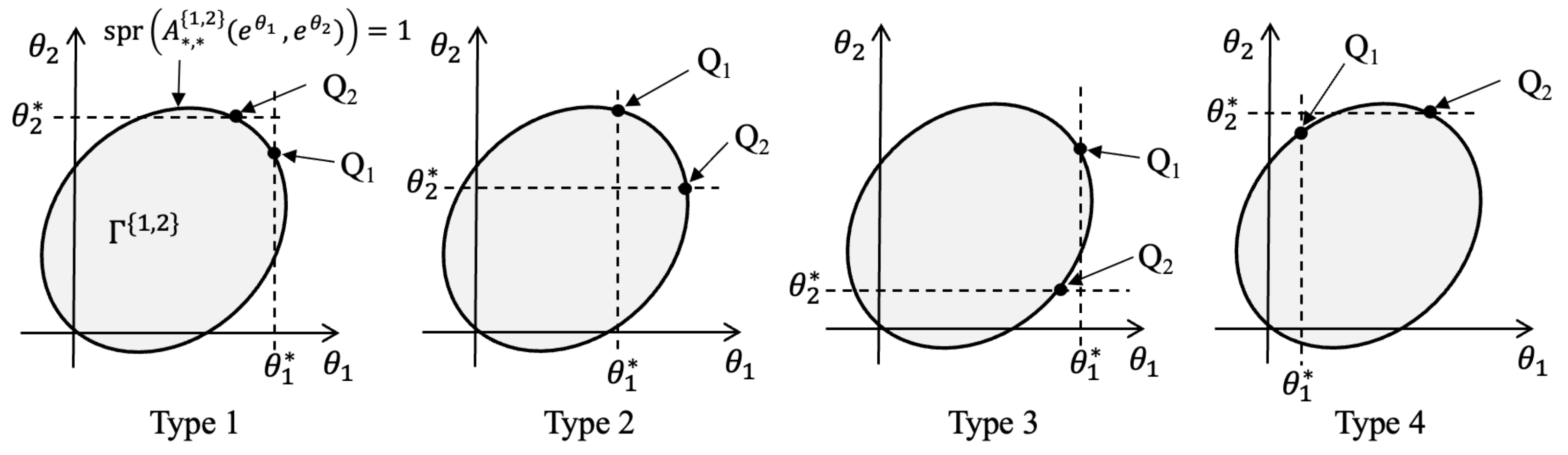} 
\caption{Classification}
\label{fig:classification}
\end{center}
\end{figure}

%
By Proposition 2.3 of \cite{Ozawa22}, for any direction vector $\bc=(c_1,c_2)\in\mathbb{N}^2$, the asymptotic decay rate in the direction $\bc$ is space homogeneous. Hence, we denote it by $\xi_{\bc}$, which satisfies, for any $(\bx,j)\in\mathbb{Z}_+^2\times S_0$, 
\begin{equation}
\xi_{\bc} = - \lim_{k\to\infty} \frac{1}{k} \log \nu_{(k\bc+\bx,j)}. 
\end{equation}
The asymptotic decay rate $\xi_{\bc}$ has already been obtained in \cite{Ozawa22}, and  as described in Section 4.1 of \cite{Ozawa22}, it is given as follows.
\begin{theorem} \label{th:decay_rate} 
Let $\bc=(c_1,c_2)$ be an arbitrary direction vector in $\mathbb{N}^2$.
\begin{itemize}
\item[] Type 1:
\[
\xi_{\bc} = \left\{ \begin{array}{ll}
c_1 \theta_1^* + c_2 \bar{\eta}_2(\theta_1^*) & \mbox{if $-\frac{c_1}{c_2}< \bar{\eta}'_2(\theta_1^*)$}, \cr
\theta_{\bc}^{max} & \mbox{if $\bar{\eta}'_2(\theta_1^*) \le -\frac{c_1}{c_2}\le \bar{\eta}'_1(\theta_2^*)^{-1}$}, \cr
c_1 \bar{\eta}_1(\theta_2^*) +c_2 \theta_2^* & \mbox{if $-\frac{c_1}{c_2}>\bar{\eta}'_1(\theta_2^*)^{-1}$}, 
\end{array} \right.
\]
where $\bar{\eta}'_2(x) = \frac{d}{dx} \bar{\eta}_2(x)$ and $\bar{\eta}'_1(x) = \frac{d}{dx} \bar{\eta}_1(x)$.
\item[] Type 2: 
\[
\xi_{\bc} = \left\{ \begin{array}{ll}
c_1 \theta_1^* + c_2 \bar{\eta}_2(\theta_1^*) & \mbox{if $-\frac{c_1}{c_2}\le \frac{\theta_2^*-\bar{\eta}_2(\theta_1^*)}{\bar{\eta}_1(\theta_2^*)-\theta_1^*}$}, \cr
c_1 \bar{\eta}_1(\theta_2^*) +c_2 \theta_2^* & \mbox{if $-\frac{c_1}{c_2}> \frac{\theta_2^*-\bar{\eta}_2(\theta_1^*)}{\bar{\eta}_1(\theta_2^*)-\theta_1^*}$}.
\end{array} \right.
\]
\item[] Type 3: $\xi_{\bc}=c_1 \bar{\eta}_1(\theta_2^*)+c_2 \theta_2^*$.
\item[] Type 4: $\xi_{\bc}=c_1 \theta_1^*+c_2 \bar{\eta}_2(\theta_1^*)$. 
\end{itemize}
\end{theorem}

Under Assumption \ref{as:2dQBD_stable}, since $\Gamma^{\{1,2\}}\cap\Gamma^{\{1\}}\cap\Gamma^{\{2\}}$ includes positive points, we have $\xi_{\bc}>0$ for every $\bc\in\mathbb{N}^2$. 
%
The asymptotic decay function $h_{\bc}(k)$ in the direction $\bc$ is defined as the function that satisfies, for some positive vector $\bg_{\bc}$, 
\begin{equation}
\lim_{k\to\infty} \frac{\bnu_{k\bc}}{h_{\bc}(k)} = \bg_{\bc}. 
\end{equation}
It is given as follows.
\begin{theorem} \label{th:main_theorem}
Let $\bc$ be an arbitrary direction vector in $\mathbb{N}^2$.  Then, the asymptotic decay function $h_{\bc}(k)$ is given as
\begin{equation}
h_{\bc}(k) = 
\left\{ \begin{array}{ll}
k^{-\frac{1}{2}} e^{-\xi_{\bc} k} & \mbox{if $\bar{\eta}'_2(\theta_1^*) < -\frac{c_1}{c_2} < \bar{\eta}'_1(\theta_2^*)^{-1}$ in Type 1}, \cr
e^{-\xi_{\bc} k} & \mbox{otherwise}.
\end{array} \right. 
\label{eq:main_theorem}
\end{equation}
\end{theorem}

Except for the case where $\bar{\eta}'_2(\theta_1^*) \le -\frac{c_1}{c_2} \le \bar{\eta}'_1(\theta_2^*)^{-1}$ in Type 1, Theorem \ref{th:main_theorem} has already been proved in \cite{Ozawa22}, see Theorem 3.2 therein. Hence, to this end, it suffices to prove the following proposition. 
\begin{proposition} \label{pr:power_term}
Assume Type 1 and set $\bc=(c_1,c_2)=(1,1)$. Then, the asymptotic decay function $h_{\bc}(k)$ is given as
\begin{equation}
h_{\bc}(k) =
\left\{ \begin{array}{ll}
k^{-\frac{1}{2}} e^{-\theta_{\bc}^{max} k} & \mbox{if $\bar{\eta}'_2(\theta_1^*) < -\frac{c_1}{c_2} =-1< \bar{\eta}'_1(\theta_2^*)^{-1}$}, \cr
e^{-\theta_{\bc}^{max} k} & \mbox{if $\bar{\eta}'_2(\theta_1^*) =-1$ or  $\bar{\eta}'_1(\theta_2^*)=-1$}.
\end{array} \right. 
\label{eq:power_term}
\end{equation}
\end{proposition}

From this proposition, we can obtain the same result for a general direction vector $\bc\in\mathbb{N}^2$, by using the block state process derived from the original 2d-QBD process; See Section 3.3 of \cite{Ozawa22}. We, therefore, prove the proposition in Section \ref{sec:gf_analytic_properties}. 
The asymptotic decay function is space-homogeneous with respect to the level of the 2d-QBD process, as follows. We prove this corollary in Section \ref{sec:decay_function_homogeneous}.
\begin{corollary} \label{co:decay_function_homogeneous}
Let $\bc$ be an arbitrary direction vector in $\mathbb{N}^2$. For every $\bx\in\mathbb{Z}_+^2$, the asymptotic function $h_{\bc}(k)$ in Theorem \ref{th:main_theorem} satisfies, for some positive vector $\bg_{\bc,\bx}$, 
\begin{equation}
\lim_{k\to\infty} \frac{\bnu_{k\bc+\bx}}{h_{\bc}(k)} = \bg_{\bc,\bx}. 
\end{equation}
\end{corollary}

\begin{remark}
When the model is a 2d-RRW without a background process, the result of Theorem \ref{th:main_theorem} coincides with that of Theorem 1 of Malyshev \cite{Malyshev73} (also see Theorem 7.1 of \cite{Fayolle99} ). For the correspondence between them, see Section 4.1 of \cite{Ozawa22}. 
\end{remark}

%
%
\section{Preliminaries} \label{sec:preliminary}

Let $z$ and $w$ be complex valuables unless otherwise stated. 
For a positive number $r$, denote by $\Delta_r$ the open disk of center $0$ and radius $r$ on the complex plain, and $\partial \Delta_r$ the circle of the same center and radius.  
%
For $a,b\in\mathbb{R}_+$ such that $a<b$, let $\Delta_{a,b}$ be an open annular domain on $\mathbb{C}$ defined as $\Delta_{a,b} =\{z\in\mathbb{C}: a<|z|<b\}$. We denote by $\bar{\Delta}_{a,b}$ the closure of $\Delta_{a,b}$. 
For $r>0$, $\varepsilon>0$ and $\theta\in[0,\pi/2)$, define 
\[
\tilde{\Delta}_r(\varepsilon,\theta) = \{z\in\mathbb{C}: |z|<r+\varepsilon,\ z\ne r,\ |\arg(z-r)|>\theta \}. 
\]
For $r>0$, we denote by ``$\tilde{\Delta}_r\ni z\to r$" that $\tilde{\Delta}_r(\varepsilon,\theta)\ni z\to r$ for some $\varepsilon>0$ and some $\theta\in [0,\pi/2)$. 
In the rest of the paper, instead of proving that a function $f(z)$ is analytic in $\tilde{\Delta}_r(\varepsilon,\theta)$ for some $\varepsilon>0$ and $\theta\in[0,\pi/2)$, we often demonstrate that the function $f(z)$ is analytic in ${\Delta}_r$ and on $\partial\Delta_r\setminus\{r\}$.

In order to give general results, this section is described independently from other parts of the paper.

%
%
\subsection{Analytic extension of a G-matrix function} \label{sec:G_analyticextension}

%
First, we define a G-matrix function according to \cite{Ozawa18}. 
For $i,j\in\{-1,0,1\}$, let $A_{i,j}$ be a substochastic matrix with a finite dimension $s_0$, and define the following matrix functions:
\[
A_{*,j}(z) = \sum_{i\in\{-1,0,1\}} z^i A_{i,j},\ j=-1,0,1,\quad 
A_{*,*}(z,w)=\sum_{i,j\in\{-1,0,1\}} z^i w^j A_{i,j}. 
\]
We assume the following condition. 
\begin{assumption} \label{as:Aij_stochastic_pre}
$A_{*,*}(1,1)$ is stochastic. 
\end{assumption}
Let $\chi(z,w)$ be the spectral radius of $A_{*.*}(z,w)$, i.e., $\chi(z,w)=\spr(A_{*,*}(z,w))$, and $\Gamma$ be a domain on $\mathbb{R}^2$ defined as 
\[
\Gamma=\{(\theta_1,\theta_2)\in\mathbb{R}^2: \chi(e^{\theta_1},e^{\theta_2})<1\}.
\]
We assume the following condition. 
\begin{assumption} \label{as:Aij_irreducible_pre}
The Markov modulated random walk on $\mathbb{Z}^2\times\{1,2,...,s_0\}$ governed by $\{A_{i,j}: i,j\in\{-1,0,1\} \}$ is irreducible and aperiodic. 
\end{assumption}

Under this assumption, $A_{*,*}(1,1)$ is also irreducible and aperiodic. Furthermore, by Lemma 2.2 of \cite{Ozawa18}, $\Gamma$ is bounded. 
Since $\chi(e^{\theta_1},e^{\theta_2})$ is convex in $(\theta_1, \theta_2)\in\mathbb{R}^2$, the closure of $\Gamma$ is a convex set. Define extreme points $\theta_1^{min}$ and $\theta_1^{max}$ as follows:
\[
\theta_1^{min} = \inf_{(\theta_1,\theta_2)\in\Gamma} \theta_1,\quad 
\theta_1^{max} = \sup_{(\theta_1,\theta_2)\in\Gamma} \theta_1.
\]
For $\theta_1\in[\theta_1^{min}, \theta_1^{max}]$, let $\underline{\theta}_2(\theta_1)$ and $\bar{\theta}_2(\theta_1)$ be the two real solutions to equation $\chi(e^{\theta_1},e^{\theta_2})=1$, counting multiplicity, where $\underline{\theta}_2(\theta_1)\le\bar{\theta}_2(\theta_1)$. 
%
For $n\ge 1$, define the following set of index sequences: 
\begin{align*}
%
&\scrI_n = \biggl\{\bi_{(n)}\in\{-1,0,1\}^n:\ \sum_{l=1}^k i_l\ge 0\ \mbox{for $k\in\{1,2,...,n-1\}$}\ \mbox{and} \sum_{l=1}^n i_l=-1 \biggr\}, 
%
\end{align*}
where $\bi_{(n)}=(i_1,i_2,...,i_n)$, and define the following matrix function: 
\begin{align*}
%
&D_n(z) = \sum_{\bi_{(n)}\in\scrI_n} A_{*,i_1}(z) A_{*,i_2}(z) \cdots A_{*,i_n}(z). 
%
\end{align*}
%
Define a matrix function $G(z)$ as 
\begin{align*}
G(z) = \sum_{n=1}^\infty D_n(z). 
\end{align*}
By Lemma 4.1 of \cite{Ozawa18}, this matrix series absolutely converges entry-wise in $z\in\bar{\Delta}_{e^{\theta_1^{min}},e^{\theta_1^{max}}}$. We call this $G(z)$ the G-matrix function generated from $\{A_{i,j}: i,j\in\{-1,0,1\}\}$. 
For $z\in\bar{\Delta}_{e^{\theta_1^{min}}, e^{\theta_1^{max}}}$, $G(z)$ satisfies the inequality $|G(z)|\le G(|z|)$ and the following matrix quadratic equation: 
\begin{align}
A_{*,-1}(z)+A_{*,0}(z) G(z) +A_{*,1}(z) G(z)^2 = G(z). \label{eq:Gfunction_equation}
\end{align}
Furthermore, for $x\in[e^{\theta_1^{min}}, e^{\theta_1^{max}}]$, $G(x)$ is the minimum nonnegative solution to equation \eqref{eq:Gfunction_equation}. Hence, $G(z)$ is an extension of a usual G-matrix in the queueing theory, see, for example, \cite{Neuts94}. By Proposition 2.5 of \cite{Ozawa18}, we see that, for $x\in[e^{\theta_1^{min}}, e^{\theta_1^{max}}]$, the Perron-Frobenius eigenvalue of $G(x)$ is given by $e^{\underline{\theta}_2(\log x)}$, i.e., $\spr(G(x))=e^{\underline{\theta}_2(\log x)}$. 
By Lemma 4.1 of \cite{Ozawa18}, $G(z)$ satisfies
\begin{align}
I-A_{*,*}(z,w) &= w^{-1} \left( I-A_{*,0}(z)-w A_{*,1}(z)+A_{*,1}(z) G(z) \right) (w I-G(z)). 
\label{eq:Ass_Gz_relation}
\end{align}
%
%
%
By Lemma 4.2 of \cite{Ozawa18}, the following property holds true for $G(z)$.
\begin{lemma} \label{le:G_analytic1}
$G(z)$ is entry-wise analytic in the open annular domain $\Delta_{e^{\theta_1^{min}},e^{\theta_1^{max}}}$. 
\end{lemma}

%
We give the eigenvalues of $G(z)$  according to \cite{Ozawa18}. Note that our final aim in this subsection is to give an analytic extension of $G(z)$ through its Jordan canonical form without assuming all the eigenvalues of $G(z)$ are distinct. In \cite{Ozawa18}, the eigenvalues were assumed to be distinct. 
Define a matrix function $L(z,w)$ as 
\[
L(z,w) = z w (I-A_{*,*}(z,w)). 
\]
Each entry of $L(z,w)$ is a polynomial in $z$ and $w$ with at most degree $2$ for each variable. 
We use a notation $\Xi$, defined as follows. Let $f(z,w)$ be an irreducible polynomial in $z$ and $w$ and assume its degree with respect to $w$ is $m\ge 1$. Let $a(z)$ be the coefficient of $w^m$ in $f(z,w)$. Define a point set $\Xi(f)$ as 
\begin{align*}
\Xi(f) 
&= \{ z\in\mathbb{C}: \mbox{$a(z)=0$ or ($f(z,w)=0$ and $f_w(z,w)=0$ for some $w\in\mathbb{C}$)} \}, 
\end{align*}
where $f_w(z,w)=(\partial/\partial w)f(z,w)$. Each point in $\Xi(f)$ is an algebraic singularity of the algebraic function $w=\alpha(z)$ defined by polynomial equation $f(z,w)=0$. For each point $z\in\mathbb{C}\setminus\Xi(f)$, $f(z,w)=0$ has just $m$ distinct solutions, which correspond to the $m$ branches of the algebraic function. 
Let $\phi(z,w)$ be a polynomial in $z$ and $w$ defined as 
\[
\phi(z,w)=\det L(z,w)
\]
and $m_\phi$ its degree with respect to $w$, where $s_0\le m_\phi \le 2 s_0$.  Let $\alpha_1(z)$, $\alpha_2(z)$, ..., $\alpha_{m_\phi}(z)$ be the $m_\phi$ branches of the algebraic function $w=\alpha(z)$ defined by the polynomial equation $\phi(z,w)=0$, counting multiplicity. 
We number the brunches so that they satisfy the following:
\begin{itemize}
\item[(1)] For every $z\in\bar{\Delta}_{e^{\theta_1^{min}},e^{\theta_1^{max}}}$ and for every $k\in\{1,2,...,s_0\}$, $|\alpha_k(z)| \le e^{\underline{\theta}_2(\log |z|)}$. 
\item[(2)] For every $z\in\bar{\Delta}_{e^{\theta_1^{min}},e^{\theta_1^{max}}}$ and for every $k\in\{s_0+1,s_0+2,...,m_\phi\}$, $|\alpha_k(z)| \ge e^{\bar{\theta}_2(\log |z|)}$. 
\item[(3)] For every $x\in[e^{\theta_1^{min}},e^{\theta_1^{max}}]$, $\alpha_{s_0}(x)=e^{\underline{\theta}_2(\log x)}$ and $\alpha_{s_0+1}(x)=e^{\bar{\theta}_2(\log x)}$. 
\end{itemize}
This is possible by Lemma 4.3 of \cite{Ozawa18}. By Lemmas 4.3 and 4.4 of \cite{Ozawa18}, the G-matrix function $G(z)$ satisfies the following property.
\begin{lemma} \label{le:Gmatrix_eigenvalues}
For every $z\in\bar{\Delta}_{e^{\theta_1^{min}},e^{\theta_1^{max}}}$, the eigenvalues of $G(z)$ are given by $\alpha_1(z)$, $\alpha_2(z)$, ..., $\alpha_{s_0}(z)$. 
\end{lemma}

%
Without loss of generality, we assume that, for some $n_\phi\in\mathbb{N}$ and $l_1,l_2,...,l_{n_\phi}\in\mathbb{N}$, the polynomial $\phi(z,w)$ is factorized as 
\begin{equation}
\phi(z,w) = f_1(z,w)^{l_1} f_2(z,w)^{l_2} \cdots f_{n_\phi}(z,w)^{l_{n_\phi}}, 
\label{eq:phi_fk}
\end{equation}
where $f_k(z,w),\,k=1,2,...,n_\phi$, are irreducible polynomials in $z$ and $w$ and they are relatively prime. Since the field of coefficients of polynomials is $\mathbb{C}$, this factorization is unique. 
For every $k\in\{1,2,...,m_\phi\}$, $\alpha_k(z)$ is a branch of the algebraic function $w=\alpha(z)$ defined by the polynomial equation $f_n(z,w)=0$ for some $n\in\{1,2,...,n_\phi\}$. We denote such $n$ by $q(k)$, i.e., $f_{q(k)}(z,\alpha_k(z))=0$. Since $\alpha_{s_0}(z)$ is the Perron-Frobenius eigenvalue of $G(z)$ when $z\in[e^{\theta_1^{min}},e^{\theta_1^{max}}]$, the multiplicity of $\alpha_{s_0}(z)$ is one and we have $l_{q(s_0)}=1$. 
Define a point set $\calE_1$ as 
\[
\calE_1 = \bigcup_{n=1}^{n_\phi} \Xi(f_n). 
\]
Since, for every $n$, the polynomial $f_n(z,w)$ is irreducible and not identically zero, the point set $\calE_1$ is finite. Every branch $\alpha_k(z)$ is analytic in $\mathbb{C}\setminus\calE_1$. 
Define a point set $\calE_2$ as
\begin{align*}
\calE_2 
&= \{z\in\mathbb{C}\setminus\calE_1: \mbox{$f_{n}(z,w)=f_{n'}(z,w)=0$ } \cr
&\qquad\qquad \mbox{for some $n,n'\in\{1,2,...,n_\phi\}$ such that $n\ne n'$ and for some $w\in\mathbb{C}$} \}.
\end{align*}
Since, for any $n,n'$ such that $n\ne n'$, $f_{n}(z,w)$ and $ f_{n'}(z,w)$ are relatively prime, the point set $\calE_2$ is finite. Note that every branch $\alpha_k(z)$ is analytic in a neighborhood of any $z_0\in\calE_2$. 
For every $k\in\{1,2,...,m_\phi\}$ and for every $z\in\mathbb{C}\setminus(\calE_1\cup\calE_2)$, the multiplicity of $\alpha_k(z)$ as a zero of $\det L(z,w)$ is equal to $l_{q(k)}$. This means that, for every $z\in\bar{\Delta}_{e^{\theta_1^{min}},e^{\theta_1^{max}}}\setminus(\calE_1\cup\calE_2)$, the multiplicity of the eigenvalue $\alpha_k(z)$ of $G(z)$ is $l_{q(k)}$, which does not depend on $z$. 
Define a positive integer $m_0$ as 
\[
m_0 = \sum_{k=1}^{s_0} \frac{1}{l_{q(k)}}.
\]
This $m_0$ is the number of different branches in $\{\alpha_i(z): i=1,2,...,s_0\}$ when $z\in\mathbb{C}\setminus(\calE_1\cup\calE_2)$. Denote the different branches by $\check{\alpha}_k(z),\, k=1,2,...,m_0,$ so that $\check{\alpha}_{m_0}(z)=\alpha_{s_0}(z)$. Instead of using $q(k)$, we define a function $\check{q}(k)$ so that $l_{\check{q}(k)}$ indicates the multiplicity of $\check{\alpha}_k(z)$ when $z\in\mathbb{C}\setminus(\calE_1\cup\calE_2)$. We always have $l_{\check{q}(m_0)}=1$. 

%
We give the Jordan normal form of $G(z)$. 
Define a domain $\Omega$ as $\Omega=\Delta_{e^{\theta_1^{min}},e^{\theta_1^{max}}}\setminus(\calE_1\cup\calE_2)$. 
For $k\in\{1,2,...,m_0\}$ and for $i\in\{1,2,...,l_{\check{q}(k)}\}$, define a positive integer $t_{k,i}$ as 
\[
t_{k,i} = \min_{z\in\Omega} \dim \Ker\, (\check{\alpha}_k(z) I-G(z))^i
\]
and a point set $\calG_{k,i}$ as 
\[
\calG_{k,i}=\{z\in\Omega: \dim \Ker\, (\check{\alpha}_k(z) I-G(z))^i > t_{k,i}\}.
\]
Since $\check{\alpha}_k(z)$ and $G(z)$ are analytic in $\Omega$, we see from the proof of Theorem S6.1 of \cite{Gohberg09} that each $\calG_{k,i}$ is an empty set or a set of discrete complex numbers. For $k\in\{1,2,...,m_0\}$ and $i\in\{1,2,...,l_{\check{q}(k)}\}$, define a nonnegative integer $s_{k,i}$ as 
\[
s_{k,i} = 2 t_{k,i} - t_{k,i+1} - t_{k,i-1},
\]
where $t_{k,0}=0$ and $t_{k,l_{\check{q}(k)}+1}=l_{\check{q}(k)}$. 
For $k\in\{1,2,...,m_0\}$, define a positive integer $m_{k,0}$ and point set $\calE^G_k$ as
\[
m_{k,0}=t_{k,1}, \qquad 
\calE^G_k = \bigcup_{i=1}^{l_{\check{q}(k)}} \calG_{k,i}. 
\]
When $z\in\Omega\setminus\calE^G_k$, this $m_{k,0}$ is the number of Jordan blocks of $G(z)$ with respect to the eigenvalue $\check{\alpha}_k(z)$ and, for $i\in\{1,2,...,l_{\check{q}(k)}\}$, $s_{k,i}$ is the number of Jordan blocks whose dimension is $i$. Hence, the Jordan normal form of $G(z)$ takes a common form in $z\in\Omega\setminus\bigcup_{k=1}^{m_0} \calE^G_k$. 
For $k\in\{1,2,...,m_0\}$ and for $i\in\{1,2,...,m_{k,0}\}$, denote by $m_{k,i}$ the dimension of the $i$-th Jordan block of $G(z)$ with respect to the eigenvalue $\check{\alpha}_k(z)$, where we number the Jordan blocks so that if $i\le i'$, $m_{k,i}\ge m_{k,i'}$. For each $k\in\{1,2,...,m_0\}$, they satisfy $\sum_{i=1}^{m_{k,0}} m_{k,i} = l_{\check{q}(k)}$. 
Denote by $J_n(\lambda)$ the $n$-dimensional Jordan block of eigenvalue $\lambda$. For $z\in\Omega\setminus\bigcup_{k=1}^{m_0} \calE^G_k$, the Jordan normal form of $G(z)$, $J^G(z)$, is given by
\begin{equation}
J^G(z) = \diag(J_{m_{k,i}}(\check{\alpha}_k(z)),\, k=1,2,...,m_0,\, i=1,2,...,m_{k,0}), 
\label{eq:Jordan_normal_Gz}
\end{equation}
where $m_{m_0,0}=1$ and $J_{m_{m_0,1}}(\check{\alpha}_{m_0}(z))=\alpha_{s_0}(z)$. Note that the matrix function $J^G(z)$ is well defined on $\mathbb{C}$ and analytic in $\mathbb{C}\setminus\calE_1$. 
An analytic extension of $G(z)$ is given by the following theorem. 
\begin{theorem} \label{th:Jordan_decomposition_G_pre} 
There exist vector functions: 
\[
\check{\bv}_{k,i,j}^L(z),\, k=1,2,...,m_0,\, i=1,2,...,m_{k,0},\, j=1,2,..., m_{k,i},
\]
such that they are analytic in $\mathbb{C}\setminus\calE_1$ and satisfy, for every $z\in\Delta_{e^{\theta_1^{min}},e^{\theta_1^{max}}}\setminus(\calE_1\cup\calE_0)$, 
\begin{equation}
G(z) = T^L(z) J^G(z) (T^L(z))^{-1}, 
\label{eq:Gz_analytic_extension}
\end{equation}
where $\calE_0$ is a set of discrete complex numbers and matrix function $T^L(z)$ is defined as 
\[
T^L(z) = 
\begin{pmatrix}
\check{\bv}_{k,i,j}^L(z),\, k=1,2,...,m_0,\, i=1,2,...,m_{k,0},\, j=1,2,..., m_{k,i}
\end{pmatrix}.
\]
\end{theorem}

Since the proof of Theorem \ref{th:Jordan_decomposition_G_pre} is elementary and very lengthy, we give it in Appendix \ref{sec:Jordan_decomposition_G_app}. 
In Theorem \ref{th:Jordan_decomposition_G_pre}, $\{\check{\bv}_{k,i,j}^L(z)\}$ is the set of the generalized eigenvectors of $G(z)$, but we denote them with superscript $L$ since they are generated from the matrix function $L(z,w)$; see Appendix \ref{sec:Jordan_decomposition_G_app}. 
Define a point set $\calE_T^L$ as
\[
\calE_T^L = \{z\in\mathbb{C}\setminus\calE_1: \det\,T^L(z) = 0\}, 
\]
which is an empty set or a set of discrete complex numbers since $\det\,T^L(z)$ is not identically zero. Define a matrix function $\check{G}(z)$ as 
\begin{equation}
\check{G}(z) = T^L(z) J^G(z) (T^L(z))^{-1} = \frac{T^L(z) J^G(z)\, \adj(T^L(z))}{\det(T^L(z))}.
\label{eq:G_Jordan_decomposition}
\end{equation}
Then, it is entry-wise analytic in $\mathbb{C}\setminus(\calE_1\cup\calE_T^L)$. By Theorem \ref{th:Jordan_decomposition_G_pre} and the identity theorem for analytic functions, this $\check{G}(z)$ is an analytic extension of the matrix function $G(z)$. Hence, we denote $\check{G}(z)$ by $G(z)$. 
By Lemma \ref{le:G_analytic1}, $G(z)$ is entry-wise analytic in $\Delta_{e^{\theta_1^{min}},e^{\theta_1^{max}}}$. The following corollary asserts that $G(z)$ is also analytic on the outside boundary of $\Delta_{e^{\theta_1^{min}},e^{\theta_1^{max}}}$ except for the point $z=e^{\theta_1^{max}}$.
\begin{corollary} \label{co:G_analytic_boundary}
The extended G-matrix function $G(z)$ is entry-wise analytic on $\partial\Delta_{e^{\theta_1^{max}}}\setminus\{e^{\theta_1^{max}}\}$. 
\end{corollary}

Since this corollary can be proved in a manner similar to that used in the proof of Lemma 4.7 of \cite{Ozawa18}, we omit it. 

%
Denote by $\check{\bu}^L_{m_0,1,1}(z)$ the last row of the matrix function $(T^L(z))^{-1}$, and define a diagonal matrix function $J_{s_0}(z)$ as $J_{s_0}(z) = \diag\!\begin{pmatrix} 0 & \cdots & 0 & \alpha_{s_0}(z) \end{pmatrix}$, where $\alpha_{s_0}(z)=\check{\alpha}_{m_0}(z)$. Then, since $m_{m_0,0}=1$ and $m_{m_0,1}=1$, we obtain the following decomposition of $G(z)$ from \eqref{eq:G_Jordan_decomposition}:
\begin{equation}
G(z) = G^\dagger(z) + \alpha_{s_0}(z) \check{\bv}^L_{m_0,1,1}(z) \check{\bu}^L_{m_0,1,1}(z),
\label{eq:G_desomposition_1}
\end{equation}
where 
\[
G^\dagger(z) = T^L(z) (J^G(z) - J_{s_0}(z)) (T^L(z))^{-1} .
\]
By the definition, $G(z)$ satisfies, for $n\ge 1$,  
\begin{equation}
G(z)^n = G^\dagger(z)^n + \alpha_{s_0}(z)^n \check{\bv}^L_{m_0,1,1}(z) \check{\bu}^L_{m_0,1,1}(z), 
\label{eq:G_desomposition_2}
\end{equation}
and $G^\dagger(z)$ satisfies, for $z\in\bar{\Delta}_{e^{\theta_1^{min}},e^{\theta_1^{max}}}$, $\spr(G^\dagger(z))\le\spr(G^\dagger(|z|))<\spr(G(|z|))=\alpha_{s_0}(|z|)$. Furthermore, in a neighborhood of $z=e^{\theta_1^{max}}$, we have $\spr(G^\dagger(z))<\alpha_{s_0}(e^{\theta_1^{max}})$. 
%
Since the point $z=e^{\theta_1^{max}}$ is a branch point of $\check{\alpha}_{m_0}(z)$ ($=\alpha_{s_0}(z)$), there exists a function $\tilde{\alpha}_{s_0}(\zeta)$ being analytic in a neighborhood of $\zeta=0$ and satisfying  
\[
\check{\alpha}_{m_0}(z) = \alpha_{s_0}(z) = \tilde{\alpha}_{s_0}((e^{\theta_1^{max}}-z)^{\frac{1}{2}}). 
\]
Let $\tilde{\bv}_{s_0}(\zeta)$ be a vector function satisfying 
\[
L(e^{\theta_1^{max}}-\zeta^2,\tilde{\alpha}_{s_0}(\zeta)) \tilde{\bv}_{s_0}(\zeta) = \bzero, 
\]
where $\tilde{\bv}_{s_0}(\zeta)$ is element-wise analytic in a neighborhood of $\zeta=0$. 
Denote by $\tilde{T}(\zeta)$ the matrix function given by replacing the last column of $T^L(e^{\theta_1^{max}}-\zeta^2)$ with $\tilde{\bv}_{s_0}(\zeta)$, and denote by $\tilde{\bu}_{s_0}(\zeta)$ the last row of $\tilde{T}(\zeta)^{-1}$. By the definition, $\tilde{T}(\zeta)$ as well as $\tilde{\bu}_{s_0}(\zeta)$ is entry-wise analytic in a neighborhood of $\zeta=0$. 
Define a diagonal matrix function $\tilde{J}_{s_0}(\zeta)$ as $\tilde{J}_{s_0}(\zeta) = \diag\!\begin{pmatrix} 0 & \cdots & 0 & \tilde{\alpha}_{s_0}(\zeta) \end{pmatrix}$.
%
For later use, we give the following lemma.
\begin{lemma} \label{le:G_decomposition}
There exists a matrix function $\tilde{G}(\zeta)$ being entry-wise analytic in a neighborhood of $\zeta=0$ and satisfying $G(z) = \tilde{G}((e^{\theta_1^{max}}-z)^{\frac{1}{2}})$ in a neighborhood of $z=e^{\theta_1^{max}}$. 
This $\tilde{G}(\zeta)$ is represented as 
\begin{equation}
\tilde{G}(\zeta) = \tilde{G}^\dagger(\zeta) + \tilde{\alpha}_{s_0}(\zeta) \tilde{\bv}_{s_0}(\zeta) \tilde{\bu}_{s_0}(\zeta),
\label{eq:tildeG_desomposition_1}
\end{equation}
where $\tilde{G}^\dagger(\zeta)$ is a matrix function being entry-wise analytic in a neighborhood of $\zeta=0$ and satisfying $G^\dagger(z)=\tilde{G}^\dagger((e^{\theta_1^{max}}-z)^{\frac{1}{2}})$ in a neighborhood of $z=e^{\theta_1^{max}}$. In a neighborhood of $\zeta=0$, $\spr(\tilde{G}^\dagger(\zeta))<\tilde{\alpha}_{s_0}(0)=\alpha_{s_0}(e^{\theta_1^{max}})$. 
\end{lemma}

\begin{proof}
Give $\tilde{G}^\dagger(\zeta)$ as 
\[
\tilde{G}^\dagger(\zeta) = \tilde{T}(\zeta) (J^G(e^{\theta_1^{max}}-\zeta^2) - J_{s_0}(e^{\theta_1^{max}}-\zeta^2)) \tilde{T}(\zeta)^{-1} .
\]
Then, by \eqref{eq:G_desomposition_1}, we obtain the results of the lemma. 
\end{proof}

%
The following limit with respect to $\alpha_{s_0}(z)$ is given by Proposition 5.5 of \cite{Ozawa18} (also see Lemma 10 of \cite{Kobayashi13}).
\begin{lemma} \label{le:alpha_s0_limit}
\begin{align}
\lim_{\tilde{\Delta}_{e^{\theta_1^{max}}}\ni z\to e^{\theta_1^{max}}} \frac{\alpha_{s_0}(e^{\theta_1^{max}})-\alpha_{s_0}(z)}{(e^{\theta_1^{max}}-z)^{\frac{1}{2}}} 
&= -\alpha_{s_0,1}
= \frac{\sqrt{2}}{\sqrt{-\bar{\zeta}_{1,w^2}(\underline{\zeta}_2(e^{\theta_1^{max}}))}} > 0, 
\label{eq:limit_eigen_G1}
\end{align}
where $z=\bar{\zeta}_1(w)$ is the larger one of the two real solutions to equation $\chi(z,w)=1$ and $\bar{\zeta}_{1,w^2}(w)=(d^2/d w^2)\, \bar{\zeta}_1(w)$.
\end{lemma}

%
Let $R(z)$ be the rate matrix function generated from $\{A_{i,j}: i,j=-1,0,1\}$; for the definition of $R(z)$, see Section 4.1 of \cite{Ozawa18}. Define a matrix function $N(z)$ as 
\[
N(z)=(I-A_{*,0}(z)-A_{*,1}(z)G(z))^{-1}.
\]
$N(z)$ is well defined for every $z\in\bar{\Delta}_{e^{\theta_1^{min}},e^{\theta_1^{max}}}$. 
The extended $G(z)$ satisfies the following property.
\begin{lemma} \label{le:G_limit}
\begin{align}
\lim_{\tilde{\Delta}_{e^{\theta_1^{max}}}\ni z\to e^{\theta_1^{max}}} \frac{G(e^{\theta_1^{max}})-G(z)}{(e^{\theta_1^{max}}-z)^{\frac{1}{2}}} 
&= -G_1 \cr
&=  -\alpha_{s_0,1} N(e^{\theta_1^{max}}) \bv^{R}(e^{\theta_1^{max}}) \bu^G_{s_0}(e^{\theta_1^{max}}) \ge O,\ \ne O, 
\label{eq:limit_G1}
\end{align}
where $\bu^G_{s_0}(e^{\theta_1^{max}})$ is the left eigenvector of $G(e^{\theta_1^{max}})$ with respect to the eigenvalue $e^{\underline{\theta}_2(\theta_1^{max})}=\alpha_{s_0}(e^{\theta_1^{max}})$, $\bv^{R}(e^{\theta_1^{max}})$ the right eigenvector of $R(e^{\theta_1^{max}})$ with respect to the eigenvalue $e^{-\bar{\theta}_2(\theta_1^{max})}=e^{-\underline{\theta}_2(\theta_1^{max})}$ and they satisfy  $\bu^G_{s_0}(e^{\theta_1^{max}}) N(e^{\theta_1^{max}}) \bv^{R}(e^{\theta_1^{max}}) =1$.
\end{lemma}

Since this lemma can be proved in a manner similar to that used in the proof of Proposition 5.6 of \cite{Ozawa18}, we omit it.

%
%
\subsection{G-matrix in the reverse direction} \label{sec:Gmatrix_reverse}

Let $A_{-1}$, $A_0$ and $A_1$ be square nonnegative matrices \textit{with a finite dimension}. Define a matrix function $A_*(z)$ and matrix $Q$ as
\begin{align}
&A_*(z) = z^{-1} A_{-1} + A_0 + z A_1, \label{eq:As_definition} \\
&Q = \begin{pmatrix}
A_0 & A_1 & & & \cr
A_{-1} & A_0 & A_1 & & \cr
& A_{-1} & A_0 & A_1 & \cr
& & \ddots & \ddots & \ddots 
\end{pmatrix}. \label{eq:Q_definitoin} 
\end{align}
We assume: 
\begin{itemize}
\item[(a1)] $Q$ is irreducible. 
\item[(a2)] The infimum of the maximum eigenvalue of $A_*(e^\theta)$ in $\theta\in\mathbb{
R}$ is less than or equal to $1$, i.e., 
$\inf_{\theta\in\mathbb{R}} \spr(A_*(e^\theta)) \le 1$.
\end{itemize}
Then, there exist two real solutions to equation $\spr(A_*(e^\theta))=1$, counting multiplicity, see comments to Condition 2.6 of \cite{Ozawa21}. We denote the solutions by $\underline{\theta}$ and $\bar\theta$, where $\underline{\theta}\le\bar\theta$. The rate matrix and G-matrix generated from the triplet $\{A_{-1},A_0,A_1\}$ also exist; we denote them by $R$ and $G$, respectively.  $R$ and $G$ are the minimal nonnegative solutions to the following matrix quadratic equations:
\begin{align}
&R = R^2 A_{-1} + R A_0 + A_1, \label{eq:R_definition} \\
&G = A_{-1} + A_0 G + A_1 G^2. \label{eq:G_definition}
\end{align}
We have
\begin{align}
& I-A_*(z) = (I-z R)(I-H)(I-z^{-1} G),  \label{eq:ImAs_fact} \\
& \spr(R)=e^{-\bar\theta},\quad \spr(G)=e^{\underline{\theta}} \label{eq:R_G_cp}, 
\end{align}
where $H=A_0+A_1 G$; see, for example, Lemma 2.2 of \cite{Ozawa21}.
We define a rate matrix and G-matrix \textit{in the reverse direction} generated from the triplet $\{A_{-1},A_0,A_1\}$, denoted by $R^r$ and $G^r$, as the minimal nonnegative solutions to the following matrix quadratic equations:
\begin{align}
&R^r = (R^r)^2 A_1 + R^r A_0 + A_{-1}, \label{eq:Rr_definition} \\
&G^r = A_1 + A_0 G^r + A_{-1} (G^r)^2. \label{eq:Gr_definition}
\end{align}
In other words,  $R^r$ and $G^r$ are, respectively, the rate matrix and G-matrix generated from the triplet by exchanging $A_{-1}$ and $A_1$.  
Since $z^{-1} A_1+A_0+z A_{-1}=A_*(z^{-1})$, we obtain, by \eqref{eq:ImAs_fact} and \eqref{eq:R_G_cp}, 
\begin{align}
& I-A_*(z^{-1}) = (I-z R^r)(I-H^r)(I-z^{-1} G^r),  \label{eq:ImAs_fact_reverse}\\
& \spr(R^r)=e^{\underline{\theta}},\quad \spr(G^r)=e^{-\bar\theta},
\end{align}
where $H^r=A_0+A_{-1}G^r$. 
We use the following property in the proof of Proposition \ref{pr:spr_hatU_eq1}. 
\begin{lemma} \label{le:G_Gr_v}
Let $\bv$ be the right eigenvector of $G$ with respect to the eigenvalue $e^{\underline{\theta}}$ and $\bv^r$ that of $G^r$ with respect to the eigenvalue $e^{-\bar{\theta}}$, i.e., $G\bv=e^{\underline{\theta}}\bv$ and $G^r\bv^r=e^{-\bar{\theta}}\bv^r$. 
If $\underline{\theta}=\bar\theta$, we have $\bv=\bv^r$, up to multiplication by a positive constant. 
\end{lemma}

\begin{proof}
By \eqref{eq:ImAs_fact} and \eqref{eq:ImAs_fact_reverse}, we have
\[
A_*(e^{\underline{\theta}})\bv=\bv,\quad 
A_*(e^{\bar{\theta}})\bv^r=A_*(e^{\underline{\theta}})\bv^r=\bv^r.
\]
Since $\spr(A_*(e^{\underline{\theta}}))=1$ and $A_*(e^{\underline{\theta}})$ is irreducible, the right eigenvector of $A_*(e^{\underline{\theta}})$ with respect to the eigenvalue of $1$ is unique, up to multiplication by a positive constant. This implies $\bv=\bv^r$. 
\end{proof}

%
%
\section{Proof of Proposition \ref{pr:power_term} and Corollary \ref{co:decay_function_homogeneous}} \label{sec:gf_analytic_properties}

%
%
\subsection{Methodology and outline of the proof} \label{sec:methodology}

Define the vector generating function of the stationary probabilities in direction $\bc\in\mathbb{N}^2$, $\bvarphi^{\bc}(z)$, as
\[
\bvarphi^{\bc}(z) = \sum_{k=0}^\infty z^k \bnu_{k\bc}. 
\]
Hereafter, we set $\bc=(1,1)$.
In order to obtain the asymptotic function of the stationary tail probabilities in the direction $\bc$, we apply the following lemma to $\bvarphi^{\bc}(z)$. 
\begin{lemma}[Theorem VI.4 of Flajolet and Sedgewick \cite{Flajolet09}] \label{le:asympto_formula_ss}
Let $f$ be a generating function of a sequence of real numbers $\{a_n,\, n\in\mathbb{Z}_+\}$, i.e., $f(z)=\sum_{n=0}^\infty a_n z^n$. If $f(z)$ is singular at $z=z_0>0$ and analytic in $\tilde{\Delta}_{z_0}(\varepsilon,\theta)$ for some $\varepsilon>0$ and some $\theta\in[0,\pi/2)$ and if it satisfies 
\begin{align}
\lim_{\tilde{\Delta}_{z_0}\ni z\to z_0} (z_0-z)^\alpha f(z) = c_0
\label{eq:lim_fz_alpha}
\end{align}
for $\alpha\in\mathbb{R}\setminus\{0,-1,-2,...\}$ and some nonzero constant $c_0\in\mathbb{R}$, then 
\begin{align}
\lim_{n\to\infty} \left( \frac{n^{\alpha-1}}{\Gamma(\alpha)} z_0^{-n} \right)^{-1} a_n = c
\label{eq:asymp_an}
\end{align}
for some real number $c$, where $\Gamma(z)$ is the gamma function. This means that the asymptotic function of the sequence $\{a_n\}$ is given by $n^{\alpha-1} z_0^{-n}$. 
\end{lemma}

For the purpose, we prove the following propositions in Section \ref{sec:phic_expansion}. 

\begin{proposition} \label{pr:varphic_analytic_domain}
Assume Type 1. If $\bar{\eta}_1'(\theta_2^*)\le -c_1/c_2=-1 \le 1/\bar{\eta}_2'(\theta_1^*)$, the vector function $\bvarphi^{\bc}(z)$ is element-wise analytic in $\tilde{\Delta}_{e^{\theta_{\bc}^{max}}}(\varepsilon,\theta)$ for some $\varepsilon>0$ and some $\theta\in[0,\pi/2)$.
\end{proposition}

\begin{proposition} \label{pr:varphic_expansion}
Assume Type 1. If $\bar{\eta}_1'(\theta_2^*)\le -c_1/c_2=-1 \le 1/\bar{\eta}_2'(\theta_1^*)$, there exists a vector function $\tilde{\bvarphi}^{\bc}(\zeta)$ being meromorphic in a neighborhood of $\zeta=0$ and satisfying $\bvarphi^{\bc}(z)=\tilde{\bvarphi}^{\bc}((e^{\theta_{\bc}^{max}}-z)^{\frac{1}{2}})$ in a neighborhood of $z=e^{\theta_{\bc}^{max}}$. 
If $\bar{\eta}_1'(\theta_2^*)< -1 < 1/\bar{\eta}_2'(\theta_1^*)$, the point $\zeta=0$ is a pole of $\tilde{\bvarphi}^{\bc}(\zeta)$ with at most order one; if $\bar{\eta}_1'(\theta_2^*)=-1$ or  $\bar{\eta}_2'(\theta_1^*)=-1$, it is a pole of $\tilde{\bvarphi}^{\bc}(\zeta)$ with at most order two.
\end{proposition}

By Proposition \ref{pr:varphic_expansion}, if $\bar{\eta}_1'(\theta_2^*)< -1 < 1/\bar{\eta}_2'(\theta_1^*)$, the Puiseux series of $\bvarphi^{\bc}(z)$ is represented as 
\begin{equation}
\bvarphi^{\bc}(z) = \sum_{k=-1}^\infty \bvarphi^{\bc}_{1,k} (e^{\theta_{\bc}^{max}}-z)^{\frac{k}{2}}, 
\label{eq:phic_Puiseux_1}
\end{equation}
where $\{\bvarphi^{\bc}_{1,k}\}$ is the sequence of coefficient vectors; if $\bar{\eta}_1'(\theta_2^*)=-1$ or  $\bar{\eta}_2'(\theta_1^*)=-1$, it is represented as 
\begin{equation}
\bvarphi^{\bc}(z) = \sum_{k=-2}^\infty \bvarphi^{\bc}_{2,k} (e^{\theta_{\bc}^{max}}-z)^{\frac{k}{2}}, 
\label{eq:phic_Puiseux_2}
\end{equation}
where $\{\bvarphi^{\bc}_{2,k}\}$ is the sequence of coefficient vectors. 
The coefficient vector $\bvarphi^{\bc}_{1,-1}$ in \eqref{eq:phic_Puiseux_1} satisfies the following. 
\begin{proposition} \label{pr:varphic_limit1}
Assume Type 1 and $\bar{\eta}_1'(\theta_2^*)< -1 < 1/\bar{\eta}_2'(\theta_1^*)$. Then, we have, for a positive vector $\bu_0^{\bc}$, 
\begin{equation}
\bvarphi^{\bc}_{1,-1}
= \lim_{\tilde{\Delta}_{e^{\theta_{\bc}^{max}}}\ni z\to e^{\theta_{\bc}^{max}}} (e^{\theta_{\bc}^{max}}-z)^{\frac{1}{2}} \bvarphi^{\bc}(z) 
= \bu_0^{\bc}.
\label{eq:varphic_limit1}
\end{equation}
\end{proposition}

This proposition will be proved in Section \ref{sec:varphic_coefficient}, where $\bu_0^{\bc}$ is specified. Since $\bvarphi^{\bc}_{1,-1}$ is positive, we obtain, by Lemma \ref{le:asympto_formula_ss}, 
\[
h_{\bc}(k) = k^{-\frac{1}{2}} e^{-\theta_{\bc}^{max} k}.
\]
This completes the former half of the proof of Proposition \ref{pr:power_term}. 
The coefficient vector $\bvarphi^{\bc}_{2,-2}$ in \eqref{eq:phic_Puiseux_2} satisfies the following. 
\begin{proposition} \label{pr:varphic_limit2}
Assume Type 1. Then, we have, for some positive vectors $\bu_1^{\bc}$ and $\bu_2^{\bc}$, 
\begin{equation}
\bvarphi^{\bc}_{2,-2}
= \lim_{\tilde{\Delta}_{e^{\theta_{\bc}^{max}}}\ni z\to e^{\theta_{\bc}^{max}}} (e^{\theta_{\bc}^{max}}-z) \bvarphi^{\bc}(z) 
= \left\{ \begin{array}{ll}
\bu_1^{\bc} & \mbox{if $\bar{\eta}_1'(\theta_2^*)=-1$ and $\bar{\eta}_2'(\theta_1^*)<-1$}, \cr 
\bu_2^{\bc} & \mbox{if $\bar{\eta}_1'(\theta_2^*)<-1$ and $\bar{\eta}_2'(\theta_1^*)=-1$}, \cr
\bu_1^{\bc}+\bu_2^{\bc} & \mbox{if $\bar{\eta}_1'(\theta_2^*)=\bar{\eta}_2'(\theta_1^*)=-1$}. 
\end{array} \right.
\label{eq:varphic_limit2}
\end{equation}
\end{proposition}
This proposition will be proved in Section \ref{sec:phic_expansion}. Since $\bvarphi^{\bc}_{2,-2}$ is positive, we obtain, by Lemma \ref{le:asympto_formula_ss},
\[
h_{\bc}(k) = e^{-\theta_{\bc}^{max} k}.
\]
This completes the latter half of the proof of Proposition \ref{pr:power_term}. 

\begin{remark}
By Propositions \ref{pr:varphic_limit1} and \ref{pr:varphic_limit2}, the order of the pole of the vector function $\tilde{\bvarphi}^{\bc}(\zeta)$ in Proposition \ref{pr:varphic_expansion} at the point $\zeta=0$ is just one if $\bar{\eta}_1'(\theta_2^*)< -1 < 1/\bar{\eta}_2'(\theta_1^*)$; it is just two if $\bar{\eta}_1'(\theta_2^*)= -1$ or $\bar{\eta}_2'(\theta_1^*)=-1$.
\end{remark}

%
%
\subsection{Proof of Propositions \ref{pr:varphic_analytic_domain}, \ref{pr:varphic_expansion} and \ref{pr:varphic_limit2}} \label{sec:phic_expansion}

The direction vector $\bc$ is set as $\bc=(1,1)$. Notation follows \cite{Ozawa22}. 

%
Denote by $\Phi^{\{1,2\}}=(\Phi^{\{1,2\}}_{\bx,\bx'};\bx,\bx'\in\mathbb{Z}^2)$ the fundamental matrix of $P^{\{1,2\}}$, i.e., $\Phi^{\{1,2\}}=\sum_{n=0}^\infty (P^{\{1,2\}})^n$, where $P^{\{1,2\}}=(P^{\{1,2\}}_{\bx,\bx'};\bx,\bx'\in\mathbb{Z}^2)$ is the transition probability matrix of the induced MA-process $\{\bY^{\{1,2\}}_n\}$. For $\bx\in\mathbb{Z}^2$, define the matrix generating function of the blocks of $\Phi^{\{1,2\}}$ in direction $\bc$, $\Phi^{\bc}_{\bx,*}(z)$, as
\[
\Phi^{\bc}_{\bx,*}(z)
=  \sum_{k=-\infty}^\infty z^k \Phi^{\{1,2\}}_{\bx,k\bc}.
\]
According to equation (3.3) of \cite{Ozawa22}, we divide $\bvarphi^{\bc}(z)$ into three parts as follows: 
\begin{equation}
\bvarphi^{\bc}(z) = \bvarphi^{\bc}_0(z) + \bvarphi^{\bc}_1(z) + \bvarphi^{\bc}_2(z), 
\label{eq:varphic_eq1}
\end{equation}
where
\begin{align}
&\bvarphi^{\bc}_0(z) = \sum_{i_1,i_2\in\{-1,0,1\}} \bnu_{(0,0)} (A^\emptyset_{i_1,i_2}-A^{\{1,2\}}_{i_1,i_2}) \Phi^{\bc}_{(i_1,i_2),*}(z), 
\label{eq:phic0_definition}\\
&\bvarphi^{\bc}_1(z) = \sum_{k=1}^\infty\ \sum_{i_1,i_2\in\{-1,0,1\}} \bnu_{(k,0)} (A^{\{1\}}_{i_1,i_2}-A^{\{1,2\}}_{i_1,i_2}) \Phi^{\bc}_{(k+i_1,i_2),*}(z), \\
&\bvarphi^{\bc}_2(z) = \sum_{k=1}^\infty\ \sum_{i_1,i_2\in\{-1,0,1\}} \bnu_{(0,k)} (A^{\{2\}}_{i_1,i_2}-A^{\{1,2\}}_{i_1,i_2}) \Phi^{\bc}_{(i_1,k+i_2),*}(z). 
\label{eq:phic2_def}
\end{align}

%
According to \cite{Ozawa22}, we focus on $\bvarphi^{\bc}_2(z)$ and consider another skip-free MA-process generated from $\{\bY^{\{1,2\}}_n\}$. The MA-process is $\{\hat{\bY}_n\}=\{(\hat{\bX}_n,\hat{\bJ}_n)\}=\{(\hat{X}_{1,n},\hat{X}_{2,n}),(\hat{R}_n,\hat{J}_n)\}$, where $\hat{X}_{1,n}=X^{\{1,2\}}_{1,n}$, $\hat{X}_{2,n}$ and $\hat{R}_n$ are the quotient and remainder of $X^{\{1,2\}}_{2,n}-X^{\{1,2\}}_{1,n}$ divided by $2$, respectively, and $\hat{J}_n=J^{\{1,2\}}_n$. 
The state space of $\{\hat{\bY}_n\}$ is $\mathbb{Z}^2\times\{0,1\}\times S_0$ and the additive part $\{\hat{\bX}_n\}$ is skip free. From the definition, if $\hat{\bX}_n=(x_1,x_2)$ and $\hat{R}_n=r$ in the new MA-process, it follows that $X^{\{1,2\}}_{1,n}=x_1$,  $X^{\{1,2\}}_{2,n}=x_1+2 x_2+r$ in the original MA-process. Hence, $\hat{\bY}_n=(k,0,0,j)$ means $\bY^{\{1,2\}}_n=(k,k,j)$.
Denote by $\hat{P}=(\hat{P}_{\bx,\bx'};\bx,\bx'\in\mathbb{Z}^2)$ the transition probability matrix of $\{\hat{\bY}_n\}$, which is given as 
\[
\hat{P}_{\bx,\bx'} 
= \left\{ \begin{array}{ll} 
\hat{A}^{\{1,2\}}_{\bx'-\bx}, & \mbox{if $\bx'-\bx\in\{-1,0,1\}^2$}, \cr
O, & \mbox{otherwise}, 
\end{array} \right.
\]
where 
\begin{align*}
&\hat{A}^{\{1,2\}}_{-1,1} = \begin{pmatrix} A^{\{1,2\}}_{-1,1} & O \cr  A^{\{1,2\}}_{-1,0} & A^{\{1,2\}}_{-1,1} \end{pmatrix}, \quad
\hat{A}^{\{1,2\}}_{0,1} = \begin{pmatrix} O & O \cr A^{\{1,2\}}_{0,1} & O \end{pmatrix}, \quad
\hat{A}^{\{1,2\}}_{1,1} = \begin{pmatrix} O & O \cr O &O \end{pmatrix}, \cr
&\hat{A}^{\{1,2\}}_{-1,0} = \begin{pmatrix} A^{\{1,2\}}_{-1,-1} & A^{\{1,2\}}_{-1,0} \cr  O & A^{\{1,2\}}_{-1,-1} \end{pmatrix}, \quad
\hat{A}^{\{1,2\}}_{0,0} = \begin{pmatrix} A^{\{1,2\}}_{0,0} & A^{\{1,2\}}_{0,1} \cr A^{\{1,2\}}_{0,-1} & A^{\{1,2\}}_{0,0} \end{pmatrix}, \quad
\hat{A}^{\{1,2\}}_{1,0} = \begin{pmatrix} A^{\{1,2\}}_{1,1} & O \cr A^{\{1,2\}}_{1,0} & A^{\{1,2\}}_{1,1} \end{pmatrix}, \cr
&\hat{A}^{\{1,2\}}_{-1,-1} = \begin{pmatrix} O & O \cr  O & O \end{pmatrix}, \quad
\hat{A}^{\{1,2\}}_{0,-1} = \begin{pmatrix} O & A^{\{1,2\}}_{0,-1} \cr O & O \end{pmatrix}, \quad
\hat{A}^{\{1,2\}}_{1,-1} = \begin{pmatrix} A^{\{1,2\}}_{1,-1} & A^{\{1,2\}}_{1,0} \cr O & A^{\{1,2\}}_{1,-1} \end{pmatrix}.
\end{align*}
Denote by $\hat{\Phi}=(\hat{\Phi}_{\bx,\bx'};\bx,\bx'\in\mathbb{Z}^2)$ the fundamental matrix of $\hat{P}$, i.e., $\hat{\Phi}=\sum_{n=0}^\infty (\hat{P})^n$, and for $\bx=(x_1,x_2)\in\mathbb{Z}^2$, define a matrix generating function $\hat{\Phi}_{\bx,*}(z)$ as
\begin{equation}
\hat{\Phi}_{\bx,*}(z) 
= \sum_{k=-\infty}^\infty z^k \hat{\Phi}_{\bx,(k,0)}
= \begin{pmatrix} 
\Phi^{\bc}_{(x_1,x_1+2 x_2),*}(z) &  \Phi^{\bc}_{(x_1,x_1+2 x_2-1),*}(z) \cr 
\Phi^{\bc}_{(x_1,x_1+2 x_2+1),*}(z) &  \Phi^{\bc}_{(x_1,x_1+2 x_2),*}(z)
\end{pmatrix}.
\label{eq:hatPhixs_def}
\end{equation}
%
Define blocks $\hat{A}^{\{2\}}_{i_1,i_2},\,i_1,i_2\in\{-1,0,1\},$ as $\hat{A}^{\{2\}}_{-1,1} = \hat{A}^{\{2\}}_{-1,0} = \hat{A}^{\{2\}}_{-1,-1} = O$ and 
\begin{align*}
\hat{A}^{\{2\}}_{0,1} = \begin{pmatrix} O & O \cr A^{\{2\}}_{0,1} & O \end{pmatrix}, \quad
\hat{A}^{\{2\}}_{0,0} = \begin{pmatrix} A^{\{2\}}_{0,0} & A^{\{2\}}_{0,1} \cr A^{\{2\}}_{0,-1} & A^{\{2\}}_{0,0} \end{pmatrix}, \quad
\hat{A}^{\{2\}}_{0,-1} = \begin{pmatrix} O & A^{\{2\}}_{0,-1} \cr O & O \end{pmatrix}, \cr
\hat{A}^{\{2\}}_{1,1} = \begin{pmatrix} O & O \cr O &O \end{pmatrix}, \quad 
\hat{A}^{\{2\}}_{1,0} = \begin{pmatrix} A^{\{2\}}_{1,1} & O \cr A^{\{2\}}_{1,0} & A^{\{2\}}_{1,1} \end{pmatrix}, \quad 
\hat{A}^{\{2\}}_{1,-1} = \begin{pmatrix} A^{\{2\}}_{1,-1} & A^{\{2\}}_{1,0} \cr O & A^{\{2\}}_{1,-1} \end{pmatrix}.
\end{align*}
For $i_1,i_2\in\{-1,0,1\}$, define the following matrix generating functions:
\begin{align*}
&\hat{A}^{\{1,2\}}_{*,i_2}(z)=\sum_{i\in\{-1,0,1\}} z^i \hat{A}^{\{1,2\}}_{i,i_2},\quad
\hat{A}^{\{1,2\}}_{i_1,*}(z)=\sum_{i\in\{-1,0,1\}} z^i \hat{A}^{\{1,2\}}_{i_1,i},\\
&\hat{A}^{\{2\}}_{*,i_2}(z)=\sum_{i\in\{0,1\}} z^i \hat{A}^{\{2\}}_{i,i_2}, \quad
\hat{A}^{\{2\}}_{i_1,*}(z)=\sum_{i\in\{-1,0,1\}} z^i \hat{A}^{\{2\}}_{i_1,i}.
\end{align*}
Define a vector function $\hat{\bvarphi}_2(z) $ as
\begin{align}
\hat{\bvarphi}_2(z) 
&=\begin{pmatrix} \hat{\bvarphi}_{2,1}(z) & \hat{\bvarphi}_{2,2}(z) \end{pmatrix}
= \sum_{k=1}^\infty\ \sum_{i_1,i_2\in\{-1,0,1\}} \hat{\bnu}_{(0,k)} (\hat{A}^{\{2\}}_{i_1,i_2}-\hat{A}^{\{1,2\}}_{i_1,i_2}) \hat{\Phi}_{(i_1,k+i_2),*}(z), 
\label{eq:hatphic2_def}
\end{align}
where, for $\bx=(x_1,x_2)\in\mathbb{Z}_+^2$, $\hat{\bnu}_{\bx}=\begin{pmatrix} \bnu_{(x_1,x_1+2x_2)} & \bnu_{(x_1,x_1+2x_2+1)} \end{pmatrix}$ and hence, for $k\ge 0$, 
\[
\hat{\bnu}_{(0,k)}=\begin{pmatrix} \bnu_{(0,2 k)} & \bnu_{(0,2 k+1)} \end{pmatrix}. 
\]
By equation (3.9) of \cite{Ozawa22}, $\bvarphi^{\bc}_2(z)$ is represented as 
\begin{equation}
\bvarphi^{\bc}_2(z) 
= \hat{\bvarphi}_{2,1}(z) + \sum_{i_1,i_2\in\{-1,0,1\}} \bnu_{(0,1)} (A^{\{2\}}_{i_1,i_2}-A^{\{1,2\}}_{i_1,i_2}) \Phi^{\bc}_{(i_1,i_2+1),*}(z). 
\label{eq:phic2_hatphic2_relation}
\end{equation}
We, therefore, consider analytic properties of the vector function $\bvarphi^{\bc}_2(z)$ through $\hat\bvarphi_2(z)$ and $\hat{\Phi}_{\bx,*}(z)$. 

%
Let $\hat{G}_{0,*}(z)$ be the G-matrix function generated from the triplet $\{\hat{A}^{\{1,2\}}_{*,-1}(z),\hat{A}^{\{1,2\}}_{*,0}(z),\hat{A}^{\{1,2\}}_{*,1}(z)\}$. By equations (3.11) and (3.13)  of \cite{Ozawa22}, we have, for $x_2\ge 0$, 
\begin{equation}
\hat{\Phi}_{(x_1,x_2),*}(z) = z^{x_1} \hat{G}_{0,*}(z)^{x_2} \hat{\Phi}_{(0,0),*}(z), 
\label{eq:hatPhi_expression}
\end{equation}
and this leads us to 
\begin{align}
&\hat{\bvarphi}_2(z) = \sum_{k=1}^\infty\ \sum_{i_2\in\{-1,0,1\}} \hat{\bnu}_{(0,k)} (\hat{A}^{\{2\}}_{*,i_2}(z)-\hat{A}^{\{1,2\}}_{*,i_2}(z))\, \hat{G}_{0,*}(z)^{k+i_2} \hat{\Phi}_{(0,0),*}(z).
\label{eq:hatphic2_eq1}
\end{align}
Hence, analytic properties of the vector function $\hat{\bvarphi}_2(z)$ as well as the matrix function $\hat{\Phi}_{\bx,*}(z)$ can be clarified  through $\hat{G}_{0,*}(z)$ and $\hat{\Phi}_{(0,0),*}(z)$.

%
By \eqref{eq:hatphic2_eq1}, $\hat{\bvarphi}_2(z)$ is represented as 
\begin{equation}
\hat{\bvarphi}_2(z) = \hat{\ba}(z,\hat{G}_{0,*}(z)) \hat{\Phi}_{(0,0),*}(z), 
\label{eq:hatphic2_eq2}
\end{equation}
where
\begin{align*}
&\hat{\ba}(z,w) = \sum_{k=1}^\infty\ \hat{\bnu}_{(0,k)} \hat{D}(z,\hat{G}_{0,*}(z)) w^{k-1}, \\
&\hat{D}(z,w) = \hat A^{\{2\}}_{*,-1}(z)+\hat A^{\{2\}}_{*,0}(z) w+\hat A^{\{2\}}_{*,1}(z) w^2- I w.
\end{align*}
First, we consider $\hat{\Phi}_{(0,0),*}(z)$. 
%
Let $\hat{G}^r_{0,*}(z)$ be the G-matrix function in the reverse direction generated from the triplet $\{\hat{A}^{\{1,2\}}_{*,-1}(z),\hat{A}^{\{1,2\}}_{*,0}(z),\hat{A}^{\{1,2\}}_{*,1}(z)\}$, which means that $\hat{G}^r_{0,*}(z)$ is the G-matrix function generated from the triplet by exchanging $\hat{A}^{\{1,2\}}_{*,-1}(z)$ and $\hat{A}^{\{1,2\}}_{*,1}(z)$; see Section \ref{sec:Gmatrix_reverse}. Define a matrix function $\hat{U}(z)$ as
\begin{equation}
\hat{U}(z) = \hat{A}^{\{1,2\}}_{*,-1}(z)\hat{G}^r_{0,*}(z) + \hat{A}^{\{1,2\}}_{*,0}(z) + \hat{A}^{\{1,2\}}_{*,1}(z)\hat{G}_{0,*}(z). 
\label{eq:hatU_definition}
\end{equation}
Then, $\hat\Phi_{(0,0),*}(z)$ is given as
\begin{equation}
\hat\Phi_{(0,0),*}(z) = \sum_{n=0}^\infty \hat{U}(z)^n = (I- \hat{U}(z))^{-1}=\frac{\adj(I- \hat{U}(z))}{\det(I- \hat{U}(z))}. 
\label{eq:hatPhi00s_hatU_relation}
\end{equation}
For $\theta\in[\theta_{\bc}^{min},\theta_{\bc}^{max} ]$, let $(\eta^R_{\bc,1}(\theta),\eta^R_{\bc,2}(\theta))$ and $(\eta^L_{\bc,1}(\theta), \eta^L_{\bc,2}(\theta))$ be the two real roots of the simultaneous equations: 
\begin{equation}
\spr(A^{\{1,2\}}_{*,*}(e^{\theta_1},e^{\theta_2}))=1,\quad \theta_1+\theta_2=\theta, 
\label{eq:simuleq_11}
\end{equation}
counting multiplicity, where $\eta^L_{\bc,1}(\theta)\le \eta^R_{\bc,1}(\theta)$ and $\eta^L_{\bc,2}(\theta)\ge \eta^R_{\bc,2}(\theta)$. Note that  $\eta^L_{\bc,1}(\theta^{max}_{\bc})=\eta^R_{\bc,1}(\theta^{max}_{\bc})$ and $\eta^L_{\bc,2}(\theta^{max}_{\bc})=\eta^R_{\bc,2}(\theta^{max}_{\bc})$. 
By equations (3.18) and (3.32) of \cite{Ozawa22}, we have 
\begin{equation}
\spr(\hat{G}_{0,*}(e^\theta))=e^{2\eta^R_{\bc,2}(\theta)}. 
\end{equation}
Since the eigenvalues of $\hat{G}^r_{0,*}(z)$ are coincide with those of the rate matrix function generated from the same triplet $\{\hat{A}^{\{1,2\}}_{*,-1}(z),\hat{A}^{\{1,2\}}_{*,0}(z),\hat{A}^{\{1,2\}}_{*,1}(z)\}$, we have 
\begin{equation}
\spr(\hat{G}^r_{0,*}(e^\theta))=e^{-2\eta_{\bc,2}^L(\theta)}. 
\end{equation}
By Lemmas \ref{le:G_analytic1} and \ref{le:G_decomposition} and Corollary \ref{co:G_analytic_boundary}, $\hat{G}_{0,*}(z)$ and $\hat{G}^r_{0,*}(z)$ satisfy the following properties. 
\begin{proposition} \label{pr:G_Gr_analyticproperties}
\begin{itemize}
\item[(1)] The extended G-matrix functions $\hat{G}_{0,*}(z)$ and $\hat{G}^r_{0,*}(z)$ are entry-wise analytic in $\Delta_{e^{\theta_{\bc}^{min}},e^{\theta_{\bc}^{max}}}\cup\partial\Delta_{e^{\theta_{\bc}^{max}}}\setminus\{e^{\theta_{\bc}^{max}}\}$. The point $z=e^{\theta_{\bc}^{max}}$ is a common branch point of $\hat{G}_{0,*}(z)$ and $\hat{G}^r_{0,*}(z)$ with order one. 
\item[(2)] There exist matrix functions $\tilde{G}_{0,*}(\zeta)$ and $\tilde{G}^r_{0,*}(\zeta)$ being analytic in a neighborhood of $\zeta=0$ and satisfying $\hat{G}_{0,*}(z)=\tilde{G}_{0,*}((e^{\theta_{\bc}^{max}}-z)^\frac{1}{2})$ and $\hat{G}^r_{0,*}(z)=\tilde{G}^r_{0,*}((e^{\theta_{\bc}^{max}}-z)^\frac{1}{2})$, respectively, in a neighborhood of $z=e^{\theta_{\bc}^{max}}$. 
\end{itemize}
\end{proposition}

%
In order to investigate singularity of $\hat\Phi_{(0,0),*}(z)$ at $z=e^{\theta_{\bc}^{max}}$, we give the following proposition.
\begin{proposition} \label{pr:spr_hatU_eq1}
The maximum eigenvalue of $\hat{U}(e^{\theta_{\bc}^{max}})$ is 1, and it is simple.
\end{proposition}
\begin{proof}
By equation (3.30) of \cite{Ozawa22}, we have $\spr(\hat{A}^{\{1,2\}}_{*,*}(e^{\theta_{\bc}^{max}},e^{2\eta_{\bc,2}^R(\theta_{\bc}^{max})}))=1$. Let $\bv$ be the right eigenvector of $\hat{A}^{\{1,2\}}_{*,*}(e^{\theta_{\bc}^{max}},e^{2\eta_{\bc,2}^R(\theta_{\bc}^{max})})$ with respect to eigenvalue 1. Since $\spr(\hat{G}_{0,*}(e^{\theta_{\bc}^{max}}))=e^{2\eta_{\bc,2}^R(\theta_{\bc}^{max})}$ and $\spr(\hat{G}^r_{0,*}(e^{\theta_{\bc}^{max}}))=e^{-2\eta_{\bc,2}^L(\theta_{\bc}^{max})}=e^{-2\eta_{\bc,2}^R(\theta_{\bc}^{max})}$, we have, by Lemma \ref{le:G_Gr_v},  
\[
\hat{G}_{0,*}(e^{\theta_{\bc}^{max}})\bv=e^{2\eta_{\bc,2}^R(\theta_{\bc}^{max})} \bv,\quad 
\hat{G}^r_{0,*}(e^{\theta_{\bc}^{max}})\bv=e^{-2\eta_{\bc,2}^R(\theta_{\bc}^{max})} \bv. 
\]
Hence, 
\[
\hat{U}(e^{\theta_{\bc}^{max}}) \bv
= \hat{A}^{\{1,2\}}_{*,*}(e^{\theta_{\bc}^{max}},e^{2\eta_{\bc,2}^R(\theta_{\bc}^{max})}) \bv 
= 1.
\]
This means that the value of $1$ is an eigenvalue of $\hat{U}(e^{\theta_{\bc}^{max}})$, and we obtain $\spr(\hat{U}(e^{\theta_{\bc}^{max}}))\ge 1$. 

Suppose $\spr(\hat{U}(e^{\theta_{\bc}^{max}}))> 1$. In a manner similar to that used in the proof of Proposition 3.6 of \cite{Ozawa13}, we see that every entry of $\hat{G}_{0,*}(e^{\theta})$ and $\hat{G}^r_{0,*}(e^{\theta})$ is log-convex in $\theta\in\mathbb{R}$.
Hence, by \eqref{eq:hatU_definition}, every entry of $\hat{U}(e^\theta)$ is also log-convex, and $\spr(\hat{U}(e^\theta))$ is convex in $\theta\in\mathbb{R}$.
By this, there exists a positive number $\theta_0<\theta^{max}_{\bc}$ such that $\spr(\hat{U}(e^{\theta_0}))=1$, and $\hat\Phi_{(0,0),*}(z)$ diverges at $z=e^{\theta_0}<e^{\theta_{\bc}^{max}}$. This contradicts Proposition 3.1 of \cite{Ozawa22}, which asserts that $\hat\Phi_{(0,0),*}(z)$ absolutely converges in $z\in\Delta_{e^{\theta_{\bc}^{min}},e^{\theta_{\bc}^{max}}}$. 
Hence, $\spr(\hat{U}(e^{\theta_{\bc}^{max}}))\le 1$, and this implies that the maximum eigenvalue of $\hat{U}(e^{\theta_{\bc}^{max}})$ is $1$. 
Since $\hat{U}(e^{\theta_{\bc}^{max}})$ is irreducible, the eigenvalue of $1$ is simple. 
\end{proof}

%
Let $\hat{\lambda}^U(z)$ be the eigenvalue of $\hat{U}(z)$ satisfying $\hat{\lambda}^U(x)=\spr(\hat{U}(x))$ for $x\in[e^{\theta_{\bc}^{min}},e^{\theta_{\bc}^{max}}]$. Let $\hat{\bu}^{U}(z)$ and $\hat{\bv}^{U}(z)$ be the left and right eigenvectors of $\hat{U}(z)$ with respect to the eigenvalue $\hat{\lambda}^U(z)$, respectively, satisfying $\hat{\bu}^{U}(z)\hat{\bv}^{U}(z)=1$. 
Define a matrix function $\tilde{U}(\zeta)$ as
\[
\tilde{U}(\zeta) = \hat{A}^{\{1,2\}}_{*,-1}(e^{\theta_{\bc}^{max}}-\zeta^2)\tilde{G}^r_{0,*}(\zeta) + \hat{A}^{\{1,2\}}_{*,0}(e^{\theta_{\bc}^{max}}-\zeta^2) + \hat{A}^{\{1,2\}}_{*,1}(e^{\theta_{\bc}^{max}}-\zeta^2)\tilde{G}_{0,*}(\zeta). 
\]
By Proposition \ref{pr:G_Gr_analyticproperties}, $\tilde{U}(\zeta)$ is entry-wise analytic in a neighborhood of $\zeta=0$ and satisfies $\hat{U}(z)=\tilde{U}((e^{\theta_{\bc}^{max}}-z)^{\frac{1}{2}})$ in a neighborhood of $z=e^{\theta_{\bc}^{max}}$. 
Define a matrix function $\tilde\Phi_{(0,0),*}(\zeta)$ as
\begin{equation}
\tilde\Phi_{(0,0),*}(\zeta) = (I- \tilde{U}(\zeta))^{-1}=\frac{\adj(I- \tilde{U}(\zeta))}{\det(I- \tilde{U}(\zeta))}. 
\label{eq:tildePhi00s_def}
\end{equation}
$\hat{\Phi}_{(0,0),*}(z)$ and $\tilde\Phi_{(0,0),*}(\zeta)$ satisfy the following properties.
\begin{proposition} \label{pr:hatPhi00s_analyticproperties}
\begin{itemize}
\item[(1)] The matrix function $\hat{\Phi}_{(0,0),*}(z)$ is entry-wise analytic in $\Delta_{e^{\theta_{\bc}^{min}},e^{\theta_{\bc}^{max}}}\cup\partial\Delta_{e^{\theta_{\bc}^{max}}}\setminus\{e^{\theta_{\bc}^{max}}\}$.
\item[(2)] $\tilde\Phi_{(0,0),*}(\zeta)$ is entry-wise meromorphic in a neighborhood of $\zeta=0$, and the point $\zeta=0$ is a pole of $\tilde\Phi_{(0,0),*}(\zeta)$ with order one. $\hat{\Phi}_{(0,0),*}(z)$ is represented as $\hat{\Phi}_{(0,0),*}(z)=\tilde{\Phi}_{(0,0),*}((e^{\theta_{\bc}^{max}}-z)^{\frac{1}{2}})$ in a neighborhood of $z=e^{\theta_{\bc}^{max}}$. 
\item[(3)] $\hat{\Phi}_{(0,0),*}(z)$ satisfies
\begin{equation}
\lim_{\tilde{\Delta}_{e^{\theta_{\bc}^{max}}}\ni z\to e^{\theta_{\bc}^{max}}} (e^{\theta_{\bc}^{max}}-z)^{\frac{1}{2}}\hat{\Phi}_{(0,0),*}(z) 
= \hat{g}^\Phi \hat{\bv}^{U}(e^{\theta_{\bc}^{max}}) \hat{\bu}^{U}(e^{\theta_{\bc}^{max}}) 
> O, 
\label{eq:hatPhi00s_limit}
\end{equation}
where both $\hat{\bv}^{U}(e^{\theta_{\bc}^{max}})$ and $ \hat{\bu}^{U}(e^{\theta_{\bc}^{max}})$ are positive, 
\begin{align}
&\hat{g}^\Phi 
= -\left( \hat{\bu}^{U}(e^{\theta_{\bc}^{max}}) \bigl( \hat{A}^{\{1,2\}}_{*,-1}(e^{\theta_{\bc}^{max}})\hat{G}^r_{0,*,1}+\hat{A}^{\{1,2\}}_{*,1}(e^{\theta_{\bc}^{max}})\hat{G}_{0,*,1} \bigr) \hat{\bv}^{U}(e^{\theta_{\bc}^{max}}) \right)^{-1}>0,
\label{eq:hatgPhi}
\end{align}
and $\hat{G}^r_{0,*,1}$ and $\hat{G}_{0,*,1}$ are the limits of $\hat{G}^r_{0,*}(z)$ and $\hat{G}_{0,*}(z)$, respectively, given by Lemma \ref{le:G_limit}.
\end{itemize}
\end{proposition}

\begin{proof}
By \eqref{eq:hatU_definition} and  Proposition \ref{pr:G_Gr_analyticproperties}, $\hat{U}(z)$ is entry-wise analytic in $\Delta_{e^{\theta_{\bc}^{min}},e^{\theta_{\bc}^{max}}}\cup\partial\Delta_{e^{\theta_{\bc}^{max}}}\setminus\{e^{\theta_{\bc}^{max}}\}$. Hence, by \eqref{eq:hatPhi00s_hatU_relation}, $\hat{\Phi}_{(0,0),*}(z)$ is entry-wise meromorphic in the same region. 
Recall that, under Assumption \ref{as:MAprocess_irreducible}, the induced MA-process $\{\bY^{\{1,2\}}_n\}$ is irreducible and aperiodic. Hence, in a manner similar to that used in the proof of Proposition 5.2 of \cite{Ozawa18}, we obtain by Proposition \ref{pr:spr_hatU_eq1} that, for every $z\in\Delta_{e^{\theta_{\bc}^{min}},e^{\theta_{\bc}^{max}}}\cup\partial\Delta_{e^{\theta_{\bc}^{max}}}\setminus\{e^{\theta_{\bc}^{max}}\}$, 
\[
\spr(\hat{U}(z)) < \spr(\hat{U}(|z|)) < \spr(\hat{U}(e^{\theta_{\bc}^{max}}))=1,
\]
and this leads us to $\det(I-\hat{U}(z))\ne 0$. This completes the proof of statement (1). 

By \eqref{eq:tildePhi00s_def}, $\tilde\Phi_{(0,0),*}(\zeta)$ is entry-wise meromorphic in a neighborhood of $\zeta=0$. Since $\tilde{U}(0)=\hat{U}(e^{\theta_{\bc}^{max}})$, we see by Proposition \ref{pr:spr_hatU_eq1} that $\det(I-\tilde{U}(0))=0$ and the multiplicity of zero of $\det(I-\tilde{U}(\zeta))$ at $\zeta=0$ is one. Hence, the point $\zeta=0$ is a pole of $\tilde\Phi_{(0,0),*}(\zeta)$ with order one. This completes the proof of statement (2). 

Define a function $\hat f(\lambda,z)$ as
\[
\hat f(\lambda,z) = \det(\lambda I-\hat{U}(z)). 
\]
By Corollary 2 of Seneta \cite{Seneta06} and Proposition \ref{pr:spr_hatU_eq1} (also see Proposition 5.11 of \cite{Ozawa18}), 
\begin{equation}
\adj(I-\hat{U}(e^{\theta_{\bc}^{max}})) 
= \hat f_\lambda(1,e^{\theta_{\bc}^{max}}) \hat{\bv}^{U}(e^{\theta_{\bc}^{max}}) \hat{\bu}^{U}(e^{\theta_{\bc}^{max}}), 
\end{equation}
where $\hat f_\lambda(\lambda,z)=\frac{\partial}{\partial \lambda} \hat f(\lambda,z)$ and both $\hat{\bv}^{U}(e^{\theta_{\bc}^{max}})$ and $\hat{\bu}^{U}(e^{\theta_{\bc}^{max}})$ are positive since $\hat{U}(e^{\theta_{\bc}^{max}})$ is irreducible.  
Furthermore, in a manner similar to that used in the proof of Proposition 5.9 of \cite{Ozawa18}, we obtain 
\begin{align}
&\lim_{\tilde{\Delta}_{e^{\theta_{\bc}^{max}}}\ni z\to e^{\theta_{\bc}^{max}}} (e^{\theta_{\bc}^{max}}-z)^{-\frac{1}{2}} \hat f(1,z)
= -c_0 \hat f_\lambda(1,e^{\theta_{\bc}^{max}}), 
\end{align}
where $c_0=\hat{\bu}^{U}(e^{\theta_{\bc}^{max}}) \bigl( \hat{A}^{\{1,2\}}_{*,-1}(e^{\theta_{\bc}^{max}})\hat{G}^r_{0,*,1}+\hat{A}^{\{1,2\}}_{*,1}(e^{\theta_{\bc}^{max}})\hat{G}_{0,*,1} \bigr)\hat{\bv}^{U}(e^{\theta_{\bc}^{max}})<0$ and, by Lemma \ref{le:G_limit}, both $\hat{G}_{0,*,1}$ and $\hat{G}^r_{0,*,1}$ are nonzero and nonpositive. 
By \eqref{eq:tildePhi00s_def}, this completes the proof of statement (3).
\end{proof}

%
By Proposition \ref{pr:hatPhi00s_analyticproperties}, the analytic properties of $\Phi^{\bc}_{\bx,*}(z)$ are obtained as follows. We will give the result corresponding to statement (3) of Proposition \ref{pr:hatPhi00s_analyticproperties} in Proposition \ref{pr:Phics_limit} of Section \ref{sec:varphic_coefficient}.
\begin{corollary} \label{co:Phics_analyticproperties}
Let $\bx=(x_1,x_2)$ be an arbitrary point in $\mathbb{Z}_+^2$. 
\begin{itemize}
\item[(1)] The matrix function $\Phi^{\bc}_{\bx,*}(z)$ is entry-wise analytic in $\Delta_{e^{\theta_{\bc}^{min}},e^{\theta_{\bc}^{max}}}\cup\partial\Delta_{e^{\theta_{\bc}^{max}}}\setminus\{e^{\theta_{\bc}^{max}}\}$.
\item[(2)]  There exists a matrix function $\tilde{\Phi}^{\bc}_{\bx,*}(\zeta)$ entry-wise meromorphic in a neighborhood of $\zeta=0$ such that the point $\zeta=0$ is a pole of $\tilde{\Phi}^{\bc}_{\bx,*}(\zeta)$ with order one. $\Phi^{\bc}_{\bx,*}(z)$ is represented as $\Phi^{\bc}_{\bx,*}(z)=\tilde{\Phi}^{\bc}_{\bx,*}((e^{\theta_{\bc}^{max}}-z)^{\frac{1}{2}})$ in a neighborhood of $z=e^{\theta_{\bc}^{max}}$. 
\end{itemize}
\end{corollary}

\begin{proof}
Applying Propositions \ref{pr:G_Gr_analyticproperties} and \ref{pr:hatPhi00s_analyticproperties} to \eqref{eq:hatPhi_expression}, we see that, for every $\bx\in\mathbb{Z}_+^2$, $\hat{\Phi}_{\bx,*}(z)$ is entry-wise analytic in $\Delta_{e^{\theta_{\bc}^{min}},e^{\theta_{\bc}^{max}}}\cup\partial\Delta_{e^{\theta_{\bc}^{max}}}\setminus\{e^{\theta_{\bc}^{max}}\}$. Hence, by \eqref{eq:hatPhixs_def}, for every $\bx\in\mathbb{Z}_+^2$, $\Phi^{\bc}_{\bx,*}(z)$ is entry-wise analytic in the same region. This completes the proof of statement (1). 

Let  $\tilde{G}_{0,*}(\zeta)$ and $\tilde{\Phi}_{(0,0),*}(\zeta)$ be the matrix functions given in Propositions \ref{pr:G_Gr_analyticproperties} and \ref{pr:hatPhi00s_analyticproperties}, respectively. Define $\tilde{\Phi}_{(x_1,x_2),*}(\zeta)$ as 
\[
\tilde{\Phi}_{(x_1,x_2),*}(\zeta) = (e^{\theta_{\bc}^{max}}-\zeta^2)^{x_1} \tilde{G}_{0,*}(\zeta)^{x_2} \tilde{\Phi}_{(0,0),*}(\zeta).
\]
As mentioned in the proof of Proposition \ref{pr:varphic_expansion}, for every $\bx\in\mathbb{Z}_+^2$, $\tilde{\Phi}_{\bx,*}(\zeta)$ is entry-wise meromorphic in a neighborhood of $\zeta=0$.
The point $\zeta=0$ is a pole of $\tilde{\Phi}_{\bx,*}(\zeta)$ with order one, and $\tilde{\Phi}_{\bx,*}(\zeta)$ satisfies $\hat{\Phi}_{\bx,*}(z)=\tilde{\Phi}_{\bx,*}((e^{\theta_{\bc}^{max}}-z)^{\frac{1}{2}})$ in a neighborhood of $z=e^{\theta_{\bc}^{max}}$. Hence, by \eqref{eq:hatPhixs_def}, for every $\bx\in\mathbb{Z}_+^2$, there exists a matrix function $\tilde{\Phi}^{\bc}_{\bx,*}(\zeta)$ satisfying statement (2). 
\end{proof}

%
Let $\alpha_{s_0}(z)$ be the eigenvalue of $\hat{G}_{0,*}(z)$ that satisfies,  for $\theta\in[\theta_{\bc}^{min},\theta_{\bc}^{max}]$, $\alpha_{s_0}(e^{\theta})=\spr(\hat{G}_{0,*}(e^{\theta}))=e^{2\eta_{\bc,2}^R(\theta)}$.  Let $\hat\bu^G(z)$ and $\hat\bv^G(z)$ be the left and right eigenvectors of $\hat{G}_{0,*}(z)$ with respect to the eigenvalue $\alpha_{s_0}(z)$, satisfying $\hat\bu^G(z)\hat\bv^G(z)=1$. 
By Lemma \ref{le:G_decomposition}, $\tilde{G}_{0,*}(\zeta)$ in Proposition \ref{pr:G_Gr_analyticproperties} satisfies the following property. 
\begin{proposition} \label{pr:tildeG_decomposition}
There exists a matrix function $\tilde{G}_{0,*}^\dagger(\zeta)$ entry-wise analytic in a neighborhood of $\zeta=0$ such that $\tilde{G}_{0,*}(\zeta)$ is represented as 
\begin{equation}
\tilde{G}_{0,*}(\zeta) = \tilde{G}^\dagger_{0,*}(\zeta) + \tilde\alpha_{s_0}(\zeta) \tilde\bv^G(\zeta) \tilde\bu^G(\zeta), 
\label{eq:tildeG0s_decomposition}
\end{equation}
where function $\tilde\alpha_{s_0}(\zeta)$, row vector function $\tilde\bu^G(\zeta)$ and column vector $\tilde\bv^G(\zeta)$ are element-wise analytic in a neighborhood of $\zeta=0$ and satisfying  $\alpha_{s_0}(z)=\tilde\alpha_{s_0}((e^{\theta_{\bc}^{max}}-z)^{\frac{1}{2}})$, $\hat\bu^G(z)=\tilde\bu^G((e^{\theta_{\bc}^{max}}-z)^{\frac{1}{2}})$ and  $\hat\bv^G(z)=\tilde\bv^G((e^{\theta_{\bc}^{max}}-z)^{\frac{1}{2}})$, respectively, in a neighborhood of $z=e^{\theta_{\bc}^{max}}$.
In a neighborhood of $\zeta=0$, $\tilde{G}^\dagger_{0,*}(\zeta)$ satisfies $\spr(\tilde{G}^\dagger_{0,*}(\zeta))<\alpha_{s_0}(e^{\theta_{\bc}^{max}})=e^{2\eta^R_{\bc,2}(\theta_{\bc}^{max})}$. 
Furthermore, $\tilde{G}_{0,*}(\zeta)$ satisfies, for $n\ge 1$, 
\begin{equation}
\tilde{G}_{0,*}(\zeta)^n = \tilde{G}^\dagger_{0,*}(\zeta)^n + \tilde\alpha_{s_0}(\zeta)^n \tilde\bv^G(\zeta) \tilde\bu^G(\zeta). 
\label{eq:tildeG0s_decomposition_n}
\end{equation}
\end{proposition}

%
Let $\hat{\bnu}_{(0,*)}(z)$ be the vector generating function of $\{\hat{\bnu}_{(0,k)}\}$ defined as $\hat{\bnu}_{(0,*)}(z)=\sum_{k=1}^\infty  z^k \hat{\bnu}_{(0,k)}$. 
Define a matrix function $\hat{U}_2(z)$ as
\[
\hat{U}_2(z) = \hat{A}^{\{2\}}_{0,*}(z)+\hat{A}^{\{2\}}_{1,*}(z) \hat{G}_{*,0}(z), 
\]
and let $\hat\bu_2^U(z)$ and $\hat\bv_2^U(z)$ be the left and right eigenvectors of $\hat{U}_2(z)$ with respect to the maximum eigenvalue of $\hat{U}_2(z)$, satisfying $\hat\bu_2^U(z)\hat\bv_2^U(z)=1$. 
By Lemma 5.3 of \cite{Ozawa18} (also see Proposition 3.5 of \cite{Ozawa22}), $\hat{\bnu}_{(0,*)}(z)$ satisfies the following properties. 
\begin{proposition} \label{pr:hatvarphi2_analyticproperties}
\begin{itemize}
\item[(1)] The vector function $\hat{\bnu}_{(0,*)}(z)$ is element-wise analytic in $\Delta_{e^{2\theta_2^*}}\cup\partial\Delta_{e^{2\theta_2^*}}\setminus\{e^{2\theta_2^*}\}$. 
\item[(2)] If $\theta_2^*<\theta_2^{max}$,  $\hat{\bnu}_{(0,*)}(z)$ is element-wise meromorphic in a neighborhood of $z=e^{2\theta_2^*}$ and the point  $z=e^{2\theta_2^*}$ is a pole of $\hat{\bnu}_{(0,*)}(z)$ with order one. It satisfies, for some positive constant $\hat g_2$, 
\begin{equation}
\lim_{\tilde{\Delta}_{e^{2\theta_2^*}}\ni z\to e^{2\theta_2^*}} (e^{2\theta_2^*}-z) \hat\bvarphi_2(z) 
= \hat g_2 \hat\bu_2^U(e^{2\theta_2^*}), 
\label{eq:hatvarphi12_limit}
\end{equation}
where $\hat\bu_2^U(e^{2\theta_2^*})$ is positive. 
\end{itemize}
\end{proposition}

%
Define a vector function $\tilde{\ba}(\zeta,w)$ as 
\begin{equation}
\tilde{\ba}(\zeta,w) = \sum_{k=1}^\infty\ \hat{\bnu}_{(0,k)} \hat{D}(e^{\theta_{\bc}^{max}}-\zeta^2,\tilde{G}_{0,*}(\zeta)) w^{k-1}.
\label{eq:tildea_def}
\end{equation}
Then, the vector functions $\hat{\ba}(z,\hat{G}_{0,*}(z))$ in \eqref{eq:hatphic2_eq2} and $\tilde{\ba}(\zeta,\tilde{G}_{0,*}(\zeta))$ satisfy the following properties. 
\begin{proposition} \label{pr:hata_analyticproperties} 
Assume Type 1. 
\begin{itemize}
\item[(1)] If $\bar{\eta}_1'(\theta_2^*)\le -c_1/c_2=-1$, the vector function $\hat{\ba}(z,\hat{G}_{0,*}(z))$ is element-wise analytic in $\Delta_{e^{\theta_{\bc}^{min}},e^{\theta_{\bc}^{max}}}\cup\partial\Delta_{e^{\theta_{\bc}^{max}}}\setminus\{e^{\theta_{\bc}^{max}}\}$. 
\item[(2)] If $\bar{\eta}_1'(\theta_2^*)<-1$, $\tilde{\ba}(\zeta,\tilde{G}_{0,*}(\zeta))$ is element-wise analytic in a neighborhood of $\zeta=0$; if $\bar{\eta}_1'(\theta_2^*)=-1$, it is element-wise meromorphic in a neighborhood of $\zeta=0$ and the point $\zeta=0$ is a pole of it with order one.  The vector function $\hat{\ba}(z,\hat{G}_{0,*}(z))$ is represented as $\hat{\ba}(z,\tilde{G}_{0,*}(z))=\tilde{\ba}((e^{\theta_{\bc}^{max}}-z)^{\frac{1}{2}},\tilde{G}_{0,*}((e^{\theta_{\bc}^{max}}-z)^{\frac{1}{2}}))$ in a neighborhood of $z=e^{\theta_{\bc}^{max}}$. 
\item[(3)] If $\bar{\eta}_1'(\theta_2^*)=-1$, $\hat{\ba}(z,\hat{G}_{0,*}(z))$ satisfies, for a positive constant $\hat{g}_2^a$, 
\begin{align}
&\lim_{\tilde{\Delta}_{e^{\theta_{\bc}^{max}}}\ni z\to e^{\theta_{\bc}^{max}}} (e^{\theta_{\bc}^{max}}-z)^{\frac{1}{2}} \hat{\ba}(z,\hat{G}_{0,*}(z)) = \hat{g}_2^a \hat\bu^G(e^{\theta_{\bc}^{max}}) \ge \bzero^\top, \ne\bzero^\top . 
\label{eq:hata_limit}
\end{align}
\end{itemize}
\end{proposition}

\begin{proof}
By Proposition 4.2 of \cite{Ozawa18}, if $\bar{\eta}_1'(\theta_2^*)\le -1$, we have for $z\in\Delta_{e^{\theta_{\bc}^{min}},e^{\theta_{\bc}^{max}}}\cup\partial\Delta_{e^{\theta_{\bc}^{max}}}\setminus\{e^{\theta_{\bc}^{max}}\}$ that $|\alpha_{s_0}(z)| < \alpha_{s_0}(e^{\theta_{\bc}^{max}})=e^{2\eta^R_{\bc,2}(\theta_{\bc}^{max})}\le e^{2\theta_2^*}$, and this implies $\spr(\hat{G}_{0,*}(z))<e^{2\theta_2^*}$. 
Hence, by Lemma 3.2 of \cite{Ozawa18} and Proposition \ref{pr:hatvarphi2_analyticproperties}, the vector function $\hat{\ba}(z,\hat{G}_{0,*}(z))$ is element-wise analytic in $\Delta_{e^{\theta_{\bc}^{min}},e^{\theta_{\bc}^{max}}}\cup\partial\Delta_{e^{\theta_{\bc}^{max}}}\setminus\{e^{\theta_{\bc}^{max}}\}$. 
This completes the proof of statement (1). 

By Proposition \ref{pr:tildeG_decomposition}, we have 
\begin{align}
\tilde{\ba}(\zeta,\tilde{G}_{0,*}(\zeta))
&= \tilde{\ba}(\zeta,\tilde{G}^\dagger_{0,*}(\zeta)) \cr
&\qquad + (\tilde\alpha_{s_0}(\zeta)^{-1} \hat{\bnu}_{(0,*)}(\tilde\alpha_{s_0}(\zeta))-\hat{\bnu}_{(0,1)})  \hat{D}(e^{\theta_{\bc}^{max}}-\zeta^2,\tilde\alpha_{s_0}(\zeta) ) \tilde\bv^G(\zeta) \tilde\bu^G(\zeta). 
\label{eq:tildea_decomposition}
\end{align}
If $\bar{\eta}_1'(\theta_2^*)\le -1$, $\spr(\tilde{G}^\dagger_{0,*}(\zeta))<e^{2\eta^R_{\bc,2}(\theta_{\bc}^{max})}\le e^{2\theta_2^*}$ in a neighborhood of $\zeta=0$. Hence, the vector function $\tilde{\ba}(\zeta,\tilde{G}^\dagger_{0,*}(\zeta))$ is element-wise analytic in a neighborhood of $\zeta=0$. 
If $\bar{\eta}_1'(\theta_2^*)<-1$, $\tilde\alpha_{s_0}(0)=\alpha_{s_0}(e^{\theta_{\bc}^{max}})=e^{2\eta^R_{\bc,2}(\theta_{\bc}^{max})}< e^{2\theta_2^*}$, and this implies $|\tilde\alpha_{s_0}(\zeta)|< e^{2\theta_2^*}$ in a neighborhood of $\zeta=0$. Hence, by Proposition \ref{pr:hatvarphi2_analyticproperties}, the vector function $\tilde{\ba}(\zeta,\tilde{G}_{0,*}(\zeta))$ as well as $\hat{\bnu}_{(0,*)}(\tilde\alpha_{s_0}(\zeta))$ is element-wise analytic in a neighborhood of $\zeta=0$. 
If $\bar{\eta}_1'(\theta_2^*)=-1$, $\tilde\alpha_{s_0}(0)=e^{2\eta^R_{\bc,2}(\theta_{\bc}^{max})}=e^{2\theta_2^*}$. Hence, by Proposition \ref{pr:hatvarphi2_analyticproperties}, the vector function $\tilde{\ba}(\zeta,\tilde{G}_{0,*}(\zeta))$ as well as $\hat{\bnu}_{(0,*)}(\tilde\alpha_{s_0}(\zeta))$ is meromorphic in a neighborhood of $\zeta=0$ and the point $\zeta=0$ is a pole of it with order one, where we use the fact that the limit given by statement (3) is nonzero. 
This completes the proof of statement (2). 

If $\bar{\eta}_1'(\theta_2^*)=-1$, $\alpha_{s_0}(e^{\theta_{\bc}^{max}})=e^{2\eta^R_{\bc,2}(\theta_{\bc}^{max})}=e^{2\theta_2^*}$. Hence, by Lemma \ref{le:alpha_s0_limit} and Proposition \ref{pr:hatvarphi2_analyticproperties}, we have 
\begin{align*}
&\lim_{\tilde{\Delta}_{e^{\theta_{\bc}^{max}}}\ni z\to e^{\theta_{\bc}^{max}}} (e^{\theta_{\bc}^{max}}-z)^{\frac{1}{2}} \hat{\bnu}_{(0,*)}(\alpha_{s_0}(z)) \cr
&\qquad\qquad = \lim_{\tilde{\Delta}_{e^{\theta_{\bc}^{max}}}\ni z\to e^{\theta_{\bc}^{max}}}  \frac{(e^{\theta_{\bc}^{max}}-z)^{\frac{1}{2}}}{\alpha_{s_0}(e^{\theta_{\bc}^{max}})-\alpha_{s_0}(z)} (\alpha_{s_0}(e^{\theta_{\bc}^{max}})-\alpha_{s_0}(z)) \hat{\bnu}_{(0,*)}(\alpha_{s_0}(z)) \cr
&\qquad\qquad = (-\alpha_{s_0,1})^{-1} \hat g_2 \hat\bu_2^U(e^{2\theta_2^*}), 
\end{align*}
where $\alpha_{s_0,1}$ is the limit of $\alpha_{s_0}(z)$ given by \eqref{eq:limit_eigen_G1} and it is negative. 
By \eqref{eq:tildea_decomposition}, this leads us to 
\begin{align}
&\lim_{\tilde{\Delta}_{e^{\theta_{\bc}^{max}}}\ni z\to e^{\theta_{\bc}^{max}}} (e^{\theta_{\bc}^{max}}-z)^{\frac{1}{2}} \hat{\ba}(z,\hat{G}_{0,*}(z)) \cr
&\qquad\qquad = (-\alpha_{s_0,1})^{-1} \hat g_2 e^{-2\theta_2^*} \hat\bu_2^U(e^{2\theta_2^*}) D(e^{\theta_{\bc}^{max}},e^{2\theta_2^*}) \hat\bv^G(e^{\theta_{\bc}^{max}}) \hat\bu^G(e^{\theta_{\bc}^{max}}). 
\label{eq:hata_limit1}
\end{align}
From this, we see that $\hat{g}_2^a \hat\bu^G(e^{\theta_{\bc}^{max}})$ in \eqref{eq:hata_limit} is given by the right-hand side of \eqref{eq:hata_limit1}. 
In a manner similar to that used in the proof of Lemma 5.5 (part (1)) of \cite{Ozawa18}, we see that 
\[
\hat\bu_2^U(e^{2\theta_2^*}) D(e^{\theta_{\bc}^{max}},e^{2\theta_2^*}) \hat\bv^G(e^{\theta_{\bc}^{max}})>0,
\]
and hence, $\hat{g}_2^a$ is also positive. 
This completes the proof of statement (3). 
\end{proof}

%
Finally, we give the proof of Propositions \ref{pr:varphic_analytic_domain}, \ref{pr:varphic_expansion} and \ref{pr:varphic_limit2}. 
%
\begin{proof}[Proof of Proposition \ref{pr:varphic_analytic_domain}]
Assume Type 1 and $\bar{\eta}_1'(\theta_2^*)\le -c_1/c_2=-1\le 1/\bar{\eta}_2'(\theta_1^*)$. Since $\bvarphi^{\bc}(z)$ is a probability vector generating function, it is automatically analytic element-wise in $\Delta_{e^{\theta_{\bc}^{max}}}$. 
Hence, we prove $\bvarphi^{\bc}(z)$ is element-wise analytic on $\partial\Delta_{e^{\theta_{\bc}^{max}}}\setminus\{e^{\theta_{\bc}^{max}}\}$. For the purpose, we use equations \eqref{eq:varphic_eq1},  \eqref{eq:phic0_definition}, \eqref{eq:phic2_hatphic2_relation}, \eqref{eq:hatPhi_expression} and \eqref{eq:hatphic2_eq2}.

By Propositions \ref{pr:G_Gr_analyticproperties}, \ref{pr:hatPhi00s_analyticproperties} and \ref{pr:hata_analyticproperties}, $\hat{G}_{0,*}(z)$, $\hat{\ba}(z,\hat{G}_{0,*}(z))$ and $\hat{\Phi}_{(0,0),*}(z)$ are element-wise analytic on $\partial\Delta_{e^{\theta_{\bc}^{max}}}\setminus\{e^{\theta_{\bc}^{max}}\}$. Hence, by \eqref{eq:hatPhi_expression} and \eqref{eq:hatphic2_eq2}, $\hat{\Phi}_{\bx,*}(z)$ and $\hat{\bvarphi}_2(z)$ are also analytic element-wise on $\partial\Delta_{e^{\theta_{\bc}^{max}}}\setminus\{e^{\theta_{\bc}^{max}}\}$. 
By Corollary \ref{co:Phics_analyticproperties}, $\Phi^{\bc}_{\bx,*}(z)$ is entry-wise analytic on $\partial\Delta_{e^{\theta_{\bc}^{max}}}\setminus\{e^{\theta_{\bc}^{max}}\}$, and by \eqref{eq:phic2_hatphic2_relation}, $\bvarphi^{\bc}_2(z)$ is element-wise analytic on $\partial\Delta_{e^{\theta_{\bc}^{max}}}\setminus\{e^{\theta_{\bc}^{max}}\}$. 
In the same way, we can see that if $\bar{\eta}_2'(\theta_1^*)\le -1$, $\bvarphi^{\bc}_1(z)$ is element-wise analytic on $\partial\Delta_{e^{\theta_{\bc}^{max}}}\setminus\{e^{\theta_{\bc}^{max}}\}$. 
By \eqref{eq:phic0_definition}, the analytic property of $\Phi^{\bc}_{\bx,*}(z)$ implies that $\bvarphi^{\bc}_0(z)$ is element-wise analytic on $\partial\Delta_{e^{\theta_{\bc}^{max}}}\setminus\{e^{\theta_{\bc}^{max}}\}$. 
As a result, we see by \eqref{eq:varphic_eq1} that $\bvarphi^{\bc}(z)$ is element-wise analytic on $\partial\Delta_{e^{\theta_{\bc}^{max}}}\setminus\{e^{\theta_{\bc}^{max}}\}$. 
This completes the proof. 
\end{proof}

%
\begin{proof}[Proof of Proposition \ref{pr:varphic_expansion}]
Assuming Type 1 and $\bar{\eta}_1'(\theta_2^*)\le -c_1/c_2=-1\le 1/\bar{\eta}_2'(\theta_1^*)$, we also use equations \eqref{eq:varphic_eq1},  \eqref{eq:phic0_definition}, \eqref{eq:phic2_hatphic2_relation}, \eqref{eq:hatPhi_expression} and \eqref{eq:hatphic2_eq2}.

First, we consider about $\bvarphi^{\bc}_0(z)$. 
By Corollary \ref{co:Phics_analyticproperties}, for $\bx=(x_1,x_2)\in\mathbb{Z}_+^2$, there exists a matrix function $\tilde{\Phi}^{\bc}_{\bx,*}(\zeta)$ being entry-wise meromorphic in a neighborhood of $\zeta=0$ and satisfying $\Phi^{\bc}_{\bx,*}(z)=\tilde{\Phi}^{\bc}_{\bx,*}((e^{\theta_{\bc}^{max}}-z)^{\frac{1}{2}})$ in a neighborhood of $z=e^{\theta_{\bc}^{max}}$. The point $\zeta=0$ is a pole of $\tilde{\Phi}^{\bc}_{\bx,*}(\zeta)$ with order one. 
Define $\tilde{\bvarphi}^{\bc}_0(\zeta)$ as 
\[
\tilde{\bvarphi}^{\bc}_0(\zeta) = \sum_{i_1,i_2\in\{-1,0,1\}} \bnu_{(0,0)} (A^\emptyset_{i_1,i_2}-A^{\{1,2\}}_{i_1,i_2}) \tilde{\Phi}^{\bc}_{(i_1,i_2),*}(\zeta),
\]
which satisfies the same analytic properties as $\tilde{\Phi}^{\bc}_{\bx,*}(\zeta)$. It also satisfies $\bvarphi^{\bc}_0(z)=\tilde{\bvarphi}^{\bc}_0((e^{\theta_{\bc}^{max}}-z)^{\frac{1}{2}})$ in a neighborhood of $z=e^{\theta_{\bc}^{max}}$.

Next, we consider about $\bvarphi^{\bc}_2(z)$.  Define $\tilde{\bvarphi}_2(\zeta)$ as
\[
\tilde{\bvarphi}_2(\zeta) = \tilde{\ba}(\zeta,\tilde{G}_{0,*}(\zeta)) \tilde{\Phi}_{(0,0),*}(\zeta).
\]
By Propositions \ref{pr:hatPhi00s_analyticproperties} and \ref{pr:hata_analyticproperties} and \eqref{eq:hatphic2_eq2}, $\tilde{\bvarphi}_2(\zeta)$ is entry-wise meromorphic in a neighborhood of $\zeta=0$ and satisfying $\hat{\bvarphi}_2(z)=\tilde{\bvarphi}_2((e^{\theta_{\bc}^{max}}-z)^{\frac{1}{2}})$ in a neighborhood of $z=e^{\theta_{\bc}^{max}}$. If $\bar{\eta}_1'(\theta_2^*)<-1$, the point $\zeta=0$ is a pole of $\tilde{\bvarphi}_2(\zeta)$ with at most order one; if $\bar{\eta}_1'(\theta_2^*)=-1$, it is a pole of $\tilde{\bvarphi}_2(\zeta)$ with at most order two. 
Represent $\tilde{\bvarphi}_2(\zeta)$ in block form as  $\tilde{\bvarphi}_2(\zeta)=\begin{pmatrix} \tilde{\bvarphi}_{2,1}(\zeta) & \tilde{\bvarphi}_{2,2}(\zeta) \end{pmatrix}$ and define $\tilde{\bvarphi}^{\bc}_2(\zeta)$ as 
\[
\tilde{\bvarphi}^{\bc}_2(\zeta) 
= \tilde{\bvarphi}_{2,1}(\zeta) + \sum_{i_1,i_2\in\{-1,0,1\}} \bnu_{(0,1)} (A^{\{2\}}_{i_1,i_2}-A^{\{1,2\}}_{i_1,i_2}) \tilde{\Phi}^{\bc}_{(i_1,i_2+1),*}(\zeta). 
\]
Then, the vector function $\tilde{\bvarphi}^{\bc}_2(\zeta)$ is element-wise meromorphic in a neighborhood of $\zeta=0$, and by \eqref{eq:phic2_hatphic2_relation}, it satisfies $\bvarphi^{\bc}_2(z)=\tilde{\bvarphi}^{\bc}_2((e^{\theta_{\bc}^{max}}-z)^{\frac{1}{2}})$ in a neighborhood of $z=e^{\theta_{\bc}^{max}}$. 
If $\bar{\eta}_1'(\theta_2^*)<-1$, the point $\zeta=0$ is a pole of $\tilde{\bvarphi}^{\bc}_2(\zeta)$ with at most order one; if $\bar{\eta}_1'(\theta_2^*)=-1$, it is a pole of $\tilde{\bvarphi}^{\bc}_2(\zeta)$ with at most order two. 

Finally, we consider about $\bvarphi^{\bc}(z)$. 
In the same way as that used for $\bvarphi^{\bc}_2(z)$, we can see that there exists a vector function $\tilde{\bvarphi}^{\bc}_1(\zeta)$ being element-wise meromorphic in a neighborhood of $\zeta=0$ and satisfying $\bvarphi^{\bc}_1(z)=\tilde{\bvarphi}^{\bc}_1((e^{\theta_{\bc}^{max}}-z)^{\frac{1}{2}})$ in a neighborhood of $z=e^{\theta_{\bc}^{max}}$. If $\bar{\eta}_2'(\theta_1^*)<-1$, the point $\zeta=0$ is a pole of $\tilde{\bvarphi}^{\bc}_1(\zeta)$ with at most order one; if $\bar{\eta}_2'(\theta_1^*)=-1$, it is a pole of $\tilde{\bvarphi}^{\bc}_1(\zeta)$ with at most order two. 
Define $\tilde{\bvarphi}^{\bc}(\zeta)$ as 
\[
\tilde{\bvarphi}^{\bc}(\zeta) = \tilde{\bvarphi}^{\bc}_0(\zeta) + \tilde{\bvarphi}^{\bc}_1(\zeta) + \tilde{\bvarphi}^{\bc}_2(\zeta). 
\]
Then, the vector function $\tilde{\bvarphi}^{\bc}(\zeta)$ is element-wise meromorphic in a neighborhood of $\zeta=0$, and by \eqref{eq:varphic_eq1}, it satisfies $\bvarphi^{\bc}(z)=\tilde{\bvarphi}^{\bc}((e^{\theta_{\bc}^{max}}-z)^{\frac{1}{2}})$ in a neighborhood of $z=e^{\theta_{\bc}^{max}}$. 
If $\bar{\eta}_1'(\theta_2^*)< -c_1/c_2=-1< 1/\bar{\eta}_2'(\theta_1^*)$, the point $\zeta=0$ is a pole of $\tilde{\bvarphi}^{\bc}(\zeta)$ with at most order one; if $\bar{\eta}_1'(\theta_2^*)=-1$ or $\bar{\eta}_2'(\theta_1^*)=-1$, it is a pole of $\tilde{\bvarphi}^{\bc}(\zeta)$ with at most order two. 
This completes the proof.
\end{proof}

%
\begin{proof}[Proof of Proposition \ref{pr:varphic_limit2}]
Assume Type 1.  
By Corollary \ref{co:Phics_analyticproperties} and \eqref{eq:phic0_definition}, 
\begin{equation}
\lim_{\tilde{\Delta}_{e^{\theta_{\bc}^{max}}}\ni z\to e^{\theta_{\bc}^{max}}} (e^{\theta_{\bc}^{max}}-z) \varphi_0^{\bc}(z) = \bzero^\top.
\label{eq:varphic0_limit_1}
\end{equation}
%
If $\bar{\eta}_1'(\theta_2^*)=-1$, by Propositions \ref{pr:hatPhi00s_analyticproperties} and \ref{pr:hata_analyticproperties} and equations \eqref{eq:phic2_hatphic2_relation} and \eqref{eq:varphic0_limit_1}, representing $\hat{\bu}^{U}(e^{\theta_{\bc}^{max}})$ in block form as $\hat{\bu}^{U}(e^{\theta_{\bc}^{max}})=\begin{pmatrix} \hat{\bu}^U_1(e^{\theta_{\bc}^{max}}) & \hat{\bu}^U_2(e^{\theta_{\bc}^{max}}) \end{pmatrix}$, we obtain
\begin{equation}
\lim_{\tilde{\Delta}_{e^{\theta_{\bc}^{max}}}\ni z\to e^{\theta_{\bc}^{max}}} (e^{\theta_{\bc}^{max}}-z) \bvarphi^{\bc}_2(z) 
=\bu_2^{\bc}
= \hat{g}_2^a \hat{g}^\Phi \hat\bu^G(e^{\theta_{\bc}^{max}}) \hat{\bv}^{U}(e^{\theta_{\bc}^{max}}) \hat{\bu}^U_1(e^{\theta_{\bc}^{max}}) > \bzero^\top, 
\label{eq:varphic2_limit_1}
\end{equation}
where $\hat\bu^G(e^{\theta_{\bc}^{max}})$ is nonzero and nonnegative and other terms on the right-hand side of the equation are positive; if $\bar{\eta}_1'(\theta_2^*)<-1$, we have 
\begin{equation}
\lim_{\tilde{\Delta}_{e^{\theta_{\bc}^{max}}}\ni z\to e^{\theta_{\bc}^{max}}} (e^{\theta_{\bc}^{max}}-z) \bvarphi^{\bc}_2(z) 
= \bzero^\top.
\label{eq:varphic2_limit_2}
\end{equation}
In a manner similar to that used for $\bvarphi^{\bc}_2(z)$, we see that if $\bar{\eta}_2'(\theta_1^*)=-1$, then for some positive vector $\bu_1^{\bc}$, 
\begin{equation}
\lim_{\tilde{\Delta}_{e^{\theta_{\bc}^{max}}}\ni z\to e^{\theta_{\bc}^{max}}} (e^{\theta_{\bc}^{max}}-z) \bvarphi^{\bc}_1(z) 
=\bu_1^{\bc}, 
\label{eq:varphic1_limit_1}
\end{equation}
and if $\bar{\eta}_1'(\theta_2^*)<-1$, 
\begin{equation}
\lim_{\tilde{\Delta}_{e^{\theta_{\bc}^{max}}}\ni z\to e^{\theta_{\bc}^{max}}} (e^{\theta_{\bc}^{max}}-z) \bvarphi^{\bc}_1(z) 
= \bzero^\top.
\label{eq:varphic1_limit_2}
\end{equation}
As a result, by \eqref{eq:varphic_eq1}, \eqref{eq:varphic2_limit_1},  \eqref{eq:varphic2_limit_2},  \eqref{eq:varphic1_limit_1} and  \eqref{eq:varphic1_limit_2},  we obtain \eqref{eq:varphic_limit2} in Proposition \ref{pr:varphic_limit2}.
\end{proof}

%
%
%
\subsection{Proof of Proposition \ref{pr:varphic_limit1}} \label{sec:varphic_coefficient}

In this subsection, we also set the direction vector $\bc=(1,1)$.

Recall that  $P=(P_{\bx,\bx'}; \bx,\bx'\in\mathbb{N}_+^2)$ and $\bnu=(\bnu_{\bx}, \bx\in\mathbb{N}_+^2)$ are the transition probability matrix and stationary distribution of the original 2d-QBD process $\{\bY_n\}$, respectively. Define vector generating functions of the stationary provabilities: 
\begin{align*}
&\bvarphi(z,w) = \sum_{i=0}^\infty \sum_{j=0}^\infty \bnu_{(i,j)} z^i w^j,\quad 
\bvarphi_+(z,w) = \sum_{i=1}^\infty \sum_{j=1}^\infty \bnu_{(i,j)} z^i w^j, \\
&\bvarphi_1(z) = \sum_{i=1}^\infty  \bnu_{(i,0)} z^i,\quad 
\bvarphi_2(w) = \sum_{j=1}^\infty  \bnu_{(0,j)} w^j,
\end{align*}
where 
\begin{equation}
\bvarphi(z,w) = \bvarphi_+(z,w) + \bvarphi_1(z) + \bvarphi_2(w) + \bnu_{(0,0)}.
\end{equation}
For $\alpha\in\scrI_2=\{\emptyset, \{1\}, \{2\}, \{1,2\} \}$, define the matrix generating functions of the transition probability blocks $A^\alpha_{i,j}$ as
\[
A^\alpha_{*,*}(z,w) = \sum_{i,j\in\{-1,0,1\}} A^\alpha_{i,j} z^i w^j,
\]
where $A^{\{1,2\}}_{*,*}(z,w)$ has already been defined in Section \ref{sec:mainresults}. 
%
\begin{figure}[t]
\begin{center}
\includegraphics[width=60mm,trim=0 0 0 0]{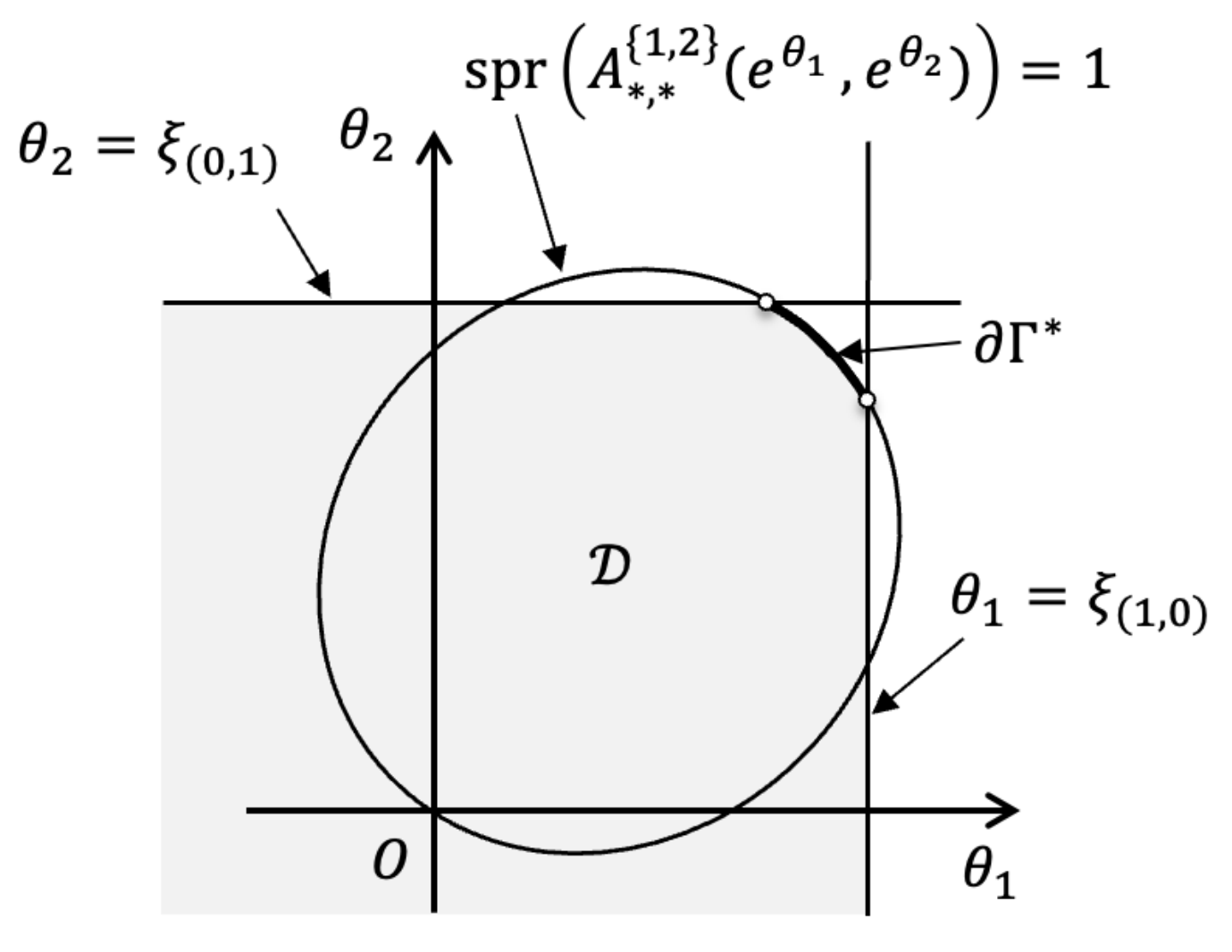} 
\caption{Domain $\calD$ of Type 1}
\label{fig:fig43}
\end{center}
\end{figure}
%
Define the following domains (see Fig.\ \ref{fig:fig43}): 
\begin{align*}
&\calD_{\bvarphi} = \mbox{the interior of $\{(\theta_1,\theta_2)\in\mathbb{R}^2: \bvarphi(e^{\theta_1},e^{\theta_2})<\infty$, element-wise\}}, \\
&\calD_0 =  \{(\theta_1,\theta_2)\in\mathbb{R}^2: \theta_1<\xi_{(1,0)},\ \theta_2<\xi_{(0,1)}  \}, \\
&\calD^{\{1,2\}} = \{(\theta_1,\theta_2)\in\mathbb{R}^2: \mbox{$(\theta_1,\theta_2)<(\theta_1',\theta_2')$ for some $(\theta_1',\theta_2')\in\Gamma^{\{1,2\}}$} \}, \\
&\Gamma = \Gamma^{\{1,2\}} \cap \Gamma^{\{1\}} \cap  \Gamma^{\{2\}} = \{(\theta_1,\theta_2)\in\Gamma^{\{1,2\}}: \theta_1<\theta_1^*,\ \theta_2<\theta_2^* \}, \\
&\calD = \{(\theta_1,\theta_2)\in\mathbb{R}^2: \mbox{$(\theta_1,\theta_2)<(\theta_1',\theta_2')$ for some $(\theta_1',\theta_2')\in\Gamma$} \}, 
\end{align*}
where $\calD_0$ is the intersection of the convergence domains of $\bvarphi_1(e^{\theta_1})$ and $\bvarphi_2(e^{\theta_2})$ and satisfies $\calD\subset \calD_0$. 
For a point set $S$, we denote by $\partial S$ the boundary of $S$. For example, $\partial \Gamma^{\{1,2\}}=\{(\theta_1,\theta_2)\in\mathbb{R}^2: A^{\{1,2\}}_{*,*}(e^{\theta_1},e^{\theta_2})=1\}$. 
Define a partial boundary of $\Gamma^{\{1,2\}}$, $\partial\Gamma^*$, as 
\[
\partial\Gamma^* = \partial\calD^{\{1,2\}} \cap \calD_0 =\partial\calD \cap \calD_0, 
\]
which is empty if the 2d-QBD process is not of Type 1 (see Fig.\ \ref{fig:fig43}). 
The convergence domain $\calD_{\bvarphi}$ is given as follows.
\begin{lemma} \label{le:varphi_domain}
$\calD_{\bvarphi}=\calD$.
\end{lemma}

\begin{proof}
By the definition, we have $\calD=\calD^{\{1,2\}}\cap\calD_0$ and $\calD_{\bvarphi}\subset\calD_0$. Furthermore, by Lemma 3.1 of \cite{Ozawa18}, we have $\calD\subset\calD_{\bvarphi}$. Hence, in order to prove the lemma, it suffices to prove $\calD_{\bvarphi}\subset\calD^{\{1,2\}}$, which implies $\calD_{\bvarphi}\subset\calD$.
For every $\bc'=(c_1',c_2')\in\mathbb{N}^2$, if $(\theta_1,\theta_2)\in\calD_{\bvarphi}$, then we have 
\[
\infty > \bvarphi(e^{\theta_1},e^{\theta_2}) \ge \bvarphi^{\bc'}(e^{c_1'\theta_1+c_2'\theta_2})
\]
and this implies 
\[
c_1'\theta_1+c_2'\theta_2 \le \theta_{\bc'}^{max} = \sup\{ c_1'\theta_1'+c_2'\theta_2': (\theta_1',\theta_2')\in\Gamma^{\{1,2\}} \}. 
\]
Hence, in a manner similar to that used in the proof of Theorem 5.2 of \cite{Ozawa21}, since $\bc'$ is arbitrary, we obtain $\calD_{\bvarphi}\subset\calD^{\{1,2\}}$. 
This completes the proof. 
\end{proof}

Note that the convergence domain of $\bvarphi_+(e^{\theta_1},e^{\theta_2})$ coincides with that of $\bvarphi(e^{\theta_1},e^{\theta_2})$ and it is given by $\calD$.  
%
We define complex domains $\Omega$ and $\Omega_0$ as 
\begin{align*}
&\Omega = \{(z,w)\in\mathbb{C}^2: (\log |z|,\log |w|)\in\calD\},\\
&\Omega_0 = \{(z,w)\in\mathbb{C}^2: (\log |z|,\log |w|)\in\calD_0\}, 
\end{align*}
where $\Omega\subset\Omega_0$. 
Define a vector function $\bg(z,w)$ as 
\begin{equation}
\bg(z,w) = \bvarphi_1(z)(A^{\{1\}}_{*,*}(z,w)-I) + \bvarphi_2(w)(A^{\{2\}}_{*,*}(z,w)-I) + \bnu_{(0.0)}(A^\emptyset_{*,*}(z,w)-I).
\end{equation}
Then, if $(z,w)\in\Omega$, series $\bvarphi(z,w)$ absolutely converges element-wise and we have, by the stationary equation $\bnu P=\bnu$ (see Sectin 3.1 of \cite{Ozawa18}), 
\begin{equation}
\bvarphi_+(z,w)(I-A^{\{1,2\}}_{*,*}(z,w)) = \bg(z,w).
\label{eq:varphiAp_identity}
\end{equation}
Using this equation, we analytically extend $\bvarphi_+(z,w)$ outer $\Omega$. From \eqref{eq:varphiAp_identity}, we obtain
\begin{equation}
\bvarphi_+(z,w) = \frac{\bg(z,w)\,\adj(I-A^{\{1,2\}}_{*,*}(z,w))}{\det(I-A^{\{1,2\}}_{*,*}(z,w))}.
\label{eq:varphiAp_expression}
\end{equation}
Both the numerator and denominator of the right hand side of \eqref{eq:varphiAp_expression} are analytic in $\Omega_0$. Hence, by the identity theorem, $\bvarphi_+(z,w)$ can analytically be extended over $\Omega_0$ and the extended $\bvarphi_+(z,w)$ is meromorphic in $\Omega_0$, where zeros of $\det(I-A^{\{1,2\}}_{*,*}(z,w))$ may be poles of $\bvarphi_+(z,w)$. 

%
For positive real numbers $x$ and $y$, denote by $\chi^A(x,y)$ the maximum eigenvalue (Perron-Frobenius eigenvalue) of $A^{\{1,2\}}_{*,*}(x,y)$ and by $\bu^A(x,y)$ and $\bv^A(x,y)$ the left and right eigenvectors with respect to the eigenvalue, respectively, satisfying $\bu^A(x,y) \bv^A(x,y) =1$. Since $A^{\{1,2\}}_{*,*}(x,y)$ is irreducible, the eigenvalue $\chi^A(x,y)$ is simple and the eigenvectors  $\bu^A(x,y)$ and $\bv^A(x,y)$ are positive. 
The vector function $\bg(z,w)$ satisfies the following property. 
\begin{proposition} \label{pr:gvA_positivity}
For every $(\theta_1,\theta_2)\in\partial\Gamma^*$, $\bg(e^{\theta_1},e^{\theta_2})\bv^A(e^{\theta_1},e^{\theta_2})>0$. 
\end{proposition}

\begin{proof}
Assume Type 1 and let $(\theta_1,\theta_2)$ be an arbitrary point on $\partial\Gamma^*$. 
Set $w=e^{\theta_2}$. Then, each element of $\bvarphi_+(z, e^{\theta_2})$ is a function of one variable, and it is meromorphic in $\Delta_{e^{\xi_{(1,0)}}}\setminus\{0\}$. Furthermore, the point $z=e^{\theta_1}$ is a pole of it since the convergence domain of $\bvarphi_+(e^{\theta_1'},e^{\theta_2'})$ is $\calD$ and $\partial\Gamma^*\subset\partial\calD$. 
Define a function $f(\lambda,z,w)$ as $f(\lambda,z,w)=\det(\lambda I-A^{\{1,2\}}_{*,*}(z,w))$, and $f_{\lambda}(\lambda,z,w)$ and $f_z(\lambda,z,w)$ as $f_{\lambda}(\lambda,z,w)=\frac{\partial}{\partial \lambda} f(\lambda,z,w)$ and $f_z(\lambda,z,w)=\frac{\partial}{\partial z} f(\lambda,z,w)$, respectively. 
For positive numbers $x$ and $y$, since $A^{\{1,2\}}_{*,*}(x,y)$ is irreducible, the eigenvalue $\chi^A(x,y)$ is simple and differentiable in $x$ and $y$. Define $\chi^A_x(x,y)$ as $\chi^A_x(x,y)=\frac{\partial}{\partial x} \chi^A(x,y)$. 
By Corollary 2 of \cite{Seneta06}, we have
\begin{equation}
\adj(I-A^{\{1,2\}}_{*,*}(e^{\theta_1},e^{\theta_2})) = f_{\lambda}(1,e^{\theta_1},e^{\theta_2}) \bv^A(e^{\theta_1},e^{\theta_2}) \bu^A(e^{\theta_1},e^{\theta_2}).
\end{equation}
In a manner similar to that used in the proof of Proposition 5.3 of \cite{Ozawa18}, we obtain
\begin{equation}
\lim_{x\uparrow e^{\theta_1}} (e^{\theta_1}-x)^{-1} f(1,x,e^{\theta_2}) = \chi^A_x(e^{\theta_1},e^{\theta_2}) f_{\lambda}(1,e^{\theta_1},e^{\theta_2}), 
\end{equation}
where $ \chi^A_x(e^{\theta_1},e^{\theta_2})$ is positive since $\chi^A(e^{\theta_1'},e^{\theta_2'})$ is convex in $(\theta_1',\theta_2')$ and, under Assumption \ref{as:2dQBD_stable}, 
\[
\chi^A(e^{\theta_1},e^{\theta_2})=1>\min_{(\theta_1',\theta_2')\in\Gamma^{\{1,2\}}} \chi^A(e^{\theta_1'},e^{\theta_2'}).
\] 
Hence, by \eqref{eq:varphiAp_expression},  we obtain, 
\begin{equation}
\lim_{x\uparrow e^{\theta_1}} (e^{\theta_1}-x) \bvarphi_+(x,e^{\theta_2}) 
= \chi^A_x(e^{\theta_1},e^{\theta_2})^{-1} \bg(e^{\theta_1},e^{\theta_2}) \bv^A(e^{\theta_1},e^{\theta_2}) \bu^A(e^{\theta_1},e^{\theta_2}), 
\end{equation}
where $\chi^A_x(e^{\theta_1},e^{\theta_2})$ and $\bu^A(e^{\theta_1},e^{\theta_2})$ are positive. 
Since the point $z=e^{\theta_1}$ is a zero of $f(1,z,e^{\theta_2})$ with order one, if $\bg(e^{\theta_1},e^{\theta_2}) \bv^A(e^{\theta_1},e^{\theta_2})=0$, the point $z=e^{\theta_1}$ becomes a removable singularity of each element of $\bvarphi_+(z, e^{\theta_2})$, and this contradicts that the point $z=e^{\theta_1}$ is a pole of each element of $\bvarphi_+(z, e^{\theta_2})$. 
Hence, we have $\bg(e^{\theta_1},e^{\theta_2}) \bv^A(e^{\theta_1},e^{\theta_2})>0$.
\end{proof}

%
Denote $\eta_{\bc,1}^{max}=\eta^R_{\bc,1}(\theta_{\bc}^{max})$ and $\eta_{\bc,2}^{max}=\eta^R_{\bc,2}(\theta_{\bc}^{max})$. Since $\bc=(1,1)$, they satisfy $\theta_{\bc}^{max}=\eta_{\bc,1}^{max}+\eta_{\bc,2}^{max}$.
The limit of $\Phi^{\bc}_{\bx,*}(z)$ is given as follows. 
\begin{proposition} \label{pr:Phics_limit} 
Let $\bx=(x_1,x_2)$ be an arbitrary point in $\mathbb{Z}_+^2$. Then, the matrix function $\Phi^{\bc}_{\bx,*}(z)$ satisfies, for a positive row vector $\bu_{\bc}$, 
\begin{align}
&\lim_{\tilde{\Delta}_{e^{\theta_{\bc}^{max}}}\ni z\to e^{\theta_{\bc}^{max}}} (e^{\theta_{\bc}^{max}}-z)^{\frac{1}{2}} \Phi^{\bc}_{(x_1,x_2),*}(z) \cr
&\qquad\qquad\qquad= \hat{g}^\Phi e^{x_1 \eta_{\bc,1}^{max}+x_2 \eta_{\bc,2}^{max}} \bv^A(e^{\eta_{\bc,1}^{max}},e^{\eta_{\bc,2}^{max}}) \bu_{\bc} 
> O, 
\label{eq:hatPhix1x2s_limit}
\end{align}
where the positive constant $\hat{g}^\Phi$ is given by \eqref{eq:hatgPhi}. 
\end{proposition}

\begin{proof}
By \eqref{eq:hatPhi_expression} and \eqref{eq:hatPhi00s_limit}, we obtain that, for every $\bx=(x_1,x_2)\in\mathbb{Z}_+^2$, 
\begin{equation}
\lim_{\tilde{\Delta}_{e^{\theta_{\bc}^{max}}}\ni z\to e^{\theta_{\bc}^{max}}} (e^{\theta_{\bc}^{max}}-z)^{\frac{1}{2}}\hat{\Phi}_{(x_1,x_2),*}(z) 
= \hat{g}^\Phi (e^{\theta_{\bc}^{max}})^{x_1} \hat{G}_{0,*}(e^{\theta_{\bc}^{max}})^{x_2} \hat{\bv}^{U}(e^{\theta_{\bc}^{max}}) \hat{\bu}^{U}(e^{\theta_{\bc}^{max}}). 
\end{equation}
By the proof of Proposition \ref{pr:spr_hatU_eq1}, $\hat{\bv}^{U}(e^{\theta_{\bc}^{max}})$ is also the left eigenvector of $\hat{G}_{0,*}(e^{\theta_{\bc}^{max}})$, satisfying $\hat{G}_{0,*}(e^{\theta_{\bc}^{max}})\hat{\bv}^{U}(e^{\theta_{\bc}^{max}})=e^{2\eta_{\bc,2}^{max}} \hat{\bv}^{U}(e^{\theta_{\bc}^{max}})$. Since $e^{\theta_{\bc}^{max}}=e^{\eta_{\bc,1}^{max}+\eta_{\bc,2}^{max}}$, we have
\begin{equation}
\lim_{\tilde{\Delta}_{e^{\theta_{\bc}^{max}}}\ni z\to e^{\theta_{\bc}^{max}}} (e^{\theta_{\bc}^{max}}-z)^{\frac{1}{2}}\hat{\Phi}_{(x_1,x_2),*}(z) 
= \hat{g}^\Phi e^{\eta_{\bc,1}^{max} x_1+\eta_{\bc,2}^{max}(x_1+2x_2)} \hat{\bv}^{U}(e^{\theta_{\bc}^{max}}) \hat{\bu}^{U}(e^{\theta_{\bc}^{max}}). 
\label{eq:hatPhix1x2s_limit2}
\end{equation}
Denote $\hat{\bu}^{U}(e^{\theta_{\bc}^{max}})=\begin{pmatrix} \bu_1 & \bu_2 \end{pmatrix}$ and $\hat{\bv}^{U}(e^{\theta_{\bc}^{max}})=\begin{pmatrix} \bv_1 \cr \bv_2 \end{pmatrix}$, respectively. By \eqref{eq:hatPhixs_def} and \eqref{eq:hatPhix1x2s_limit2}, setting $\bx=(0,0)$, we obtain
\begin{equation}
\lim_{\tilde{\Delta}_{e^{\theta_{\bc}^{max}}}\ni z\to e^{\theta_{\bc}^{max}}} (e^{\theta_{\bc}^{max}}-z)^{\frac{1}{2}} 
\begin{pmatrix} 
\Phi^{\bc}_{(0,0),*}(z) &  \Phi^{\bc}_{(0,-1),*}(z) \cr 
\Phi^{\bc}_{(0,1),*}(z) &  \Phi^{\bc}_{(0,0),*}(z)
\end{pmatrix}
= \hat{g}^\Phi 
\begin{pmatrix} 
\bv_1 \bu_1 & \bv_1 \bu_2 \cr
\bv_2 \bu_1 & \bv_2 \bu_2
\end{pmatrix}.
\end{equation}
Hence, we have $\bv_1 \bu_1=\bv_2 \bu_2$, and this leads us to $\bv_2=c \bv_1$ and $\bu_2=c^{-1} \bu_1$, where $c=\bu_1\bv_2/(\bu_2\bv_2)$. 
Analogously, by \eqref{eq:hatPhixs_def} and \eqref{eq:hatPhix1x2s_limit2}, setting $\bx=(0,1)$, we obtain $e^{2\eta_{\bc,2}^{max}}\bv_1\bu_2=\bv_2\bu_1$, and this leads us to $c=e^{\eta_{\bc,2}^{max}}$. 
As a result, by \eqref{eq:hatPhixs_def} and \eqref{eq:hatPhix1x2s_limit2}, we have, for every $\bx=(x_1,x_2)\in\mathbb{Z}_+^2$,
\begin{align}
&\lim_{\tilde{\Delta}_{e^{\theta_{\bc}^{max}}}\ni z\to e^{\theta_{\bc}^{max}}} (e^{\theta_{\bc}^{max}}-z)^{\frac{1}{2}} \Phi^{\bc}_{(x_1,x_1+2x_2),*}(z)
= \hat{g}^\Phi e^{\eta_{\bc,1}^{max} x_1+\eta_{\bc,2}^{max}(x_1+2x_2)} \bv_1\bu_1, 
\label{eq:Phics_limit1} \\
&\lim_{\tilde{\Delta}_{e^{\theta_{\bc}^{max}}}\ni z\to e^{\theta_{\bc}^{max}}} (e^{\theta_{\bc}^{max}}-z)^{\frac{1}{2}} \Phi^{\bc}_{(x_1,x_1+2x_2+1),*}(z)
= \hat{g}^\Phi e^{\eta_{\bc,1}^{max} x_1+\eta_{\bc,2}^{max}(x_1+2x_2+1)} \bv_1\bu_1. 
\label{eq:Phics_limit2}
\end{align}
Since, in \eqref{eq:Phics_limit1} and \eqref{eq:Phics_limit2}, $x_1$ and $x_2$ are arbitrary, the same result holds for $\Phi^{\bc}_{(x_1',x_2'),*}(z)$, where $x_1'$ and $x_2'$ are nonnegative integers. Furthermore, we see that $\bu_{\bc}$ in \eqref{eq:hatPhix1x2s_limit} is given by $\bu_1$. 

We identify $\bv_1$. Define matrix functions:
\begin{align*}
&B_{-2}(z)=A_{1,-1} z,\quad
B_{-1}(z)=A_{0,-1} + A_{1,0} z,\quad 
B_0(z)=A_{-1,-1} z^{-1}+ A_{0,0} + A_{1,1} z,\cr
&B_1(z)=A_{-1,0} z^{-1}+ A_{0,1},\quad 
B_2(z)=A_{-1,1} z^{-1}, 
\end{align*}
and 
\[
B_*(z,w) = \sum_{j=-2}^2 B_j(z) w^j, 
\]
where we have the following  (see Section 3 of \cite{Ozawa20w}): 
\begin{equation}
B_*(e^{\theta_1},e^{\theta_2}) = A^{\{1,2\}}_{*,*}(e^{\theta_1-\theta_2},e^{\theta_2}).
\end{equation}
By Remark 2.4 of \cite{Ozawa21}, if $\hat\bv=\begin{pmatrix} \hat\bv_1 \cr \hat\bv_2 \end{pmatrix}$ is the right eigenvector of $\hat{A}^{\{1,2\}}_{*,*}(e^{\theta_1},e^{\theta_2})$ with respect to eigenvalue $\lambda$, then $\hat\bv_2=e^{\theta_2/2} \hat\bv_1$ and $\hat\bv_1$ is the right eigenvector of $B_*(e^{\theta_1},e^{\theta_2/2})$ with respect to the eigenvalue $\lambda$, and vice versa. 
By the proof of Proposition \ref{pr:spr_hatU_eq1}, we already know that $\hat{\bv}^{U}(e^{\theta_{\bc}^{max}})$ is the right eigenvector of $\hat{A}^{\{1,2\}}_{*,*}(e^{\theta_{\bc}^{max}},e^{2\eta_{\bc,2}^{max}})$ with respect to the eigenvalue $1$. 
Hence, we have
\begin{equation}
\bv_1 = B_*(e^{\theta_{\bc}^{max}},e^{\eta_{\bc,2}^{max}})\bv_1=A^{\{1,2\}}_{*,*}(e^{\eta_{\bc,1}^{max}},e^{\eta_{\bc,2}^{max}})\bv_1. 
\end{equation}
This implies $\bv_1=\bv^A(e^{\eta_{\bc,1}^{max}},e^{\eta_{\bc,2}^{max}})$. 
\end{proof}

The coefficient vector $\bvarphi^{\bc}_{1,-1}$ in the Puiseux series of $\bvarphi^{\bc}(z)$ is given as follows.
\begin{proof}[Proof of Proposition \ref{pr:varphic_limit1}]
Assume Type 1. 
If $\bar{\eta}_1'(\theta_2^*) < -c_1/c_2=-1 < 1/\bar{\eta}_2'(\theta_1^*)$, $\bvarphi_1(z)$ element-wise converges at $z=e^{\eta_{\bc,1}^{max}}<e^{\xi_{(1,0)}}$ and $\bvarphi_2(w)$ at $w=e^{\eta_{\bc,2}^{max}}<e^{\xi_{(0,1)}}$. Hence, applying \eqref{eq:hatPhix1x2s_limit} to expression \eqref{eq:varphic_eq1} of $\bvarphi^{\bc}(z)$, we obtain, for a positive row vector $\bu_{\bc}$, 
\begin{equation}
\bvarphi^{\bc}_{1,-1} 
= \lim_{\tilde{\Delta}_{e^{\theta_{\bc}^{max}}}\ni z\to e^{\theta_{\bc}^{max}}} (e^{\theta_{\bc}^{max}}-z)^{\frac{1}{2}} \bvarphi^{\bc}(z)
 =  \hat{g}^\Phi \bg(e^{\eta_{\bc,1}^{max}},e^{\eta_{\bc,2}^{max}}) \bv^A(e^{\eta_{\bc,1}^{max}},e^{\eta_{\bc,2}^{max}}) \bu_{\bc}, 
\label{eq:varphic1m1}
\end{equation}
where $\bu_{\bc}$ is given by $\bu_1$ in the proof of Proposition \ref{pr:Phics_limit}. The row vector $\bu_0^{\bc}$ in Proposition \ref{pr:varphic_limit1} is given by the right hand side of \eqref{eq:varphic1m1}, which is positive since $\bg(e^{\eta_{\bc,1}^{max}},e^{\eta_{\bc,2}^{max}}) \bv^A(e^{\eta_{\bc,1}^{max}},e^{\eta_{\bc,2}^{max}})>0$ by Proposition \ref{pr:gvA_positivity}. 
\end{proof}

%
%
%
\subsection{Proof of Corollary \ref{co:decay_function_homogeneous}} \label{sec:decay_function_homogeneous}

Define a probability distribution $\tilde{\bnu}=(\tilde{\bnu}_{\bx},\bx\in\mathbb{Z}^2)$ on $\mathbb{Z}^2\times S_0$ as 
\[
\tilde{\bnu}_{\bx} = \left\{ \begin{array}{ll} 
\bnu_{\bx}, & \bx\in\mathbb{Z}_+^2, \cr
\bzero^\top, & \mbox{otherwise}.
\end{array} \right.
\]
This $\tilde{\bnu}$ is the stationary distribution of a certain Markov chain on the state space $\mathbb{Z}^2\times S_0$, see Section 2.2 of \cite{Ozawa22}. For $\bx\in\mathbb{Z}^2$, the expression of $\tilde{\bnu}_{\bx}$ is given by (2.8) of \cite{Ozawa22}. 
For $\bc\in\mathbb{N}^2$ and $\bx,\bx'\in\mathbb{Z}^2$, define a vector generating function $\bar{\bvarphi}^{\bc}_{\bx}$ and matrix generating function $\Phi^{\bc}_{\bx',*+\bx}(z)$ as
\[
\bar{\bvarphi}^{\bc}_{\bx}(z) = \sum_{k=-\infty}^\infty z^k \tilde{\bnu}_{k\bc+\bx},\quad 
\Phi^{\bc}_{\bx',*+\bx}(z) = \sum_{k=-\infty}^\infty z^k \Phi^{\{1,2\}}_{\bx',k\bc+\bx}. 
\]
Since $\Phi^{\{1,2\}}_{\bx',k\bc+\bx}=\Phi^{\{1,2\}}_{\bx'-\bx,k\bc}$, we have $\Phi^{\bc}_{\bx',*+\bx}(z)=\Phi^{\bc}_{\bx'-\bx,*}(z)$. Hence, by (2.8) of \cite{Ozawa22}, we have
\begin{equation}
\bar{\bvarphi}^{\bc}_{\bx}(z) = \bar{\bvarphi}^{\bc}_{0,\bx}(z) + \bar{\bvarphi}^{\bc}_{1,\bx}(z) + \bar{\bvarphi}^{\bc}_{2,\bx}(z), 
\label{eq:varphicx_definition}
\end{equation}
where
\begin{align}
&\bar{\bvarphi}^{\bc}_{0,\bx}(z) = \sum_{i_1,i_2\in\{-1,0,1\}} \bnu_{(0,0)} (A^\emptyset_{i_1,i_2}-A^{\{1,2\}}_{i_1,i_2}) \Phi^{\bc}_{(i_1,i_2)-\bx,*}(z), 
\label{eq:varphicx0_definition}\\
&\bar{\bvarphi}^{\bc}_{1,\bx}(z) = \sum_{k=1}^\infty\ \sum_{i_1,i_2\in\{-1,0,1\}} \bnu_{(k,0)} (A^{\{1\}}_{i_1,i_2}-A^{\{1,2\}}_{i_1,i_2}) \Phi^{\bc}_{(k+i_1,i_2)-\bx,*}(z), \\
&\bar{\bvarphi}^{\bc}_{2,\bx}(z) = \sum_{k=1}^\infty\ \sum_{i_1,i_2\in\{-1,0,1\}} \bnu_{(0,k)} (A^{\{2\}}_{i_1,i_2}-A^{\{1,2\}}_{i_1,i_2}) \Phi^{\bc}_{(i_1,k+i_2)-\bx,*}(z). 
\label{eq:varphicx2_definition}
\end{align}
For $\bc\in\mathbb{N}^2$ and $\bx\in\mathbb{Z}_+^2$, define the vector generating function $\bvarphi^{\bc}_{\bx}(z)$ of vector sequence $\{\bnu_{k\bc+\bx}: k\in\mathbb{Z}_+\}$ as
\[
\bvarphi^{\bc}_{\bx}(z) =  \sum_{k=0}^\infty z^k \bnu_{k\bc+\bx}.
\]
The asymptotic decay function of the sequence $\{\bnu_{k\bc+\bx}: k\in\mathbb{Z}_+\}$ is obtained through the analytic property of the vector function $\bvarphi^{\bc}_{\bx}(z)$.

\begin{proof}[Proof of Corollary \ref{co:decay_function_homogeneous}]
Assume $\bc=(1,1)$. 
Let $\bx=(x_1,x_2)$ be a point in $\mathbb{Z}_+^2$ and set $k_0=\max\{x_1,x_2\}$, then $\bx-k_0\bc\le \bzero$. We have
\begin{equation}
\bar{\bvarphi}^{\bc}_{\bx-k_0\bc}(z) 
= \sum_{k=-\infty}^\infty z^k \tilde{\bnu}_{k\bc+\bx-k_0\bc}
= \sum_{k=-\min\{x_1,x_2\}}^{-1} z^{k+k_0} \bnu_{k\bc+\bx} + z^{k_0} \bvarphi^{\bc}_{\bx}(z).
\label{eq:varphicx_eq1}
\end{equation}
Since the first term and function $z^{k_0}$ on the right hand side of \eqref{eq:varphicx_eq1} are  analytic in $\mathbb{C}$, the analytic property of the vector function $\bvarphi^{\bc}_{\bx}(z)$ coincides with that of $\bar{\bvarphi}^{\bc}_{\bx-k_0\bc}(z)$. 
Furthermore, $\bar{\bvarphi}^{\bc}_{\bx-k_0\bc}(z)$ is given by \eqref{eq:varphicx_eq1} and we have 
\[
\hat\Phi^{\bc}_{(i_1,k+i_2)-(\bx-k_0\bc),*}(z) 
= (z^{i_1} \hat{G}_{0,*}(z)^{k+i_2}) (z^{k_0-x_1} \hat{G}_{0,*}(z)^{k_0-x_2}) \hat{\Phi}_{(0,0),*}(z), 
\]
where $k_0-x_1\ge 0$ and $k_0-x_2\ge 0$. 
Hence, in a manner similar to that used for the vector function $\bvarphi^{\bc}(z)$, we see that the vector function $\bar{\bvarphi}^{\bc}_{\bx-k_0\bc}(z)$ is analytic in $\Delta_{e^{\xi_{\bc}}}\cup\partial\Delta_{e^{\xi_{\bc}}}\setminus\{e^{\xi_{\bc}}\}$ and its singularity at the point $e^{\xi_{\bc}}$ coincides with that of $\bvarphi^{\bc}(z)$. 
For a general direction vector $\bc\in\mathbb{N}^2$, we can also see that the same results hold true for $\bar{\bvarphi}^{\bc}_{\bx-k_0\bc}(z)$ by using the block state process derived from the original 2d-QBD process; See Section 3.3 of \cite{Ozawa22}. 
This completes the proof. 
\end{proof}

%
%

\section{Concluding remark} \label{sec:conclusion}

Here we consider a topic with respect to an occupation measure. 
Recall that $P^{\{1,2\}}=(P^{\{1,2\}}_{\bx,\bx'}; \bx,\bx'\in\mathbb{Z}^2)$  is the transition probability matrix of the induced MA-process $\{\bY^{\{1,2\}}_n\}$ and $\Phi^{\{1,2\}}=(\Phi^{\{1,2\}}_{\bx,\bx'}; \bx,\bx'\in\mathbb{Z}^2)$ the fundamental matrix of $P^{\{1,2\}}$. Let $h_{\bc}^\Phi(k)$ be the asymptotic decay function of the matrix sequence $\{\Phi^{\{1,2\}}_{\bx,k\bc}; k\in\mathbb{N}\}$, i.e., for some positive matrix $C_{\bx}$, 
\begin{equation}
\lim_{k\to\infty} \Phi^{\{1,2\}}_{\bx,k\bc}/h_{\bc}^\Phi(k) = C_{\bx}. 
\end{equation}
By using the block state process derived from the original 2d-QBD process, we can see that Proposition \ref{co:Phics_analyticproperties} also holds for every direction vector $\bc$ in $\mathbb{N}^2$. Hence, we obtain 
\begin{equation}
h_{\bc}^\Phi(k) = k^{-\frac{1}{2}} e^{-\theta_{\bc}^{max} k}. 
\end{equation}
Furthermore, recall that $P^+$ is a partial matrix of $P^{\{1,2\}}$ given by restricting the state space of the additive part to the positive quadrant, i.e., $P^+=(P^{\{1,2\}}_{\bx,\bx'}; \bx,\bx'\in\mathbb{N}^2)$. $P^+$ is also a partial matrix of the transition probability matrix of the original 2d-QBD process, $P=(P_{\bx,\bx'}; \bx,\bx'\in\mathbb{Z}_+^2)$, i.e., $P^+=(P_{\bx,\bx'}; \bx,\bx'\in\mathbb{N}^2)$. 
Let $\tilde{Q}=(\tilde{Q}_{\bx,\bx'}; \bx,\bx'\in\mathbb{N}^2)$ be the fundamental matrix of $P^+$, i.e., $\tilde{Q}=\sum_{n=0}^\infty (P^+)^n$. For $j,j'\in S_0$, denote by $\tilde{q}_{(\bx,j),(\bx',j')}$ the $(j,j')$-entry of $\tilde{Q}_{\bx,\bx'}$. The entries of $\tilde{Q}$ are called an occupation measure in \cite{Ozawa21}. 
By Theorem 5.1 of \cite{Ozawa21}, the asymptotic decay rate of the matrix sequence $\{\tilde{Q}_{\bx,k \bc}; k\in\mathbb{N}\}$ is given by $e^{\theta_{\bc}^{max}}$, i.e., 
\begin{equation}
-\lim_{k\to\infty} \frac{1}{k} \log \tilde{q}_{(\bx,j),(k\bc,j')} = \theta_{\bc}^{max}, 
\end{equation}
which coincides with that of the matrix sequence $\{\Phi^{\{1,2\}}_{\bx,k\bc}; k\in\mathbb{N}\}$. 
One question, therefore, arises: Does the asymptotic decay function of the matrix sequence $\{\tilde{Q}_{\bx,k \bc}; k\in\mathbb{N}\}$ coincide with that of the matrix sequence $\{\Phi^{\{1,2\}}_{\bx,k\bc}; k\in\mathbb{N}\}$? The answer to the question seems not to be so obvious.

%
%

%
%
\appendix

%
%
\section{Proof of Theorem \ref{th:Jordan_decomposition_G_pre}} \label{sec:Jordan_decomposition_G_app} 

First, we give the generalized eigenvectors of $G(z)$ for $z\in\Delta_{e^{\theta_1^{min}},e^{\theta_1^{max}}}\setminus\calE_1$, then analytically extend them to $z\in\mathbb{C}\setminus\calE_1$. Set $\Omega=\Delta_{e^{\theta_1^{min}},e^{\theta_1^{max}}}\setminus(\calE_1\cup\calE_2)$. 

For each $k\in\{1,2,...,m_0\}$ and for each $z\in\Omega\setminus\bigcup_{k=1}^{m_0} \calE^G_k$, since the Jordan normal form of $G(z)$ is given by \eqref{eq:Jordan_normal_Gz}, there exist linearly independent vectors called the generalized eigenvectors of $G(z)$ with respect to the eigenvalue $\check{\alpha}_k(z)$, $\check{\bv}_{k,i,j}(z),\,i=1,2,...,m_{k,0},\,j=1,2,...,m_{k,i},$ satisfying
\begin{equation}
(\check{\alpha}_k(z) I-G(z)) \check{\bv}_{k,i,j}(z) = \check{\bv}_{k,i,j+1}(z), 
\end{equation}
where $\check{\bv}_{k,i,m_{k,i}+1}(z)=\bzero$. For each $i$, $\check{\bv}_{k,i,j}(z),\,j=1,2,...,m_{k,i},$ are called a Jordan sequence of the generalized eigenvectors. 
Using the Jordan sequences, we define $l_{\check{q}(k)}\times 1$ block vectors, $\bv_{k,i,j}(z),\,i=1,2,...,m_{k,0},\,j=1,2,...,m_{k,i},$ as 
\[
\bv_{k,i,j}(z) = \vec\!\begin{pmatrix} \check{\bv}_{k,i,j}(z)  & \check{\bv}_{k,i,j+1}(z) & \cdots & \check{\bv}_{k,i,m_{k,i}}(z) & \bzero & \cdots & \bzero \end{pmatrix}, 
\]
where, for a matrix $A=\begin{pmatrix} \ba_1 & \ba_2 & \cdots & \ba_n \end{pmatrix}$, $\vec(A)$ is the column vector given by 
\[
\vec(A)=\begin{pmatrix} \ba_1 \cr \ba_2 \cr \vdots \cr \ba_n \end{pmatrix}.
\]
We also define a vector space $\mathbb{V}^G_k(z)$ as 
\[
\mathbb{V}^G_k(z) = \mathrm{span}\,\{\bv_{k,i,j}(z): i=1,2,...,m_{k,0},\,j=1,2,...,m_{k,i} \}.
\]
Note that the generalized eigenvectors $\check{\bv}_{k,i,j}(z)$ are not unique but $\mathbb{V}^G_k(z)$ is. Since the generalized eigenvectors are linearly independent, $\bv_{k,i,j}(z)$ are also linearly independent and we have 
\[
\dim \mathbb{V}^G_k(z) = \sum_{i=1}^{m_{k,0}} m_{k,i} =  l_{\check{q}(k)}. 
\]
%
For $k\in\{1,2,...,m_0\}$, define an $l_{\check{q}(k)}\times l_{\check{q}(k)}$ block matrix function $\Lambda^G_k(z)$ as
{\small \[
\Lambda^G_k(z) 
= \begin{pmatrix}
\check{\alpha}_k(z) I -G(z) & -I &   \cr
& \check{\alpha}_k(z) I -G(z) & -I &   \cr
& \qquad\qquad\quad\ddots & \qquad\qquad\quad\ddots & \cr 
& & \check{\alpha}_k(z) I -G(z) & -I \cr
& & & \check{\alpha}_k(z) I -G(z)
\end{pmatrix}.
\] }
We give the following proposition.
\begin{proposition} \label{pr:Ker_LambdaGk}
For each $k\in\{1,2,...,m_0\}$ and for each $z\in\Omega\setminus\bigcup_{k=1}^{m_0} \calE^G_k$, 
\begin{equation}
\mathrm{Ker}\, \Lambda^G_k(z) = \mathbb{V}^G_k(z).
\label{eq:Ker_LambdaGk}
\end{equation}
\end{proposition}

\begin{proof}
Assume $\bv\in\mathbb{V}^G_k(z)$. Then, by the definition of $\mathbb{V}^G_k(z)$, we have $\Lambda^G_k(z) \bv = \bzero$ and $\bv\in \mathrm{Ker}\, \Lambda^G_k(z) $. 
For $\bv=\vec\begin{pmatrix} \bv_1 & \bv_2 & \cdots & \bv_{l_{\check{q}(k)}} \end{pmatrix}$, assume $\Lambda^G_k(z) \bv = \bzero$. If there exists an index $i$ such that $\bv_i=\bzero$, then by the assumption, for  every $j$ such that $i\le j\le l_{\check{q}(k)}$, we have $\bv_j=\bzero$, and this implies $\bv\in\mathbb{V}^G_k(z)$.
\end{proof}

By Theorem S6.1 of \cite{Gohberg09}, since the matrix function $\Lambda^G_k(z)$ is entry-wise analytic in $\Delta_{e^{\theta_1^{min}},e^{\theta_1^{max}}}\setminus\calE_1$, there exist $l_{\check{q}(k)}$ vector functions $\bv^G_{k,i}(z)$, $i=1,2,...,l_{\check{q}(k)}$, that are element-wise analytic and linearly independent in $\Delta_{e^{\theta_1^{min}},e^{\theta_1^{max}}}\setminus\calE_1$ and satisfy
\[
\Lambda^G_k(z) \bv^G_{k,i}(z) = \bzero,\ i=1,2,...,l_{\check{q}(k)}.
\]
Hence, for each $z\in\Omega\setminus\bigcup_{k=1}^{m_0} \calE^G_k$, $\bv^G_{k,i}(z)\in\mathbb{V}^G_k(z)$. 
We select vectors composing the Jordan sequences with respect to the eigenvalue $\check\alpha_k(z)$ from $\{\bv^G_{k,i}(z),\, i=1,2,...,l_{\check{q}(k)}\}$. Represent each $\bv^G_{k,i}(z)$ in block form as 
\[
\bv^G_{k,i}(z) = \vec\!\begin{pmatrix} \bv^G_{k,i,1}(z)  & \bv^G_{k,i,2}(z) & \cdots & \bv^G_{k,i,l_{\check{q}(k)}}(z) \end{pmatrix}. 
\]
From the proof of Proposition \ref{pr:Ker_LambdaGk}, we see that, for every $i\in\{1,2,...,l_{\check{q}(k)}\}$, there exists a positive integer $\mu_{k,i}$ such that $\bv^G_{k,i,j}(z)\ne\bzero$ for every $j\in\{1,2,...,\mu_{k,i}\}$ and  $\bv^G_{k,i,j}(z)=\bzero$ for every $j\in\{\mu_{k,i}+1,\mu_{k,i}+2,...,l_{\check{q}(k)}\}$. 
Renumber the elements of $\{\bv^G_{k,i}(z)\}$ so that if $i\le i'$, then $\mu_{k,i}\ge \mu_{k,i'}$. 
Define a set of vector functions, $\check{\mathbb{V}}_k$, according to the following procedure.
\begin{itemize}
\item[(S1)] Set $\check{\mathbb{V}}_k=\emptyset$ and $i=1$. 
\item[(S2)] If $\bv^G_{k,i,\mu_{k,i}}(z)$ is linearly independent of $\{\bv^G_{k,i',\mu_{k,i'}}(z): \bv^G_{k,i'}(z)\in \check{\mathbb{V}}_k\}$, append $\bv^G_{k,i}(z)$ to $\check{\mathbb{V}}_k$. 
\item[(S3)] If $i=l_{\check{q}(k)}$, stop the procedure; otherwise add $1$ to $i$ and go to (S2).
\end{itemize}

\begin{proposition} \label{pr:equality_mk0}
For $k\in\{1,2,...,m_0\}$, the number of elements of $\check{\mathbb{V}}_k$ is $m_{k,0}$. 
\end{proposition}

\begin{proof}
Since, for every $i\in\{1,2,...,l_{\check{q}(k)}\}$, $(\check{\alpha}_k(z) I -G(z)) \bv^G_{k,i,\mu_{k,i}} = \bzero$ and $\mathrm{dim}\,\mathrm{Ker}\, (\check{\alpha}_k(z) I -G(z)) = m_{k,0}$, the number of elements of $\check{\mathbb{V}}_k$ is less than or equal to $m_{k,0}$. If it is strictly less than $m_{k,0}$, we have 
\[
\mathrm{dim}\,\mathrm{Ker}\, \Lambda^G_k(z) 
=\mathrm{dim}\,\mathrm{span}\,\{\bv^G_{k,i}(z),\, i=1,2,...,l_{\check{q}(k)}\} 
< \mathrm{dim}\,\mathbb{V}^G_k(z). 
\]
This contradicts \eqref{eq:Ker_LambdaGk}, and we see that the number of elements of $\check{\mathbb{V}}_k$ is just $m_{k,0}$. 
\end{proof}

Denote by $\check{\bv}^G_{k,1}(z), \check{\bv}^G_{k,2}(z), ..., \check{\bv}^G_{k,m_{k,0}}(z)$ the elements of $\check{\mathbb{V}}_k$. For each $\check{\bv}^G_{k,i}(z)$, define $\check{\mu}_{k,i}$ in a manner similar to that used for defining $\mu_{k,i}$. We assume $\check{\bv}^G_{k,i}(z),\,i=1,2,...,m_{k,0},$ are numbered so that if $i\le i'$, then $\check{\mu}_{k,i}\ge \check{\mu}_{k,i'}$.

\begin{proposition} \label{pr:equality_muki_mki}
For $k\in\{1,2,...,m_0\}$ and for $i\in\{1,2,...,m_{k,0}\}$, $\check{\mu}_{k,i}=m_{k,i}$
\end{proposition}

\begin{proof}
For each $i\in\{1,2,...,m_{k,0}\}$, $\{\check{\bv}^G_{k,i,1}(z), \check{\bv}^G_{k,i,2}(z), ..., \check{\bv}^G_{k,i,\check{\mu}_{k,i}}(z)\}$ is a Jordan sequence of the generalized eigenvectors of $G(z)$ with respect to the eigenvalue $\check{\alpha}_k(z)$. Hence, considering the procedure defining $\check{\bv}^G_{k,i}(z)$, we see that, for every $i\in\{1,2,...,m_{k,0}\}$, $\check{\mu}_{k,i}\le m_{k,i}$. 
Suppose there exists some $i_0\in\{1,2,...,m_{k,0}\}$ such that $\check{\mu}_{k,i}=m_{k,i}$ for every $i\in\{1,2,...,i_0-1\}$ and $\check{\mu}_{k,i_0}<m_{k,i_0}$. Then, there exists a vector $\bv=\vec\begin{pmatrix} \bv_1 & \bv_2 & \cdots & \bv_{m_{k,i_0}} & \bzero & \cdots & \bzero \end{pmatrix}$ in $\mathbb{V}^G_k(z)$ such that $\bv_i\ne\bzero$ for every $i\in\{1,2,...,m_{k,i_0}\}$ and $\bv$ is linearly independent of $\{\bv^G_{k,i}(z),\, i=1,2,...,l_{\check{q}(k)}\}$. 
By the same reason as that used in the proof of Proposition \ref{pr:equality_mk0}, this contradicts  \eqref{eq:Ker_LambdaGk} and, for every $i\in\{1,2,...,m_{k,0}\}$, $\check{\mu}_{k,i}$ must be $m_{k,i}$.
\end{proof}

From this proposition, we see that, for $z\in\Omega\setminus\bigcup_{k=1}^{m_0} \calE^G_k$, $\{\check{\bv}^G_{k,i,j}(z): k=1,2,...,m_0, i=1,2,...,m_{k,0}, j=1,2,..., m_{k,i}\}$ is the set of generalized eigenvectors corresponding to the Jordan normal form \eqref{eq:Jordan_normal_Gz}. Define a matrix function $T^G(z)$ as
\[
T^G(z) = 
\begin{pmatrix}
\check{\bv}^G_{k,i,j}(z),\, k=1,2,...,m_0,\, i=1,2,...,m_{k,0},\, j=1,2,..., m_{k,i}
\end{pmatrix},
\]
which is entry-wise analytic in $\Delta_{e^{\theta_1^{min}},e^{\theta_1^{max}}}\setminus\calE_1$. Define a point set $\calE_T^G$ as
\[
\calE_T^G = \{z\in\Delta_{e^{\theta_1^{min}},e^{\theta_1^{max}}}\setminus\calE_1: \det\,T^G(z) = 0\}, 
\]
which is an empty set or a set of discrete complex numbers. Then, for $z\in\Omega\setminus(\bigcup_{k=1}^{m_0} \calE^G_k\cup\calE_T^G)$, we obtain the Jordan decomposition of $G(z)$ as
\begin{equation}
G(z) = T^G(z) J^G(z) (T^G(z))^{-1}. 
\label{eq:Gz_Jordan_dec}
\end{equation}
Since $G(z)$ is entry-wise analytic in $\Delta_{e^{\theta_1^{min}},e^{\theta_1^{max}}}$, we see by the identity theorem for analytic functions that the right hand side of \eqref{eq:Gz_Jordan_dec} is also entry-wise analytic in the same domain. 

%
Next, we analytically extend $\check{\bv}^G_{k,i,j}(z),\, k=1,2,...,m_0,\, i=1,2,...,m_{k,0},\, j=1,2,..., m_{k,i}$. 
Define matrix functions $F_1(z,w)$ and $F_2(z)$ as 
\[
F_1(z,w) = z (I-A_{*,0}(z)-2w A_{*,1}(z)), \quad
F_2(z) = z A_{*,1}(z), 
\]
where $F_1(z,w)$ is entry-wise analytic on $\mathbb{C}^2$ and $F_2(z)$ on $\mathbb{C}$. By \eqref{eq:Ass_Gz_relation}, we have 
\begin{equation}
L(z,w) = F_1(z,w) (w I-G(z)) + F_2(z) (w I-G(z))^2. 
\label{eq:Lzw_Gz_relation}
\end{equation}
For $k\in\{1,2,...,m_0\}$, define an $l_{\check{q}(k)}\times l_{\check{q}(k)}$ block matrix function $\Lambda^L_k(z)$ as
{\small \[
\Lambda^L_k(z) 
= \begin{pmatrix}
L(z,\check{\alpha}_k(z)) & -F_1(z,\check{\alpha}_k(z)) &  -F_2(z)  \cr
& L(z,\check{\alpha}_k(z)) & -F_1(z,\check{\alpha}_k(z)) &  -F_2(z)  \cr
& \qquad\qquad\quad\ddots & \qquad\qquad\quad\ddots & \qquad\qquad\quad\ddots \cr 
& & L(z,\check{\alpha}_k(z)) & -F_1(z,\check{\alpha}_k(z)) & -F_2(z)  \cr
& & & L(z,\check{\alpha}_k(z)) & -F_1(z,\check{\alpha}_k(z)) \cr
& & & & L(z,\check{\alpha}_k(z))
\end{pmatrix}, 
\] }
which is entry-wise analytic in $\mathbb{C}\setminus\calE_1$.

\begin{proposition} \label{pr:Ker_LambdaLG_equality}
For every $k\in\{1,2,...,m_0\}$ and for every $z\in\\Delta_{e^{\theta_1^{min}},e^{\theta_1^{max}}}$, 
\begin{equation}
\mathrm{Ker}\, \Lambda^L_k(z) = \mathrm{Ker}\, \Lambda^G_k(z). 
\end{equation}
\end{proposition}

Before proving this proposition, we give another one.
\begin{proposition} \label{pr:FA_invettible}
For every $k\in\{1,2,...,s_0\}$ and $z\in\Delta_{e^{\theta_1^{min}},e^{\theta_1^{max}}}$, 
\[
F_1(z,\alpha_k(z))+F_2(z)(\alpha_k(z)I-G(z)) = z \left( I-A_{*,0}(z)-\alpha_k(z) A_{*,1}(z)+A_{*,1}(z) G(z) \right) 
\]
is regular (invertible).
\end{proposition}
\begin{proof}
Let $R(z)$ be the rate matrix function generated from $\{A_{i,j}: i,j=-1,0,1\}$; for the definition of $R(z)$, see Section 4.1 of \cite{Ozawa18}. By Lemma 4.3 of \cite{Ozawa18}, nonzero eigenvalues of $R(z)$ are given by $\alpha_k(z)^{-1},\,k=s_0+1,s_0+2,...,m_\phi$. Since, for every $k\in\{1,2,...,s_0\}$, $k'\in\{s_0+1,s_0+2,...,m_\phi\}$ and $z\in\Delta_{e^{\theta_1^{min}},e^{\theta_1^{max}}}$, $|\alpha_k(z)|\le\alpha_{s_0}(|z|)<|\alpha_{k'}(z)|$, $I-\alpha_k(z) R(z)$ is regular. 
Define a matrix function $H(z)$ as $H(z)=A_{*,0}(z)+A_{*,1}(z)G(z)$, then by Corollary 4.1 of \cite{Ozawa18}, $I-H(z)$ is regular in $\Delta_{e^{\theta_1^{min}},e^{\theta_1^{max}}}$. 
By Lemma 4.1 of \cite{Ozawa18}, we have 
\begin{align}
I-A_{*,0}(z)-\alpha_k(z) A_{*,1}(z) - A_{*,1}(z) G(z) 
&= (I-\alpha_k(z) R(z))(I-H(z)),
\end{align}
where the right hand side of the equation is regular. This completes the proof. 
\end{proof}

%
\begin{proof}[Proof of Proposition \ref{pr:Ker_LambdaLG_equality}]
Assume a vector $\bv=\vec\begin{pmatrix} \bv_1 & \bv_2 & \cdots & \bv_{l_{\check{q}(k)}} \end{pmatrix}$ satisfies $\Lambda^L_k(z)\bv=0$. Then, we have, for $i\in\{1,2,...,l_{\check{q}(k)}\}$, 
\begin{equation}
L(z,\check{\alpha}_k(z)) \bv_i =  F_1(z,\check{\alpha}_k(z)) \bv_{i+1} + F_2(z) \bv_{i+2}, 
\end{equation}
where $\bv_{l_{\check{q}(k)}+1}=\bv_{l_{\check{q}(k)}+2}=\bzero$. We prove  by induction that this $\bv$ satisfies, for every $i\in\{1,2,...,l_{\check{q}(k)}\}$, $(\check{\alpha}_k(z) I-G(z)) \bv_i = \bv_{i+1}$. 
Let $i_0$ be the maximum integer less than or equal to $l_{\check{q}(k)}$ that satisfies, for every $i\in\{i_0+1,i_0+2,...,l_{\check{q}(k)}\}, \bv_i=\bzero$. Then, we have $L(z,\check{\alpha}_k(z)) \bv_{i_0}=\bzero$. By \eqref{eq:Lzw_Gz_relation}, we have 
\begin{equation}L(z,\check{\alpha}_k(z)) = ( F_1(z,\check{\alpha}_k(z))+F_2(z)(\check{\alpha}_k(z)I-G(z)) ) (\check{\alpha}_k(z)I-G(z)). 
\label{eq:L_G_relation}
\end{equation}
Hence, by Proposition \ref{pr:FA_invettible}, we obtain $(\check{\alpha}_k(z)I-G(z))\bv_{i_0}=\bzero=\bv_{i_0+1}$. 
Assume the assumption of induction holds for a positive integer $i$ less than or equal to $i_0$. Then, 
\begin{align}
L(z,\check{\alpha}_k(z)) \bv_{i-1}
&= F_1(z,\check{\alpha}_k(z))\bv_i+F_2(z)\bv_{i+1} \cr
&= (F_1(z,\check{\alpha}_k(z))+F_2(z)(\check{\alpha}_k(z)I-G(z))) \bv_i, 
\label{eq:L_F12_relation}
\end{align}
and by \eqref{eq:L_G_relation}, \eqref{eq:L_F12_relation} and Proposition \ref{pr:FA_invettible}, we obtain $(\check{\alpha}_k(z)I-G(z))\bv_{i-1}=\bv_{i}$. 
Hence, $\bv$ satisfies, for every $i\in\{1,2,...,l_{\check{q}(k)}\}$, $(\check{\alpha}_k(z) I-G(z)) \bv_i = \bv_{i+1}$, and this leads us to $\Lambda^G_k(z) \bv=\bzero$. 

%
Next, assume a vector $\bv=\vec\begin{pmatrix} \bv_1 & \bv_2 & \cdots & \bv_{l_{\check{q}(k)}} \end{pmatrix}$ satisfies $\Lambda^G_k(z)\bv=0$. Then, we have, for $i\in\{1,2,...,l_{\check{q}(k)}\}$, $(\check{\alpha}_k(z) I-G(z)) \bv_i = \bv_{i+1}$, where $\bv_{l_{\check{q}(k)}+1}=\bzero$. By \eqref{eq:L_G_relation}, this $\bv$ satisfies, for every $i\in\{1,2,...,l_{\check{q}(k)}\}$, 
\begin{align}
L(z,\check{\alpha}_k(z)) \bv_i
&= F_1(z,\check{\alpha}_k(z)) \bv_{i+1} +F_2(z)(\check{\alpha}_k(z)I-G(z)) \bv_{i+1} \cr
&= F_1(z,\check{\alpha}_k(z))\bv_{i+1}+F_2(z)\bv_{i+2}, 
\label{eq:L_F12_relation2}
\end{align}
and this implies $\Lambda^L_k(z)\bv=0$. 
\end{proof}

%
\begin{proof}[Proof of Theorem \ref{th:Jordan_decomposition_G_pre}]
Let $k$ be an arbitrary integer in $\{1,2,...,m_0\}$. 
By Propositions \ref{pr:Ker_LambdaGk} and \ref{pr:Ker_LambdaLG_equality}, we have 
\[
\mathrm{dim}\,\mathrm{Ker}\, \Lambda^L_k(z)=l_{\check{q}(k)}, 
\]
except for some discrete points in $\mathbb{C}$. 
Hence, by Theorem S6.1 of \cite{Gohberg09}, since the matrix function $\Lambda^L_k(z)$ is entry-wise analytic in $\mathbb{C}\setminus\calE_1$, there exist $l_{\check{q}(k)}$ vector functions $\bv^L_{k,i}(z)$, $i=1,2,...,l_{\check{q}(k)}$, that are element-wise analytic and linearly independent in $\mathbb{C}\setminus\calE_1$ and satisfy
\[
\Lambda^L_k(z) \bv^L_{k,i}(z) = \bzero,\,i=1,2,...,l_{\check{q}(k)}.
\]
By Proposition \ref{pr:Ker_LambdaLG_equality}, for each $i$, $\bv^L_{k,i}(z)$ also satisfies $\Lambda^G_k(z) \bv^L_{k,i}(z) = \bzero$ for every $z\in\Delta_{e^{\theta_1^{min}},e^{\theta_1^{max}}}\setminus\calE_1$. Hence, by the identity theorem, we see that $\bv^L_{k,i}(z)$ is an analytic extension of $\bv^G_{k,i}(z)$. 
According to the same procedure as that used for selecting $\{\check{\bv}^G_{k,i}(z),\,i=1,2,...,m_{k,0}\}$ from $\{\bv^G_{k,i}(z),\,i=1,2,...,l_{\check{q}(k)}\}$, we select $m_{k,0}$ vectors from $\{\bv^L_{k,i}(z),\,i=1,2,...,l_{\check{q}(k)}\}$ and denote them by $\{\check{\bv}^L_{k,i}(z),\,i=1,2,...,m_{k,0}\}$.
For each $i$, $\check{\bv}^L_{k,i}(z)$ is represented in block form as 
\[
\check{\bv}^L_{k,i}(z) = \vec\!\begin{pmatrix} \check{\bv}^L_{k,i,1}(z)  & \check{\bv}^L_{k,i,2}(z) & \cdots & \check{\bv}^L_{k,i,m_{k,i}}(z) & \bzero & \cdots & \bzero \end{pmatrix}. 
\]
Define a matrix function $T^L(z)$ as
\[
T^L(z) = 
\begin{pmatrix}
\check{\bv}_{k,i,j}^L(z),\, k=1,2,...,m_0,\, i=1,2,...,m_{k,0},\, j=1,2,..., m_{k,i}
\end{pmatrix},
\]
which is entry-wise analytic in $\mathbb{C}\setminus\calE_1$. Since each $\check{\bv}_{k,i,j}^L(z)$ is an analytic extension of $\check{\bv}_{k,i,j}^G(z)$, we have, for $z\in\Omega\setminus(\bigcup_{k=1}^{m_0} \calE^G_k\cup\calE_T^G)$, 
\[
G(z) = T^L(z) J^G(z) (T^L(z))^{-1}, 
\]
which is \eqref{eq:Gz_analytic_extension}. 
Set $\calE_0=\calE_2\cup(\bigcup_{k=1}^{m_0} \calE^G_k)\cup\calE_T^G$, then $\calE_0$ is a set of discrete complex numbers and we have $\Omega\setminus(\bigcup_{k=1}^{m_0} \calE^G_k\cup\calE_T^G) = \Delta_{e^{\theta_1^{min}},e^{\theta_1^{max}}}\setminus(\calE_1\cup\calE_0)$.
This completes the proof of Theorem \ref{th:Jordan_decomposition_G_pre}. 
\end{proof}

\end{document}